\documentclass[11pt]{article}

\usepackage[utf8]{inputenc}

\usepackage{scrpage2}
\usepackage{amsmath}
\usepackage{amsfonts,amssymb,amsthm}
\usepackage{mathtools}
\usepackage{enumerate}
\usepackage{nicefrac}
\usepackage{graphicx}
\usepackage{esint}
\usepackage{caption}
\usepackage{stmaryrd}
\usepackage[labelformat = brace, position=top]{subcaption}
\usepackage{tikz}
\usetikzlibrary{calc,intersections,through,backgrounds,patterns,arrows.meta}

\usepackage[top=1.0in, bottom=1.0in, left=1.0in,
right=1.0in]{geometry}

\usepackage[normalem]{ulem}
\usepackage{cite}

\numberwithin{equation}{section} 

\newcounter{statement}
\numberwithin{statement}{section}

\newtheorem{Def}[statement]{Definition}

\newtheorem{thm}[statement]{Theorem}
\newtheorem{lemma}[statement]{Lemma}
\newtheorem{cor}[statement]{Corollary}
\newtheorem{prop}[statement]{Proposition}

\theoremstyle{remark}
\newtheorem{remark}[statement]{Remark}

\newcommand{\R}{\mathbb R} 				
\newcommand{\N}{\mathbb N}
\newcommand{\Z}{\mathbb Z}
\newcommand{\ball}[2]{{B_{#2}\left(#1\right)}}				
\newcommand{\intd}{\, \mathrm{d}} 				
\newcommand{\Leb}{\mathcal L}				
\newcommand{\Sph}{{\mathbb S}}				
\newcommand{\Hd}{\mathcal H}				

\newcommand{\osc}{\operatorname{osc}}

\newcommand{\stcomp}[1]{{#1}^{\mathsf{c}}}		

\newcommand{\eps}{\varepsilon}

\newcommand{\interior}[1]{{\kern0pt#1}^{\mathrm{o}}	
}
\newcommand{\arcsinh}{\operatorname{arsinh}}

\def\XXint#1#2#3{{\setbox0=\hbox{$#1{#2#3}{\int}$ }
\vcenter{\hbox{$#2#3$ }}\kern-.58\wd0}}

\title{A nonlocal isoperimetric problem with dipolar repulsion}
\author{Cyrill B.\ Muratov\footnote{Department of Mathematical
    Sciences, New Jersey Institute of Technology, Newark, New Jersey
    07102, USA.  Please use {muratov@njit.edu} for correspondence.
} \and Thilo M. Simon\footnotemark[1]}

\begin{document}
\maketitle
\begin{abstract}
  We study a geometric variational problem for sets in the plane in
  which the perimeter and a regularized dipolar interaction compete
  under a mass constraint.  In contrast to previously studied nonlocal
  isoperimetric problems, here the nonlocal term asymptotically
  localizes and contributes to the perimeter term to leading order.
  We establish existence of generalized minimizers for all values of
  the dipolar strength, mass and regularization cutoff and give
  conditions for existence of classical minimizers.  For subcritical
  dipolar strengths we prove that the limiting functional is a
  renormalized perimeter and that for small cutoff lengths all
  mass-constrained minimizers are disks.  For critical dipolar
  strength, we identify the next-order $\Gamma$-limit when sending the
  cutoff length to zero and prove that with a slight modification of
  the dipolar kernel there exist masses for which classical minimizers
  are not disks.
\end{abstract}

\section{Introduction}

Understanding the emergence of spatial order from basic constitutive
interactions is one of the most important problems in the natural
sciences. For example, why do the ions of Na$^+$ and Cl$^-$ organize
themselves into an alternating pattern arranged into a cubic crystal
when precipitating from a supersaturated aqueous solution, something
that can be readily observed in a simple tabletop experiment? Our
present physical understanding is that this process is driven by the
competition of repulsive electrostatic interactions between the like
ions, attractive interactions between the opposite ions, and a
hard-core repulsion at short distances, to minimize the total
interaction energy (both quantum mechanical and thermal effects are
also present, but are believed to be of secondary importance). In
these terms, the fundamental problem was concisely articulated in 1967
by Uhlenbeck \cite[p. 581]{uhlenbeck1967summarizing}: {\em ``The basic
  difficulty lies perhaps in the fact that one does not really
  understand the existence of regular solids from the molecular
  forces.''} For ionic crystals in the much simpler periodic setting,
the question goes back much further \cite{born21} and was resolved
only very recently \cite{betermin18}.

The above question becomes even more complicated for problems
involving many-body effects. Perhaps the best known example is that of
the spatial arrangement of protons and neutrons in nuclear matter,
which is relevant both to the shape of ordinary atomic nuclei
\cite{heisenbergA2kerne} and the exotic phases of matter in the crust
of neutron stars \cite{pethick95}. The same problem is also ubiquitous
in various hard and soft condensed matter systems in which mesoscopic
phases form as a result of competing attractive and repulsive
interactions operating on different scales
\cite{seul95,andelman09,m:pre02,muthukumar97,hubert}. The earliest
model that captures the competition of short-range attractive forces
and long-range repulsive forces in the case of the atomic nuclei was
conceived in 1929 by Gamow \cite{gamow30} and further refined by
Heisenberg \cite{heisenbergA2kerne} and von Weizs\"acker
\cite{weizsacker35}. It is now known as the {\em liquid drop model} of
the atomic nucleus. In this model, the nucleus is treated as a drop of
incompressible liquid held together by surface tension and subject to
Coulombic repulsion by the uniformly distributed positive charge of
the protons. For dense nuclear matter, the same model is considered
within a large periodic box and is known to produce multiple
morphologies referred to as nuclear ``pasta'' phases \cite{pethick95}.

Mathematically, the liquid drop model belongs to a class of geometric
variational problems in which one minimizes energies of the form (for
a non-technical overview, see \cite{cmt:nams17})
\begin{align}
  \label{Egam}
  E(\Omega) := P(\Omega) + \frac12 \int_\Omega \int_\Omega G(x - y) \,
  \intd x  \intd y,
\end{align}
among measurable sets $\Omega \subset \mathbb R^n$ subject to a mass
constraint $|\Omega| = m$ for $m>0$. Here $|\Omega|$ denotes the
$n$-dimensional Lebesgue measure of the set $\Omega$, $P(\Omega)$ is
the perimeter of $\Omega$, which is a suitable generalization of the
surface measure (for precise definitions, see below), and $G(x)$ is
some ``repulsive'' kernel.  In the case of the liquid drop model
in the whole space, one chooses the Newtonian potential
$G(x) := \frac{1}{4 \pi |x|}$ and minimizes over all
$\Omega \subset \mathbb R^3$. In a periodic setting, one should
instead consider $\Omega$ to be a subset of a large three-dimensional
torus, and $G$ should be the periodic Green's function of the
Laplacian (with uniform neutralizing background)
\cite{kmn:cmp16}. Note that with different spatial dimensionalities
and different choices of repulsive kernels the model is also relevant
to a number of other physical situations
\cite{m:pre02,m:cmp10,km:cpam13,km:cpam14}. In fact, the case of
dipolar repulsion considered in this paper falls within the above
framework as well.

In the mathematical literature, nonlocal isoperimetric problems in
which perimeter competes with a nonlocal repulsive term seem to have
been largely unnoticed until quite recently, with a notable exception
of a paper by Otto on the dynamics of labyrinthine pattern formation
in ferrofluids \cite{otto98}, and a paper by Rigot dealing mostly with
regularity of minimizers of \eqref{Egam} \cite{rigot00}. Gamow's
liquid drop model caught the attention of mathematicians after
reappearing as the leading order asymptotic problem in the studies of
the Ohta-Kawasaki energy by Choksi and Peletier
\cite{choksi10,choksi11}.  Since then the problem has enjoyed a
considerable attention. In the following, we review some of the
results obtained so far (naturally, the list of references below is
not meant to be comprehensive).

The first study of the problem associated with \eqref{Egam}, in which
$\Omega \subset \mathbb R^n$ and $G(x) = \frac{1}{|x|^\alpha}$ is a
Riesz kernel with $\alpha \in (0,n)$ was carried out by Kn\"upfer and
Muratov \cite{km:cpam13, km:cpam14}. This setting includes the
classical Gamow's model, for which $n = 3$ and $\alpha = 1$.  For a
range of parameters covering the latter, their results establish
existence and radial symmetry of minimizers for sufficiently small
masses, and non-existence for sufficiently large masses. Radial
symmetry for small masses was also independently established by Julin
in $\R^n$ for all $n\geq 3$ in the case of
$G(x) = \frac{1}{|x|^{n-2}}$ \cite{julin14}, and by Bonacini and
Cristoferi for a range of Riesz kernels \cite{bonacini14}.  In bounded
domains with a particular choice of boundary conditions, Cicalese and
Spadaro proved that minimizers are close to balls in the vanishing
mass limit, but cannot be exactly spherical unless the original domain
is a ball \cite{cicalese13}.  A further generalization to all Riesz
kernels and nonlocal perimeters is due to Figalli, Fusco, Maggi,
Millot and Morini \cite{figalli15} in the spirit of the currently
developing theory of nonlocal minimal surfaces (see, for example,
\cite{caffarelli10}).  Non-existence of minimizers for large masses
for Gamow's model was also independently established by Lu and Otto
\cite{lu14}, and Frank, Killip and Nam provided an explicit estimate
for the mass beyond which minimizers do not exist \cite{frank16}.  It
is currently an open problem whether the minimizers are balls whenever
they exist in the case of Gamow's model.

On the other hand, minimizers always exist in the periodic setting.
Alberti, Choksi and Otto \cite{alberti09} proved uniform distribution
of mass, which is a popular first step for approaching the question of
periodicity of minimizers. In various low volume fraction regimes,
Choksi and Peletier \cite{choksi10} and Kn\"upfer, Muratov and Novaga
\cite{kmn:cmp16} proved that the minimizers consist of small isolated
droplets, whose shape is asymptotically determined by solutions of the
whole space problem for certain masses (as already mentioned, they are
presently unknown, although conjectured to be balls; see also
\cite{frank15}).  Similar results are also available for Gamow's model
with screening and for the Ohta-Kawasaki energy, which can be
understood as a diffuse interface approximation to Gamow's liquid drop
model \cite{choksi11,m:cmp10,gms:arma13,gms:arma14}.  In particular,
it is proved that in two dimensions the droplets become almost
circular.  In other regimes, Sternberg and Topaloglu identified
stripes as the global minimizers in the case of the two-dimension
torus \cite{sternberg11}, and Morini and Sternberg showed that such
patterns also turn out to be minimizers in thin domains
\cite{morini14}.  Another class of (anisotropic) nonlocal
isoperimetric problems in dimensions $n \geq 2$ in which minimizers
are one-dimensional and periodic is given by Goldman and Runa \cite
{goldman16}, and by Daneri and Runa \cite{daneri18}.

In this paper, we study a nonlocal isoperimetric problem described by
\eqref{Egam} in two space dimensions, $n = 2$, in which the kernel $G$
is of {\em dipolar} type, i.e., $G(x) \simeq \frac{1}{|x|^3}$. This is
the interaction experienced at large distances by spins lying on the
plane and oriented perpendicularly to it. Just like the classical
Gamow's model, the model under consideration is relevant to the
multidomain patterns observed in perpendicularly magnetized thin film
ferromagnets \cite{hubert}, as well as ferroelectric films
\cite{strukov}, Langmuir monolayers \cite{andelman87} and ferrofluid
films subject to a strong perpendicular applied field
\cite{rosensweig,jackson94} (for an experimental realization involving
a magnetic garnet film, see Figure
\ref{fig:bubbles_labyrinths}). Note, however, that setting
$G(x) = \frac{1}{|x|^3}$ would result in an ill-defined problem
because of the divergence of the integral defining the nonlocal
contribution at short scales. Furthermore, redefining the energy in
the spirit of nonlocal minimal surfaces \cite{caffarelli10} to reduce
the integral to that over $\Omega$ and $\Omega^c$, up to an additive
constant, would not help, either, since the singularity of the
kernel is still too strong. Therefore, a genuine regularization at
short scale $\delta > 0$ is necessary to make sense of the energy in
\eqref{Egam} with this kind of kernel. This is a novel feature of the
considered nonlocal isoperimetric problem compared to those studied
previously.

\begin{figure}
	\centering
     \subcaptionbox{\label{fig:bubbles_labyrinth_1}}{
	  \centering
	  \includegraphics[height=4cm]{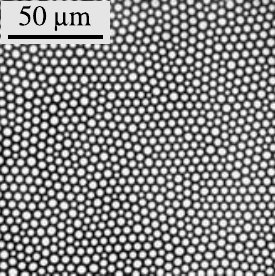}
     }
     \subcaptionbox{\label{fig:bubbles_labyrinth_2}}{
	  \centering
	  \includegraphics[height=4cm]{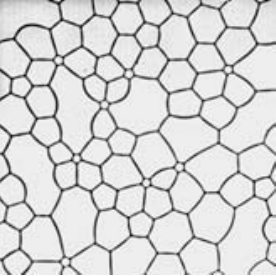}
     }
     \subcaptionbox{\label{fig:bubbles_labyrinth_3}}{
	  \centering
	  \includegraphics[height=4cm]{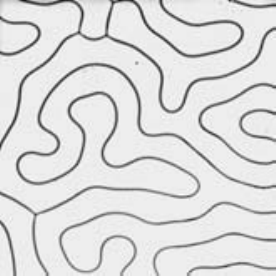}
     }
     \caption{Domain patterns in a magnetic garnet film with strong
       perpendicular anisotropy undergoing magnetization reversal
       driven by the applied magnetic field : a) (approximately)
       periodic arrangement of bubble domains; b) a network of
       interconnected stripes; c) a labyrinthine pattern. Reproduced
       from \cite[Figure 5.118]{hubert}.  }
        \label{fig:bubbles_labyrinths}
\end{figure}

More specifically, we study the following nonlocal isoperimetric
problem, as proposed by Kent-Dobias and Bernoff
\cite{kent-dobias15}. We wish to minimize
\begin{align}
  \label{Ekent0}
  E(\Omega) := \alpha P(\Omega) + \frac{\beta}{2} \int_\Omega
  \int_\Omega \frac{g_\delta(|x - y|)}{|x - y|^3} \intd x \intd y, 
\end{align}
among all finite perimeter sets $\Omega \subset \mathcal{A}_m$ of
fixed mass $m>0$, i.e., with
\begin{align}
	\label{admissible}
  \mathcal{A}_m :=\{\Omega \subset \R^2: P(\Omega) < \infty, \
  |\Omega| = m \}. 
\end{align}
Here, $\alpha, \beta > 0$ are fixed parameters, $P(\Omega)$ is the
perimeter of a measurable set $\Omega$ in the sense of De Giorgi
\cite{ambrosio}:
\begin{align}
  \label{P}
  P(\Omega) :=  \sup \left\{ \int_\Omega \nabla \cdot \phi \,
  \intd x:
  \, \phi \in C^1_c(\mathbb R^2;\mathbb R^2), \ |\phi|
  \leq 1 \right\},
\end{align}
and $g_\delta(r) := g(r / \delta)$ is a cutoff function at scale
$\delta > 0$ that makes the integral well-defined. The specific choice
is not essential, so for simplicity we work with
$g(r) := \chi_{(1,\infty)} (r)$ throughout the rest of the paper. Note
that by a rescaling we can always choose $\alpha = 1$.

We are interested in the regime in which $\delta$ is much smaller than
the characteristic length scale of minimizers, expressing the physical
condition that the regularization happens on the atomic scale, which
is much smaller than the scale of the observed patterns. Ultimately,
we wish to send the parameter $\delta$ to zero to obtain results that
are insensitive to the short-scale cutoff. As we show below, for a
meaningful limit as $\delta \to 0$ to exist at fixed value of $m > 0$
it is necessary to renormalize the strength $\beta$ of the dipolar
interaction. We set $\beta := \lambda / |\log \delta|$ for some
$\lambda > 0$ and restrict ourselves to the case
$\delta < \frac{1}{2}$.  Then the nonlocal term can be rewritten, so
that, up to an additive constant depending only on $m$, $\lambda$ and
$\delta$, and with $\alpha = 1$ the energy in \eqref{Ekent0} is equal
to
\begin{align}\label{energy}
  E_{\lambda,\delta}(\Omega) := P(\Omega) -  \frac{\lambda}{4|\log
  \delta|} \int_{\R^2} \int_{\R^2} |\chi_\Omega(x+z) -
  \chi_\Omega (x) |^2 \, \frac{g_\delta(|z|)}{|z|^3} \intd z \intd x, 
\end{align}
where $\chi_{\Omega}$ is the characteristic function of the set
$\Omega$.

First, we will make sure to demonstrate that minimizers always exist
in the generalized sense of consisting of finitely many components
that are ``infinitely far apart'' from each other.  In the subcritical
regime, $\lambda <1$, we then prove that the $\Gamma$-limit of the
energy is given by $(1-\lambda) P(\Omega)$, i.e., the nonlocal term
localizes to leading order and renormalizes the perimeter as
$\delta \to 0$. Moreover, we prove in Theorem
\ref{thm:characterization_minimizers} that the minimizers of
$E_{\lambda,\delta}$ exist also in the classical sense and are in fact
disks for all $\delta \ll 1$.  The strategy is to make use of
sufficiently uniform regularity estimates for minimizers, first on the
level of density estimates and then in terms of curvature, as well as
stability of the disk with respect to perturbations of curvature.

On the other hand, it is easy to see that disks are no longer
classical minimizers for $\delta \ll1$ as soon as $\lambda >1$, as it
is more convenient to split a single disk into multiple components.
In fact, by proving that (the components of generalized) minimizers
cannot contain disks of radius larger than $r(\delta)>0$ with
$r(\delta) \to 0$ as $\delta \to 0$ we see that generalized minimizers
cannot be asymptotically well-behaved in this limit.  As such, a more
precise analysis likely requires further insight into the question
whether minimizers are large collections of disks or exhibit
stripe-like behavior (see Figure \ref{fig:bubbles_labyrinths}).

This naturally brings us to the critical case $\lambda=1$, where in
view of the above arguments a transition from classical, radial to
non-trivial minimizers occurs for $\delta \ll 1$.  In this case, we
compute the next-order $\Gamma$-limit as $\delta \to 0$ by observing
that the sequence $|\log\delta| E_{1,\delta}(\Omega)$ for a fixed set
$\Omega \subset \R^2$ of finite perimeter is monotone decreasing in
$\delta$.  Suitably integrating by parts allows us to represent the
limit in a closed form that in fact does not fall within the class of
problems given by \eqref{Ekent0}.  While we cannot directly address
the issue of minimizers for the limit, we are able to prove that after
modifying the functional by reducing the repulsion at infinity there
exist masses for which classical minimizers exist, but are not given
by disks.  The idea of the proof is to construct a long stripe with
large mass whose energy per mass is lower than the energy per mass of
disks, ruling out the optimality of any collection of disks.

The paper is organized as follows: In Section \ref{sec:main_results},
we give the precise statements of our main results. In Section
\ref{sec:general_properties}, we give various representations of the
energies and their rescaling behaviors.  Section \ref{sec:existence}
is dedicated to proving existence of generalized minimizers together
with control over the number of their components.  We also already
start to look into the regularity properties of minimizers on a
qualitative level.  Section \ref{sec:subcritical} contains a thorough
discussion of the subcritical case $\lambda < 1$, culminating in the
proof of Theorem \ref{thm:characterization_minimizers}.  The critical
case $\lambda=1$, and in particular the existence of non-radial
minimizers (Theorem \ref{thm:non-spherical_minimizers}) for the
modified problem is dealt with in Section \ref{sec:critical}. Finally,
Proposition \ref{prop:supercritical} characterizing the behavior of
generalized minimizers as $\delta \to 0$ in the supercritical case
$\lambda > 1$ is presented in Section \ref{sec:supercr-case-lambda}.

\section{Main results}\label{sec:main_results}
Our first result identifies the minimizers of $E_{\lambda,\delta}$
over $\mathcal{A}_m$ in the subcritical regime of small $\delta$ and
$0<\lambda <1$ as disks.  The significance of the condition
$\lambda <1$ is that even though the nonlocal term renormalizes the
perimeter in the limit $\delta \to 0$, we still retain control over
the perimeter as is evident by the $L^1$-$\Gamma$-limit of
$E_{\lambda,\delta}$ for $\delta \to 0$ being given by
$(1-\lambda) P(\Omega)$, see Proposition
\ref{prop:subcritical_gamma_convergence}.  This allows us to obtain
density estimates for $\Omega$ and the reduced boundary
$\partial^\ast \Omega$ (for the definition, see, e.g., \cite[Chapter
15]{maggi}) that are uniform in $\delta$.  As the localization also
turns out to take place in the Euler-Lagrange equation of
$E_{\lambda,\delta}$, we can control the curvature of minimizers,
which together with a stability property of disks allows us to
conclude that the minimizers are disks.

\begin{thm}\label{thm:characterization_minimizers}
  There exist universal constants $\sigma_1, \sigma_2>0$ with the
  following properties: Let $0<\lambda<1$ and $0 < \delta < \frac12$.
  Under the condition
  \begin{align}
    \label{eq:sigma1}
    \frac{\lambda}{(1 - \lambda) |\log \delta|} \leq \sigma_1   
  \end{align}
  there
  exists a minimizer of $E_{\lambda,\delta}$ over
  $\mathcal{A}_\pi$. Furthermore, if
  \begin{align}
    \label{eq:sigma2}
    \frac{\lambda}{(1 - \lambda)^5 |\log \delta|} \leq \sigma_2  
  \end{align}
  then the unit disk $\ball{0}{1}$ is the unique minimizer of
  $E_{\lambda,\delta}$ over $\mathcal{A}_\pi$, up to translations.
\end{thm}

\begin{remark}
  Note that we formulated the theorem for mass $m=\pi$ in order to
  have a somewhat clean criterion for the smallness of $\delta$ in
  terms of $\lambda$.  However, the rescaling given in Lemma
  \ref{lem:rescaling} allows to obtain a similar result for all masses
  by replacing $\delta$ with $\sqrt{\frac{\pi}{m}}\delta$ and
  $\lambda$ with
  $\frac{\left|\log\left(\sqrt{\frac{\pi}{m}}\delta \right)
    \right|}{|\log\delta|} \lambda$. Also note that when $m = \pi$ and
  $\lambda$ approaches 1 for a fixed value of $\delta$, condition
  \eqref{eq:sigma1} may still be satisfied, while condition
  \eqref{eq:sigma2} fails, indicating the possibility of existence of
  non-radial minimizers. Whether or not the latter indeed happens for
  minimizers of $E_{\lambda,\delta}$ over $\mathcal A_\pi$ is an open
  problem.
\end{remark}

As the first-order $\Gamma$-limit in the critical case $\lambda = 1$
degenerates, we have to compute a higher-order $\Gamma$-limit to
obtain a helpful limiting description.  It turns out that the
appropriate sequence to analyze is $|\log\delta| E_{1,\delta}$, as
suggested, for example, by the estimate of Proposition
\ref{prop:a_priori_lower_bound}. The $L^1$-$\Gamma$-limit is then
given by
\begin{align}\label{representation_closed}
  \begin{split}
    E_{1,0}(\Omega) := -P(\Omega) & + \frac12 \int_{\partial^* \Omega}
    \int_{\left(\Omega\Delta H_-(y)\right)\cap\ball{y}{1}
    }\left|\nu(y) \cdot \frac{x-y}{|x-y|^3}\right| \intd x \intd
    \Hd^1(y)\\
    & \quad + \frac12 \int_{\partial^* \Omega} \int_{\Omega \setminus
      \ball{y}{1} }\nu(y) \cdot \frac{y-x}{|y-x|^3} \intd x \intd
    \Hd^1(y),
  \end{split}
\end{align}
for $\Omega\in \mathcal{A}_m$, where $\nu(x)$ is the outer unit normal
of $\Omega$ for $x\in \partial^* \Omega$ and
$H_-(y):= \{(x-y)\cdot \nu(y) < 0\}$, as determined by the following
theorem.  A sketch indicating the various domains of integration in
\eqref{representation_closed} can be found in Figure
\ref{fig:domains_of_integration}.

\begin{thm}\label{thm:critical_convergence}
  For every $m > 0$, the $\Gamma$-limit of
  $|\log \delta| E_{1,\delta}$ restricted to $\mathcal{A}_m$, with
  respect to the $L^1$-topology as $\delta \to 0$, is given by
  $E_{1,0}$ in the following sense:
  
  \begin{enumerate}[(i)]
  \item (Lower bound) Let
    $\Omega,\, \Omega_{\delta_n} \in \mathcal{A}_m$ such that
    $|\Omega_{\delta_n} \Delta \Omega| \to 0$ and $\delta_n \to 0$ as
    $n \to \infty$.  Then it holds that
 	 	\begin{align}
                  E_{1,0}(\Omega) \leq \liminf_{n \to \infty}
                  |\log\delta_n| E_{1,\delta_n}(\Omega_{\delta_n}). 
  		\end{align}
              \item (Upper bound) For every set
                $\Omega \in \mathcal{A}_m$, a recovery sequence is
                given by the constant sequence, i.e., it holds that
  		\begin{align}
                  E_{1,0}(\Omega) = \lim_{\delta \to 0} |\log\delta|
                  E_{1,\delta}(\Omega) \in \R \cup \{+ \infty\}. 
  		\end{align}
  		In particular, the $\Gamma$-limit coincides with the
                pointwise limit.
              \item (Compactness) For every sequence
                $\Omega_{\delta_n} \in \mathcal A_m$ such that
                $\delta_n \to 0$ as $n\to \infty$ and
        \begin{align}
          \limsup_{n \to \infty} |\log\delta_n|
          E_{1,\delta_n}(\Omega_{\delta_n}) < \infty,
        \end{align}
        that in addition satisfies
        $\Omega_{\delta_n} \subset \ball{0}{R}$ for some $R > 0$,
        there exists a subsequence (not relabeled) and
        $\Omega\in \mathcal{A}_m$ such that
        $|\Omega_{\delta_n} \Delta \Omega| \to 0$.
  \end{enumerate}
      \end{thm}
	
	\begin{figure}
		\centering
		\begin{tikzpicture}[scale=1.3]
				\begin{scope}
						\clip (-2,1.4) -- (2,1.4) -- (2,-1.6)-- (-2,-1.6);
						\fill[opacity=.25,color=lightgray] (-2,0)--(2,0) --(2,-1.6)-- (-2,-1.6);
						\begin{scope}
							\clip plot [smooth,thick,tension=1] coordinates {(-10,-5)(-3.5,-5)(-3,0)(0,0)(3,0)(3.7,7)(10,-5)(-10,-5)};
							\draw[pattern=north west lines] (-2,1) -- (2,1) -- (2.2,-2)-- (-2,-2);
							\begin{scope}
							 	\clip circle [radius=1.3];
							 	\fill[white] (-2,1) -- (2,1) -- (2.2,-2)-- (-2,-2);
							 	\fill[opacity=.25,color=lightgray] (-2,0)--(2,0) --(2,-1.6)-- (-2,-1.6);
							\end{scope}
						\end{scope}
						\begin{scope}
							\clip plot [smooth,thick,tension=1] coordinates {(-10,-5)(-3.5,-5)(-3,0)(0,0)(3,0)(3.7,7)(10,-5)(-10,-5)};
							\clip circle [radius=1.3];
							\draw[pattern=north east lines] (-2,1) -- (2,1) -- (2.2,-2)-- (-2,-2);
							\begin{scope}
							 	\clip (0,0) circle [radius=1.3];
							 	\clip (-1.3,0)--(1.3,0) -- (1.3,-1.3) -- (-1.3,-1.3) -- cycle;
							 	\fill[white] (-2,1) -- (2,1) -- (2.2,-2)-- (-2,-2);
							 	\fill[opacity=.25,color=lightgray] (-2,0)--(2,0) --(2,-1.6)-- (-2,-1.6);
							\end{scope}
						\end{scope}
						\begin{scope}
							\clip (0,0) circle [radius=1.3];
							 \clip (-1.3,0)--(1.3,0) -- (1.3,-1.3) -- (-1.3,-1.3) -- cycle;
							 \draw[pattern=north east lines] (-2,1) -- (2,1) -- (2.2,-2)-- (-2,-2);
							 \begin{scope}
                                                           \clip plot
                                                           [smooth,thick,tension=1]
                                                           coordinates
                                                           {(-10,-5)(-3.5,-5)(-3,0)(0,0)(3,0)(3.7,7)(10,-5)(-10,-5)};
                                                           \fill[white]
                                                           (-2,1) --
                                                           (2,1) --
                                                           (2.2,-2)--
                                                           (-2,-2);
                                                           \fill[opacity=.25,color=lightgray]
                                                           (-2,0)--(2,0)
                                                           --(2,-1.6)--
                                                           (-2,-1.6);
							 \end{scope}
						\end{scope}
						\begin{scope}
							\draw (0,0) circle [radius=1.3];

						\end{scope}
						
						\draw[thick] plot
                                                [smooth,thick,tension=1]
                                                coordinates
                                                {(-10,-5)(-3.5,-5)(-3,0)(0,0)(3,0)(3.7,7)(10,-5)(-10,-5)};
                                                \draw[{<[length=1.5mm,width=1.5mm]}-{>[length=1.5mm,width=1.5mm]}]
                                                (0,0) --
                                                node[anchor=east]{$1$}
                                                (310:1.3);
                                                \draw[-{Latex[width=1.5mm]}]
                                                (0,0) -- (0,.6)
                                                node[anchor=west]{$\nu(y)$};
					\end{scope}
					
					\draw[thick] (-2,0)--(2,0);
                                        \node[inner
                                        sep=1pt,fill=white] at
                                        (2,-.25) {$H_-(y)$};
                                        \node[fill=white] at
                                        (-1.6,-1.3) {$\Omega$};
                                        \fill[black] (0,0)
                                        circle[radius=1.5pt]node[anchor=south
                                        east]{$y$};
		\end{tikzpicture}
		\caption{\label{fig:domains_of_integration} Sketch
                  indicating the domains of integration around
                  $y \in \partial^\ast \Omega$ in the limiting energy
                  $E_{1,0}(\Omega)$. The first integral term in
                  \eqref{representation_closed} integrates over the
                  dashed region inside the indicated circle, while the
                  second integral term integrates over the dashed
                  region outside that circle. The half-plane $H_-(y)$
                  is shown in gray.}
	\end{figure}
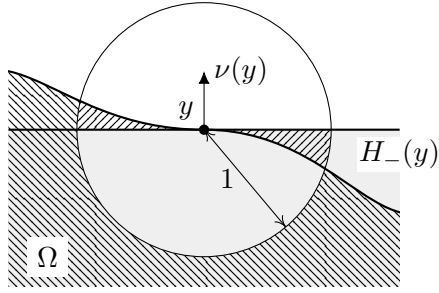

      The limiting functional has previously been investigated
      numerically by Bernoff and Kent-Dobias \cite{kent-dobias15} on
      the basis of a formal calculation. We are also able to
      rigorously justify their representation and extend it to sets
      with a $C^2$-boundary.

\begin{prop}\label{prop:representations_limit}
  If $\Omega \subset \R^2$ is a set of finite perimeter with a
  $C^2$-boundary, which is decomposed into positively oriented
    closed Jordan curves $\gamma_i:[0,P_i] \to \R^2$ of length $P_i$
  parametrized by arc length for $1\leq i \leq N$, with some
  $N \in\N$, then it holds that
  \begin{align}\label{representation_bernoff}
    \begin{split}
    E_{1,0}(\Omega)
    & = - \sum_{i=1}^N P_i \left[ \log\left(\frac{P_i}{2} \right)+ 2 
      \right] + \frac{1}{2} \sum_{i=1}^N  \int_0^{P_i}
      \int_{-\frac{P_i}{2}}^{\frac{P_i}{2}} \left( 
      \frac{1}{|s|} - \frac{\dot\gamma^\perp_i(t+s)\cdot
      \dot\gamma^\perp_i(t) }{|\gamma_i(t+s) - 
      \gamma_i(t)|} \right) \intd s \intd t\\
    & \quad - \sum_{i=1}^{N-1}
      {\sum_{j=i+1}^N} \int_{0}^{P_i}\int_{0}^{P_j}
      \frac{\dot\gamma^\perp_j(s)\cdot\dot
      \gamma^\perp_i(t)}{|\gamma_j(s)-\gamma_i(t)|} 
      \intd s 
      \intd t.
    \end{split}      
  \end{align}
\end{prop}

When trying to prove a compactness result for the whole space
corresponding to Theorem \ref{thm:critical_convergence}, one has to
deal with the issue that the functionals $E_{\lambda,\delta}$ have no
a priori compactness due to their translational symmetry. Already for
fixed parameters this issue manifests itself in minimizing sequences
potentially breaking up into multiple pieces which the repulsive
nonlocal term then pushes infinitely far apart from each other.  In
particular, we can in general only hope to prove that minimizers exist
in a generalized sense:

\begin{Def}[{\cite[Definition
    4.3]{kmn:cmp16}}]\label{def:generalized_minimizers}
  For $m > 0$, consider a functional
  $E : \mathcal{A}_m \to \R\cup \{+ \infty\}$.  We say that a
  generalized minimizer of $E$ over $\mathcal A_m$ is a minimizer of
  the functional
	\begin{align}
          \mathcal{S}E\left((\Omega_i)_{i \in \N}\right) =
          \sum_{i=1}^\infty E(\Omega_i) 
	\end{align}
	defined on the domain
	\begin{align}
          \mathcal{SA}_m := \left\{ (\Omega_i)_{i \in \N}: |\Omega_i|
          >0 \text{ for at most finitely many } i \in \N,
          P(\Omega_i) < \infty, \sum_{i=1}^\infty |\Omega_i| = m
          \right\}. 
	\end{align}
      \end{Def}

      The main point of this definition is finiteness of the number of
      pieces, which relies on almost-minimality of minimizing
      sequences and is not true for arbitrary sequences of sets with
      uniformly bounded perimeters.  The existence of generalized
      minimizers for all parameters $\lambda, \delta, m>0$ is the main
      content of Section \ref{sec:existence}.

Coming back to augmenting Theorem \ref{thm:critical_convergence} with
a compactness statement, we see that we have the same issue of
sequences of sets breaking up into potentially infinitely many parts
in the limit $\delta \to 0$.  Instead of dealing with the
technicalities of formulating a framework that can handle this issue,
we choose to only analyze the compactness properties of minimizers of
$E_{1,\delta}$, which in any case is the main consequence one would
like to extract from such a statement.

\begin{prop}\label{prop:asymptotic_compactness}
  For $m > 0$, generalized minimizers of $E_{1,\delta}$ over
  $\mathcal{A}_m$ are compact in the sense that their number of
  components is uniformly bounded as $\delta \to 0$, and each
  component is compact in $L^1$ after translation.  Additionally, the
  limiting collection of sets is a generalized minimizer of $E_{1,0}$.
\end{prop}	

Finally, we turn to investigating the properties of minimizers of
$E_{\delta,1}$ for $\delta \ll 1$. We note that we currently do not
know  what they are for $m > 0$ fixed
  and all $\delta > 0$ sufficiently small. In particular, they may or
may not be disks, depending on the values of $m$, provided that
$\delta \ll 1$. Results of experiments such as the images shown in
Figure \ref{fig:bubbles_labyrinths} suggest that in some regimes the
minimizers may indeed be disks, while in other regimes the situation
may be more complex. Our ansatz-based computations, however, indicate
that disks are always preferred over stripe-like constructions that
are tractable analytically, both at relatively low mass and in the
limit of large masses. One could hope that numerically optimized
stripes such as the ones obtained by Bernoff and Kent-Dobias
\cite{kent-dobias15} could have lower energy than disks, but this does
not seem to be the case.  Also a linear stability analysis is
unhelpful, as instability of a single disk only occurs at masses at
which two disks of equal mass have strictly lower energy. In other
words, it seems that splitting a domain into finitely many disks may
always decrease energy, thus making stripes always energetically
disadvantageous. In fact, one might be tempted to argue that the
persistence of balls as generalized minimizers is a universal feature
of systems described by \eqref{Egam}.

In the following we demonstrate that the above view is too simplistic
and that the conjecture about the universal optimality of balls in the
context of \eqref{Egam} is false. Indeed, we will show that under
appropriate conditions stripes produce better candidates for
generalized minimizers than collections of disks in $\R^2$. In order
to tilt the favor towards stripes, for $l > 0$ we consider the energy
\begin{align}\label{energy_modified}
  F_{\lambda,\delta,l}(\Omega) := P(\Omega) -
  \frac{\lambda}{4|\log\delta|} \int_{\R^2} \int_{\R^2}
  |\chi_\Omega(x+z) - \chi_\Omega (x) |^2 K_{\delta,l}(|z|) \intd z
  \intd x,
\end{align}
where
\begin{align}\label{kernel}
  K_{\delta,l}(r):=\frac{g_\delta(r)}{r^3} - \frac{r^2 - 2
  l^2}{(r^2+l^2)^{5/2}}. 
\end{align}
This is precisely the energy of two identical, but oppositely oriented
dipolar patches lying in the parallel planes separated by distance
$l$. Such a model is, for example, relevant to synthetic
antiferromagnets, in which the antiparallel alignment of spins in
adjacent layers is favored by antiferromagnetic exchange coupling
through a spacer layer (see, e.g., \cite{moser02}). Heuristically, the
kernel behaves as in a dipolar layer on scales $r \lesssim l$, while
on large scales the kernel decays faster, thus reducing the long-range
repulsion in the far field.  Notice that we recover the original
energy in the limit $l\to \infty$.

From the point of view of Theorem \ref{thm:critical_convergence}, the
modification is merely a continuous perturbation:
\begin{lemma}\label{lem:modified_compactness}
  For all $l>0$ and $m>0$, the $L^1$-$\Gamma$-limit of
  $F_{1,\delta,l}$ restricted to $\mathcal{A}_m$ is given, as
    $\delta \to 0$, by
\begin{align}\label{limit_modified}
  F_{1,0,l}\left(\Omega\right)  := E_{0,1}\left(\Omega \right) +
  \frac{1}{4} \int_{\R^2} \int_{\R^2} |\chi_{\Omega}(x+z) -
  \chi_{\Omega} (x) |^2 \frac{|z|^2 - 2 l^2}{(|z|^2+l^2)^{5/2}} \intd
  z \intd x
\end{align}
for $\Omega \in \mathcal{A}_m$.  Also for $F_{1,\delta,l}$ the number
of components of generalized minimizers over $\mathcal{A}_m$ is
uniformly bounded in $\delta$, and generalized minimizers of
$F_{1,\delta,l}$ converge to generalized minimizers of $F_{1,0,l}$ in
the same sense as in Proposition \ref{prop:asymptotic_compactness}.
\end{lemma}

For this energy we can prove existence of non-radial minimizers.
In combination with Lemma \ref{lem:modified_compactness}, this
  represents the first instance, to our knowledge, in which whole
  space minimizers of a problem belonging to the class in \eqref{Egam}
  are not radially symmetric.

\begin{thm}\label{thm:non-spherical_minimizers}
  There exists a universal constant $c>0$ with the following
  property: Let $0 < l -\frac{2}{e^2}<c$. Then there exists
  $M=M(l)>0$ such that for all masses $m>M$ any generalized
  minimizer of $F_{1,0,l}$ over $\mathcal{A}_m$ has at least one
  component which is not a disk.  In particular, there exist
  masses $m>0$ at which classical minimizers exist and are non-radial.
\end{thm}

We conclude by briefly indicating that generalized minimizers in the
supercritical case need to exhibit strongly irregular behavior, in the
sense that their components cannot contain disks of $O(1)$-radii in
the limit $\delta\to 0$.  Note that such a statement cannot
distinguish between generalized minimizers consisting of a large
number of disks or thin, fingered structures.  The point is that for
$\lambda >1$ it is beneficial to increase the perimeter of sets as the
nonlocal term overcompensates the local contribution.  We exploit this
in the following proposition by cutting out disks from the minimizer
and placing them at infinity.  We expect the resulting relation
$r\geq C\delta^\frac{\lambda-1}{\lambda}$ between the radius $r$ of
the disks and $\delta$ to be sharp in terms of the exponent
$\frac{\lambda-1}{\lambda}$, as the rescaling properties of
$E_{\lambda,\delta}$, see Lemma \ref{lem:rescaling}, imply
$E_{\lambda,\delta} \left( \delta^\frac{\lambda-1}{\lambda}\Omega
\right) = \delta^\frac{\lambda - 1}{\lambda} E_{1,\delta^{1 /
    \lambda}}(\Omega)$ for all sets $\Omega \subset \R^2$ of finite
perimeter.

\begin{prop}\label{prop:supercritical}
  There exists a universal constant $C>0$ such that for all
  $\lambda > 1$ and all $\delta > 0$ sufficiently small depending only
  on $\lambda$ the following holds: If
  $|\ball{0}{r} \setminus \Omega| = 0$ for a set $\Omega \subset \R^2$
  of finite perimeter and some
  $r \geq C\delta^{\frac{\lambda-1}{\lambda}}$, then there exists a
  set $\widetilde \Omega \subset \R^2$ and $\tilde r > 0$ such that
	\begin{align*}
          |\Omega|
          & = \left|\widetilde \Omega\right| + | \ball{0}{\tilde r}|,\\
          E_{\lambda,\delta}(\Omega)
          & > E_{\lambda,\delta}\left(\widetilde
            \Omega \right) +
            E_{\lambda,\delta}
            (\ball{0}{ r}). 
\end{align*}
\end{prop}

\noindent {\em Notation:} Within proofs, the symbols ``$c$'' and
``$C$'' denote universal constants that may change from line to line.
Furthermore, for $(a,b) \in \R^2$ we define $(a,b)^\perp:= (-b,a)$ to
be its counter-clockwise rotation by 90 degrees.

\section{General considerations}\label{sec:general_properties}
\subsection{Various representations of $E_{\lambda,\delta}$ and
  $F_{\lambda,\delta,l}$.}

Depending on the specific situation, it is helpful to write the energy
in one way or another.  For example, the nonlocal term can be
interpreted as a kind of nonlocal perimeter via straightforward
combinatorics
\begin{align}\label{energy_as_nonlocal_perimeter}
  \int_{\R^2} \int_{\R^2} |\chi_\Omega(x+z) - \chi_\Omega (x) |^2
  K_{\delta,l}(|z|) \intd z \intd x = 2 \int_{\Omega}
  \int_{\stcomp{\Omega}}K_{\delta,l}(|x-y|)\intd y \intd x. 
	\end{align}
In terms of the potential
\begin{align}\label{potential_definition}
  v_{\delta,l}(x) :=  \frac{\lambda}{2 |\log \delta |}\int_\Omega
  K_{\delta,l}(|x-y|) \intd y, 
\end{align}
we can also represent it as
\begin{align}\label{energy_expanded}
  F_{\lambda,\delta,l}(\Omega) = P(\Omega) +  \int_{\Omega} v(x)
  \intd x - \frac{\pi \lambda}{ \delta |\log \delta |}  m, 
\end{align}
because we have
$\int_{\R^2} \frac{|z|^2 - 2 l^2}{(|z|^2+l^2)^{5/2}} \intd z = 0$.

For explicit calculations it is helpful to rewrite both the energy and
the potential in terms of boundary integrals, using integration by
parts.  To this end, we define $\Phi_\delta :\R \to \R$ solving
\begin{align}\label{Phi_definition}
  \Delta_z \Phi_\delta(|z|) = \frac{g_\delta(|z|)}{|z|^3} \text{
  subject to }\Phi_\delta(r) \to 0 \text{ as } r\to 0. 
\end{align}
Writing the Laplacian in polar coordinates allows to straightforwardly
determine the solution to be
\begin{align}\label{Phi_explicit}
	\Phi_\delta(r) =
	\begin{cases}
          \frac{1}{r} & \text{ if } r\geq \delta\\
          \frac{1}{\delta}\left(1 - \log\left(\frac{r}{\delta}\right)
          \right) & \text{ if } r< \delta.
	\end{cases}
\end{align}
Similarly, we define $\Phi_{\delta,l}: \R\to \R$ as the unique
solution of
\begin{align}\label{Phi_l_definition}
  \Delta_z \Phi_{\delta,l}(|z|) = K_{\delta,l}(|z|) \text{ subject to
  }\Phi_\delta(r) \to 0 \text{ as } r\to 0, 
\end{align}
which gives
\begin{align}\label{Phi_l_explicit}
	\Phi_{\delta,l}(r) = \Phi_{\delta}(r) - \frac{1}{\sqrt{r^2+ l^2}}.
\end{align}
  
\begin{lemma}\label{lem:energy_on_boundary}
  For $\delta>0$, $m > 0$, $0<l \leq \infty$,
  $\Omega \in \mathcal A_m$ and $x\in \Omega$ the potential
  $v_{\delta,l}(x)$ can be written as
	\begin{align}\label{potential_boundary_formulation}
          v_{\delta,l}(x) =  \frac{\lambda}{2
          |\log\delta|}\int_{\partial^* \Omega} \nabla \Phi_{\delta,l}
          (y-x) \cdot \nu (y) \intd \Hd^1(y) + \frac{\pi
          \lambda}{\delta  |\log \delta |}  m, 
	\end{align}
	where $\nu$ is the outward-pointing unit normal.

	If additionally the set $\Omega$ is such that there exists
        $C_\Omega > 0$ such that
	\begin{align}\label{mild_regularity}
          \sup_{r>0} \frac{\Hd^1(\partial^* \Omega \cap
          \ball{x}{r})}{r} < C_\Omega 
	\end{align}
        for $\Hd^1$-almost every $x\in \partial^* \Omega$, then the
        energy can be represented as
	\begin{align}\label{energy_on_boundary}
          F_{\lambda,\delta,l}(\Omega) = P(\Omega) -
          \frac{\lambda}{2|\log \delta|} \int_{\partial^* \Omega}
          \int_{\partial^* \Omega} \nu(x) \cdot \nu(y)
          \Phi_{\delta,l}(|x-y|) \intd \Hd^1(y) \intd \Hd^1(x). 
	\end{align}
\end{lemma}
	
\begin{proof}
  We only need to rewrite the nonlocal term.  By integrability of the
  kernel, for every $x \in \Omega$ we get
  \begin{align}
    \begin{split}
      \int_{\stcomp{\Omega}} K_{\delta,l}(|x-y|) \intd y & =
      \lim_{r\to0} \lim_{R\to \infty}\int_{\stcomp{\Omega}\cap
        (\ball{x}{R}\setminus \ball{x}{r})}\Delta_y
      \Phi_{\delta,l}(|y-x|)
      \intd y \\
      & = \lim_{r\to0} \lim_{R\to \infty} \int_{\partial^\ast
        \left(\stcomp{\Omega}\cap (\ball{x}{R}\setminus
          \ball{x}{r})\right)} \tilde \nu(y) \cdot \nabla_y
      \Phi_{\delta,l}(|y-x|) \intd \Hd^1(y) ,
    \end{split}
  \end{align}
  where $\tilde \nu$ is the outward unit normal to the set of finite
  perimeter $\stcomp{\Omega}\cap (\ball{x}{R}\setminus \ball{x}{r})$,
  see for example \cite[Lemma 15.12]{maggi}.  The same lemma
  furthermore implies that for almost all $r,R>0$ we have the equality
  of measures
	\begin{align}\label{cut-off}
		\begin{split}
                  \tilde \nu(y) \Hd^1|_{\partial^* 
                    \left(\stcomp{\Omega}\cap (\ball{x}{R}\setminus
                      \ball{x}{r})\right)}(y) & = - \nu(y)
                  \Hd^1|_{\partial^*\Omega \cap
                    (\ball{x}{R}\setminus \ball{x}{r})}(y) \\
                  & \quad + \frac{y-x}{|y-x|}\Hd^1|_{\stcomp{\Omega}
                    \cap \partial \ball{x}{R}}(y) \\
                  & \quad -\frac{y-x}{|y-x|}\Hd^1|_{\stcomp{\Omega}
                    \cap \partial \ball{x}{r}}(y) .
		\end{split}
	\end{align}
	
	We first remove the second term.  As $\Phi'_{\delta,l}(R) $
        decays at infinity like $\frac{1}{R^2}$, we have
	\begin{align}
		\begin{split}
                  \lim_{R\to \infty} \left| \int_{\stcomp{\Omega} \cap
                      \partial \ball{x}{R}} \frac{y-x}{|y-x|}\cdot
                    \nabla_y \Phi_{\delta,l}(|y-x|) \intd \Hd^1(y)
                  \right| \leq \lim_{R\to \infty}
                  \frac{C_{\delta,l}}{R} = 0.
		\end{split}
	\end{align}
	Similarly, for every $r>0$ we have
	\begin{align}
		\begin{split}
		& \quad \lim_{R\to \infty} \int_{\partial^\ast
                    {\Omega}\cap (\ball{x}{R}\setminus
                      \ball{x}{r})} \nu(y) \cdot
                  \nabla_y \Phi_{\delta,l}(|y-x|) \intd \Hd^1(y)\\
           &= \int_{\partial^\ast
                    \Omega\setminus
                      \ball{x}{r}} \nu(y) \cdot
                  \nabla_y \Phi_{\delta,l}(|y-x|) \intd \Hd^1(y).
                \end{split}
	\end{align}
	
	Next, we argue that for almost all $x\in \Omega$ the third
        term vanishes in the limit $r \to 0$ along a suitable
        sequence.  To this end, note that $\Phi'_{\delta,l}(r) $
        blows up at the origin as $\frac{1}{r}$.  As
	\begin{align}
          \left|\Omega \setminus \Omega^1 \right| = 0, \qquad
          \Omega^1 := \left\{x\in \R^2:\lim_{r\to 0} \frac{|\Omega
          \cap \ball{x}{r}|}{|\ball{x}{r}|} =1 \right\}, 
	\end{align}
	i.e., the set $\Omega$ is up to a set of measure zero given by
        its points of $\Leb^2$-density one, we only have to prove that
	\begin{align}\label{choose_subsequence}
		\begin{split}
                  \lim_{r\to 0} \left( \operatorname{ess\
                      inf}_{0<\tilde r < r}
                    \frac{\Hd^1(\stcomp{\Omega} \cap \partial
                      \ball{x}{\tilde r})}{\tilde r} \right) = 0.
		\end{split}
	\end{align}
	Indeed, otherwise there would exist $\eps >0$ and $r>0$ such
        that 
	\begin{align}
		\begin{split}
                  \Hd^1(\stcomp{\Omega} \cap \partial \ball{x}{\tilde
                    r})\geq \eps \tilde r
		\end{split}
	\end{align}
	for almost all $0<\tilde r< r$.  Integrating over $0<\tilde
        r<\bar r$ for all $0<\bar r <r$ then gives
	\begin{align}
          |\stcomp{\Omega} \cap \ball{x}{\bar r} | \geq \frac{\eps}{2}
          \bar r^2, 
	\end{align}
	and thus $x$ cannot be a point of $\Leb^2$-density one of
        $\Omega$.

	In order to carry out the remaining limit, we note that
        $\lim_{r\to 0} \frac{\Hd^1|_{\partial^*
            \Omega}(\ball{x}{r})}{r}=0$ for $\Hd^1$-almost all
        $x\in \Omega_1 \subset \R^2 \setminus \partial^* \Omega$ due
        to \cite[Corollary 6.5]{maggi}. As a result, we get
	\begin{align}
          \lim_{n\to \infty}\int_{\partial^\ast
          \Omega\setminus
          \ball{x}{r_n}} \nu(y) \cdot
          \nabla_y \Phi_{\delta,l}(|y-x|) \intd \Hd^1(y) = \int_{\partial^\ast
          \Omega} \nu(y) \cdot
          \nabla_y \Phi_{\delta,l}(|y-x|) \intd \Hd^1(y)
	\end{align}
	for almost all $x\in \Omega$, where $r_n$ is a sequence chosen
        due to equation \eqref{choose_subsequence}.  Therefore, for
        almost all $x\in \Omega$ we obtain
	\begin{align}
          \int_{\stcomp{\Omega}} K_{\delta,l}(|x-y|) \intd y
          = - \int_{\partial^* \Omega} \nu(y) \cdot \nabla_y
          \Phi_{\delta,l}(|y-x|) \intd \Hd^1(y).
	\end{align}
	
	Consequently, after adding and subtracting the appropriate
        integral over $\stcomp{\Omega}$ in definition
        \eqref{potential_definition} we get
	\begin{align}
		\begin{split}
          v_{\delta,l}(x) &  =  \frac{\pi \lambda}{ \delta
                            |\log \delta |} m - \frac{\lambda}{2
                            |\log \delta |}\int_{\stcomp{\Omega}}
                            K_{\delta,l}(|x-y|)\intd y\\ 
                          &  =  \frac{\pi \lambda}{ \delta
                            |\log \delta |} m  +  \frac{\lambda}{2
                            |\log \delta |}\int_{\partial \Omega}
                            \nu(y) \cdot \nabla_y
                            \Phi_{\delta,l}(|y-x|) \intd \Hd^1(y).                           
       \end{split}                  
	\end{align}

	Going back to the original double integral we see that
	\begin{align}\label{fubini?}
		\begin{split}
                  \int_{\Omega}
                  \int_{\stcomp{\Omega}}K_{\delta,l}(|x-y|) \intd y
                  \intd x & =\int_{\Omega} \int_{\partial^*
                    \Omega} \nu(y) \cdot(- \nabla_y)
                  \Phi_{\delta,l}(|y-x|) \intd
                  \Hd^1(y) \intd x\\
                  & = \int_{\partial^* \Omega} \int_{\Omega}
                  \nu(y) \cdot \nabla_x \Phi_{\delta,l}(|y-x|)\intd x
                  \intd \Hd^1(y).
		\end{split}
	\end{align}
	Here we were able to apply Fubini's theorem, since the
        integrand on the right hand side satisfies
	\begin{align}
          |\nu(y) \cdot(- \nabla_y) \Phi_{\delta,l}(|y-x|)| \leq
          \frac{C_{\delta,l}}{|x-y|},
	\end{align}
	due to the explicit representation \eqref{Phi_l_explicit}, and
        thus is integrable over $\Omega \times \partial^* \Omega$
        w.r.t.\ the product measure $\Leb^2 \otimes \Hd^1$.
	
        Using a similar choice of small scale cut-offs as in
        \eqref{cut-off} with $x$ and $y$ interchanged, we get
	\begin{align}
		\begin{split}
          & \quad 	\int_{\Omega} \nu(y) \cdot \nabla_x
            \Phi_{\delta,l}(|y-x|)\intd x\\ 
          & = \lim_{\eps \to 0}  \int_{\Omega\setminus \ball{y}{\eps}}
            \nu(y) \cdot \nabla_x \Phi_{\delta,l}(|y-x|)\intd x  \\ 
          & =   \lim_{\eps \to 0} \int_{\partial^*(\Omega\setminus
            \ball{y}{\eps})} \nu(y) \cdot \tilde \nu(x)
            \Phi_{\delta,l}(|y-x|)\intd \Hd^1(x)  \\ 
          & =\lim_{\eps \to 0}   \int_{\partial^* \Omega \setminus
            \ball{y}{\eps}} \nu(y) \cdot \nu(x)
            \Phi_{\delta,l}(|y-x|)\intd \Hd^1(x),
       \end{split}
	\end{align}
	where in the last step we used the fact that $\eps
        \Phi_{\delta,l}(\eps) \to 0$. 
	
	Under the mild regularity assumption \eqref{mild_regularity}
        we can cover $\ball{y}{1}$ by a collection of overlapping
        dyadic annuli $\ball{y}{2^{-k+1}}\setminus\ball{y}{2^{-k-1}}$
        for $k\in \N$, and for $\Hd^1$-almost all
        $y \in \partial^* \Omega$ obtain
	\begin{align}
          \int_{\partial^* \Omega} |\log(|x-y|)| \intd \Hd^1(x) \leq
          C_\Omega \sum_{k\in \N} 2^{-k+1} |\log(2^{-k-1})| \leq  C_\Omega
          \sum_{k\in \N} k \cdot 2^{-k} < \infty.  
	\end{align}
	Via integrability, we obtain both
	\begin{align}\label{double_boundary_integral}
          \int_{\Omega} \nu(y) \cdot \nabla_x
          \Phi_{\delta,l}(|y-x|)\intd x = \int_{\partial^* \Omega}
          \nu(y) \cdot \nu(x) \Phi_{\delta,l}(|y-x|)\intd \Hd^1(x), 
	\end{align}
	and the desired statement \eqref{energy_on_boundary}.
\end{proof}

\subsection{Rescaling}
We first record the behavior of the energy under rescalings.
Especially the inequality \eqref{energy_drops} below turns out to be
of central importance in proving that generalized minimizers exist, as
well as in controlling the number of pieces.

\begin{lemma}\label{lem:rescaling} 
  Let $\alpha > 0$, $m > 0$, $\Omega \in \mathcal A_m$ and
  $\widetilde \Omega := \alpha^{-1} \Omega$. Then we have
  $\widetilde \Omega \in \mathcal A_{\widetilde m}$ and
	\begin{align}
          \alpha  F_{\tilde \lambda, \tilde \delta,\tilde l}(\widetilde
          \Omega) = F_{\lambda,\delta,l}(\Omega),
	\end{align}
	where $\alpha^2 \widetilde m = m$, $\alpha \tilde \delta =
        \delta$, $\alpha \tilde l=l$ and
        $\frac{\tilde \lambda}{\big|\log\tilde \delta \big|} =
        \frac{\lambda}{|\log\delta|}$.  If $\alpha<1$, then we also
        have 
	\begin{align}\label{energy_drops}
          \alpha \left(F_{\lambda, \delta,l}(\widetilde \Omega) - \frac{6
          \pi \widetilde m \lambda }{5 |\log\delta| l} \right)<
          F_{\lambda,\delta,l}(\Omega) - \frac{6 \pi m \lambda}{5
          |\log\delta| l}. 
	\end{align}
	For the functional $F_{1,0,l}$ it holds that
        \begin{align}
          \label{F10lresc}
          F_{1,0,l}(\Omega) = \alpha \left( F_{1,0,\tilde
          l}(\widetilde \Omega) - P(\widetilde \Omega) \log
          \alpha \right).
	\end{align}
\end{lemma}

\begin{proof}
	We compute
	\begin{align}
		\begin{split}
                  \alpha F_{\tilde \lambda, \tilde \delta, \tilde
                    l}(\widetilde \Omega) & = P(\Omega) -
                  \alpha\frac{\tilde \lambda}{2|\log\tilde \delta|}
                  \int_{\alpha^{-1} \Omega} \int_{\stcomp{\alpha^{-1}
                      \Omega}} \left(\frac{g_{\tilde
                        \delta}(|x-y|)}{|x-y|^3} - \frac{|x-y|^2 -2
                      \tilde l^2}{(|x-y|^2+\tilde
                      l^2)^{5/2}}\right) \intd y \intd x\\
                  & = P(\Omega) -\frac{\tilde \lambda}{2|\log\tilde
                    \delta|} \int_{\Omega} \int_{\stcomp{ \Omega}}
                  \left(\frac{g_{\alpha\tilde \delta}(|x-y|)}{|x-y|^3}
                    - \frac{|x-y|^2 -2 (\alpha\tilde
                      l)^2}{(|x-y|^2+(\alpha \tilde
                      l)^2)^{5/2}}\right) \intd y \intd x,
		\end{split}
	\end{align}
	which is equivalent to first statement.
  For the second part, note that by decomposing
  \begin{align}
    \frac{|z|^2 - 2l^2}{(|z|^2+l^2)^{5/2}} = -\frac{\frac{4}{5}|z|^2 +
    2l^2}{(|z|^2+l^2)^{5/2}} + \frac{\frac{9}{5}|z|^2
    }{(|z|^2+l^2)^{5/2}} 
  	\end{align}
  	and subtracting
        $m \int_{\R^2} \frac{\frac{9}{5}|z|^2
        }{(|z|^2+l^2)^{5/2}}\intd z =\frac{6\pi m \lambda}{5
          |\log\delta| l} $ from the energy we obtain
  \begin{alignat}{3}
    & \quad F_{\lambda,\delta,l}(\Omega) - \frac{6\pi m \lambda}{5
      |\log\delta|
      l} && \notag\\
    & = P(\Omega) - \frac{\lambda}{2|\log\delta|}\Bigg(\int_\Omega
    \int_{\stcomp{\Omega}} \Bigg( \frac{g_\delta(|x-y|)}{|x-y|^3} &&+
    \frac{\frac{4}{5}|x-y|^2 + 2 l^2}{(|x-y|^2
      + l^2)^\frac{5}{2}} \Bigg) \intd y \intd x \\
    &&&+ \int_\Omega \int_{\Omega} \frac{\frac{9}{5}|x-y|^2}{(|x-y|^2
      + l^2)^\frac{5}{2}} \intd y \intd x \Bigg).\notag
	\end{alignat}
	A similar argument gives
	 \begin{alignat}{3}
           & \quad \alpha\left(F_{\lambda,\delta,l}\left(\widetilde
               \Omega \right) - \frac{6\pi \widetilde m \lambda}{5
               |\log\delta|
               l}  \right) && \notag \\
           & = P(\Omega) -
           \frac{\lambda}{2|\log\delta|}\Bigg(\int_\Omega
           \int_{\stcomp{\Omega}}
           \Bigg(\frac{g_{\alpha\delta}(|x-y|)}{|x-y|^3} && +
           \frac{\frac{4}{5}|x-y|^2 + 2 (\alpha l)^2}{(|x-y|^2
             + (\alpha l)^2)^\frac{5}{2}}\Bigg)\intd y \intd x \\
           &&& + \int_\Omega \int_{\Omega}
           \frac{\frac{9}{5}|x-y|^2}{(|x-y|^2 + (\alpha
             l)^2)^\frac{5}{2}} \intd y \intd x \Bigg),\notag
	\end{alignat}
	
	Therefore we have
	\begin{align}\label{modulus_of_continuity_rescaling}
		\begin{split}
                  & \quad F_{\lambda,\delta,l}(\Omega) - \frac{6 \pi m
                    \lambda}{5 |\log\delta| l} - \alpha
                  \left(F_{\lambda, \delta,l}(\widetilde \Omega) -
                    \frac{6\pi \widetilde m \lambda}{5|\log\delta| l}
                  \right)\\
                  & = \frac{\lambda}{2|\log \delta|} \int_{\Omega}
                  \int_{\stcomp{\Omega}} \left(
                    \frac{g_{\alpha\delta}(|x-y|)}{|x-y|^3}-
                    \frac{g_\delta(|x-y|)}{|x-y|^3} \right)
                  \intd y \intd x \\
                  & \quad \quad + \frac{\lambda}{2|\log \delta|}
                  \int_{\Omega} \int_{\stcomp{\Omega}} \left(
                    \frac{\frac{4}{5}|x-y|^2 + 2 \alpha^2l^2}{(|x-y|^2
                      +\alpha^2
                      l^2)^\frac{5}{2}}-\frac{\frac{4}{5}|x-y|^2 + 2
                      l^2}{(|x-y|^2 + l^2)^\frac{5}{2}} \right) \intd
                  y \intd x
                  \\
                  & \quad \quad + \frac{\lambda}{2|\log \delta|}
                  \int_{\Omega} \int_{\Omega} \left(
                    \frac{\frac{9}{5}|x-y|^2}{(|x-y|^2 +
                      \alpha^2l^2)^\frac{3}{2}
                    }-\frac{\frac{9}{5}|x-y|^2}{(|x-y|^2+l^2)^\frac{3}{2}}
                  \right)
                  \intd y \intd x \\
                  & > 0,
       \end{split}
	\end{align}		
	by inspection of the various monotonicities in $\alpha<1$.
	
	As we only require the rescaling for the critical
        $\Gamma$-limit $E_{1,0}$ after the proof of Theorem
        \ref{thm:critical_convergence}, we may as well use the
        characterization as the pointwise limit to see that
	\begin{align}
          E_{1,0}(\Omega)
          & = \lim_{\delta \to 0} |\log\delta|
            E_{1,\delta}(\Omega)  = \alpha
            \lim_{\delta \to 0} \left( |\log\delta| P\left
            (\widetilde \Omega \right) -
            \frac{1}{2}\int_{\widetilde
            \Omega}\int_{\stcomp{\widetilde \Omega}} 
            \frac{g_{\delta / \alpha}(|x-y|)}{|x-y|^3} 
            \intd y \intd x \right) \notag\\  
          & \overset{\delta = \alpha \tilde \delta}{=} \alpha
            \lim_{\tilde \delta \to 0} \left( |\log \tilde \delta|
            P\left (\widetilde \Omega \right) -
            \frac{1}{2}\int_{\widetilde
            \Omega}\int_{\stcomp{\widetilde \Omega}}\frac{g_{\tilde 
            \delta}(|x-y|)}{|x-y|^3} \intd y \intd x \right) - \alpha
            P\left (\widetilde \Omega \right)  \log
            \alpha  \\ 
          & = \alpha E_{1,0}\left(\widetilde \Omega \right) - \alpha
            P\left 
            (\widetilde \Omega \right) 
            \log \alpha. \notag
	\end{align}
	The modification due to $l<\infty$ transforms this into
        \eqref{F10lresc}.
\end{proof}

We can also control how much the energy increases if we have
$\alpha >1$.  However, we postpone the proof to Section
\ref{sec:subcritical} as it is a straightforward adaptation of the
proof of Proposition \ref{prop:a_priori_lower_bound}.  As we only need
this statement for $l=\infty$ we directly formulate it for
$E_{\lambda,\delta}$.

\begin{lemma}\label{lem:rescaling2}
  As before, consider $ \alpha \widetilde \Omega = \Omega$, but now
  let $\alpha >1$.  Then we have
	\begin{align}
          \alpha E_{\lambda, \delta}(\widetilde \Omega) -
          E_{\lambda,\delta}(\Omega) \leq \frac{\lambda
          \,\log\alpha}{|\log\delta|}P(\Omega). 
	\end{align}
\end{lemma}

\section{Existence of classical and generalized minimizers of
  $E_{\lambda,\delta}$ and
  $F_{\lambda,\delta,l}$}\label{sec:existence}

The main purpose of this section is to prove existence of generalized
minimizers, together with some control over how many components a
generalized minimizer has.  This is encoded in the function
$f_{\lambda,\delta,l} : \R^+ \to \R$ defined as
\begin{align}\label{def:f} f_{\lambda,\delta,l}(m) :=
  \inf_{\mathcal{A}_m} F_{\lambda,\delta,l} - \frac{ 6\pi
  \lambda}{5|\log\delta| l}m.
\end{align}
The crucial observation is that any component of a generalized
minimizer (that is not classical) must have at least a certain amount
of mass and thus can have at most a quantitatively controlled number
of components.

\begin{prop}\label{prop:existence_of_minimizers_intermediate_masses}
  Let $\lambda > 0$, $m >0$, $0<\delta < \frac{1}{2}$ and
  $0<l \leq \infty$. For
  $\mathcal I^+ := \{m>0 : f_{\lambda,\delta,l}(m)>0\}$ we have
  \begin{enumerate}[(i)]
  \item $\mathcal I^+= (0, m_0(\lambda,\delta,l))$ for some
    $0<m_0(\lambda,\delta,l) <\infty$.
  \item $f_{\lambda,\delta,l}(m_0(\lambda,\delta,l))=0$.
  \item $f_{\lambda,\delta,l}(m)<f_{\lambda,\delta,l}(\widetilde m)$
    for all $m>\widetilde m \geq m_0(\lambda,\delta,l)$.
  \end{enumerate}
  Furthermore, for all $0 < \delta < \frac{1}{2}$ generalized
  minimizers at mass $m$ exist.  They can have at most
  $\max\left(\left\lceil\frac{m}{m_0(\lambda,\delta,l)} \right\rceil
    -1,1\right)$ components and in particular they are classical for
  masses $0<m \leq 2m_0(\lambda, \delta,l)$.
\end{prop}

Before we can embark on proving existence of (generalized) minimizers,
we first need a number of observations: We first give a rough bound of
the perimeter in terms of the energy.  It is sufficient for the
purposes of existence of minimizers, but we will do better in
Proposition \ref{prop:a_priori_lower_bound} below for asymptotic
statements.  The second observation is continuity of
$f_{\lambda,\delta,l}$. Note that throughout this section all the
constants $\lambda > 0$, $m > 0$, $0 < \delta < \frac12$ and
$0 < l \leq \infty$ are fixed, unless stated otherwise.

\begin{lemma}\label{lem:very_rough_a_priori_bound}
  For $\Omega \in \mathcal A_m$, we have
	\begin{align}\label{perimeter_bound}
          F_{\lambda,\delta,l}(\Omega) \geq P(\Omega) -
          \left(\frac{1}{\delta} + \frac{2}{3 l}\right) \frac{\pi 
          \lambda}{ |\log \delta|}  m 
	\end{align}
	and, in particular, $F_{\lambda,\delta,l}(\Omega) < \infty$
        implies $P(\Omega) < \infty$.  It also holds that
	\begin{align}\label{positivity_small_masses}
          \inf_{\mathcal{A}_m} F_{\lambda,\delta,l}- \frac{6 \pi m
          \lambda}{5 l |\log\delta|} >0 \text{ for } 0 < m <
          C(\delta,\lambda, l), 
	\end{align}
	for some $C(\delta, \lambda, l)>0$.
\end{lemma}

\begin{proof}
  The first statement is a straightforward consequence of the
  observation 
	\begin{align}
		\begin{split}
                  F_{\lambda,\delta,l}(\Omega) & \geq P(\Omega) -
                  \frac{\lambda}{2|\log\delta|}\int_{\Omega}
                  \int_{\stcomp{\Omega}} \left(
                    \frac{g_\delta(|x-y|)}{|x-y|^3} +
                    \frac{2l^2}{(|x-y|^2+l^2)^\frac{5}{2}} \right)
                  \intd y \intd x\\
                  & \geq P(\Omega) - \left(\frac{1}{\delta} +
                    \frac{2}{3 l}\right) \frac{\pi \lambda}{ |\log
                    \delta|} m.
      \end{split}
	\end{align}
	The second statement follows by combining the isoperimetric
        inequality with the above estimate to get
	\begin{align}
          F_{\lambda,\delta,l}(\Omega) -\frac{ 6\pi m
          \lambda}{5  l |\log\delta|} \geq \sqrt{4 \pi m} -
          \left(\frac{1}{\delta} + \frac{2}{l}\right)
          \frac{\pi \lambda}{ |\log \delta|} m >0 
	\end{align}
	for masses $0 < m < C(\delta, \lambda, l)$.
\end{proof}

\begin{lemma}\label{lem:continuity_of_infimum}
  The function $f_{\lambda,\delta,l}(m)$ is continuous.
\end{lemma}

\begin{proof}
  We only have to prove
  continuity of the first term
  $\tilde f(m):= \inf_{\mathcal{A}_m} F_{\lambda,\delta,l}$ in the
  definition \eqref{def:f}.  Let $m, \widetilde m >0$ and let $\eps >0$.
  By Lemma \ref{lem:very_rough_a_priori_bound} we know
  $\tilde f(m) \in \R$.  Consequently, there exists a set of finite
  perimeter $\Omega$ with $|\Omega|= m$ such that
	\begin{align}
		F_{\lambda,\delta,l}(\Omega) - \tilde f(m)< \eps.
	\end{align}

	For $\alpha = \sqrt{\frac{m}{\widetilde m}}$ the rescaling
        $\widetilde \Omega := \alpha^{-1} \Omega$ satisfies
        $|\widetilde \Omega| = \widetilde m$. Using it as a competitor for
        $\tilde f( \widetilde m)$ then gives
     \begin{align}
     	\begin{split}
          \tilde f(\widetilde m) - \tilde f(m) & \leq F_{\lambda,
            \delta,l}(\widetilde \Omega) - \tilde f(m) \\
          & \leq F_{\lambda, \delta,l}(\widetilde \Omega) - \tilde f(m) +
          \alpha^{-1}\left(-F_{\lambda,\delta,l}(\Omega) + \tilde
            f(m) +
            \eps\right)\\
          & = F_{\lambda, \delta,l}(\widetilde \Omega) - \alpha^{-1}
          F_{\lambda,\delta,l}(\Omega) + (\alpha^{-1} -1) \tilde f(m)
          + \alpha^{-1} \eps.
       \end{split}
	\end{align}
	A similar calculation as for the estimate
        \eqref{modulus_of_continuity_rescaling} gives
	\begin{align}\label{one-sided_modulus}
          \tilde f(\widetilde m) - \tilde f (m)  \leq
          \frac{\pi\lambda m}{\alpha |\log \delta|} h(\alpha)  +
          (\alpha^{-1} -1)  \tilde f(m) + \alpha^{-1} \eps,
	\end{align}
	where
	\begin{align}
          h(\alpha):=\int_{0}^\infty \left\{
          \left|\frac{g_{\alpha\delta}(r)}{r^2} - 
          \frac{g_\delta(r)}{r^2} \right| + r\left|\frac{r^2
          - 2\alpha^2 l^2}{(r^2 + \alpha^2l^2)^\frac{5}{2} }-\frac{r^2
          - 2l^2}{(r^2+l^2)^\frac{5}{2}}  \right| \right\}  \intd r  ,
	\end{align}
	satisfying $g(\alpha) \to 0$ as $\alpha \to 1$.
	
	Taking the limit $\eps \to 0$ and inserting
        $\alpha= \sqrt{\frac{m}{\widetilde m}}$, we see that
	\begin{align}
          \tilde f(\widetilde m) - \tilde f (m)  \leq \frac{\pi\lambda
          \sqrt{m \widetilde m}}{|\log \delta|} \,
          h\left(\sqrt{\frac{m}{\widetilde m}} \, \right)  +
          \left(\sqrt{\frac{\widetilde m}{m}}  -1\right)
          \tilde f(m).
	\end{align}
	As $\tilde f $ is locally bounded from above by testing with
        disks, this estimate provides a locally uniform modulus of
        continuity after symmetrizing the expression in $m$ and
        $\widetilde m$.
\end{proof}

We are now in a position to prove existence of minimizers as long as
the infimal energy is non-negative.

\begin{lemma}\label{lem:existence_minimizers_small_mass}
  Let $m > 0$ be such that $f_{\lambda, \delta,l}(m) > 0$. Then every
  minimizing sequence over $\mathcal A_m$ is compact in $L^1$ after
  translation.  In particular, minimizers exist, and any generalized
  minimizer must be classical.
\end{lemma}

\begin{proof}
  The basic strategy is applying the concentration-compactness
  principle due to Lions \cite{lions84a}, see also Struwe
  \cite[Section 4.3]{struwe}.  We have to deal with three cases:
  compactness, vanishing and splitting.  Let $\Omega_n$ be a
  minimizing sequence. By approximation, we may suppose that
  $\Omega_n$ are smooth open sets, see \cite[Theorem 13.8]{maggi}.
  Lemma \ref{lem:very_rough_a_priori_bound} implies that
	\begin{align}\label{qualitative_bound_perimeter}
		M := \limsup_{n \to \infty} P(\Omega_n) < \infty.
	\end{align}
	
	\textit{Step 1: Compactness.}\vspace{1mm}\\
	In this case we know the following: Up to extracting a
        subsequence and translating the sets, for every $\eps>0$ there
        exists $R>0$ such that
        $| \Omega_n \cap \ball{0}{R}| \geq m- \eps$ for all $n\in \N$.
        Therefore, the sequence of measures $\chi_{\Omega_n} \Leb^2$
        is tight. Together with the bound on the perimeter and the
        corresponding compact embedding theorem for BV-functions
        \cite[Corollary 3.49]{ambrosio} this implies that there exists
        a subsequence (not relabeled) and a set
        $\Omega \in \mathcal A_m$ such that
        $|\Omega_n \Delta \Omega| \to 0$.  Furthermore, the perimeter
        is lower semi-continuous with respect to this topology.
        
        As a result, in order to see that $\Omega$ is a
        minimizer of $F_{\lambda, \delta,l}$ we only have to prove
        that the quadratic form
	\begin{align}
          V(f) := \int_{\R^2} \int_{\R^2} |f(x+z) - f(x) |^2
          \left(\frac{g_\delta(|z|)}{|z|^3} -
          \frac{|z|^2-2l^2}{(|z|^2+l^2)^\frac{5}{2}} \right)\intd z
          \intd x 
	\end{align}
	is continuous on the space $L^1 \cap L^\infty$ equipped with
        the $L^1$-topology.  Making use of the inequality
        $|f(x+z) - f(x) |^2 \leq 2\left( |f(x+z)|^2 + |f(x) |^2
          \right)$ we indeed see that
	\begin{align}\label{continuity_of_nonlocal_term}
          V(f) \leq C(\delta,l) \|f\|^2_{L^2} \leq
          C(\delta,l)  \|f\|_{L^\infty}\|f\|_{L^1} 
	\end{align}
	for some $C(\delta,l)>0$.
	
	\vspace{1mm}
        \textit{Step 2: Vanishing.}\vspace{1mm}\\
	In this case, we have for all $R>0$ that
	\begin{align}\label{vanishing}
          \lim_{n\to \infty} \left(\sup_{y\in \R^2} | \Omega_n \cap
          \ball{y}{R}|\right) = 0. 
	\end{align}
	By Lemma \ref{lem:very_rough_a_priori_bound} we have
	\begin{align}
          F_{\lambda,\delta,l}(\Omega_n) \geq P(\Omega_n) -
          \left(\frac{1}{\delta} + \frac{2}{3 l}\right)\frac{\pi
          \lambda}{|\log \delta|} m. 
	\end{align}	
	To argue that the perimeter is large, we decompose the plane
        into the squares $Q_k := k + [0,1)\times[0,1)$ for
        $k\in \Z^2$.  The relative isoperimetric inequality implies
        that there exists a constant $c>0$ such that
	\begin{align}
          P_{Q_k}(\Omega_n) \geq c \min\left(|\Omega_n \cap
          Q_k|^{\frac{1}{2}}, |\stcomp{\Omega_n} \cap
          Q_k|^{\frac{1}{2}} \right) 
	\end{align}
	for all $k\in \Z^2$ and $n\in \N$, where $P_{Q_k}(\Omega_n)$
        denotes the perimeter of $\Omega_n$ relative to $Q_k$.  The
        vanishing property \eqref{vanishing} implies that for large
        $n \in \N$ we have $|\Omega_n \cap Q_k| \leq \eps$ for all
        $k \in \Z^2$, which implies
	\begin{align}
		P_{Q_k}(\Omega_n) \geq c |\Omega_n \cap
          Q_k|^{\frac{1}{2}} \geq c |\Omega_n\cap Q_k|
          \eps^{-\frac{1}{2}}. 
	\end{align}
	
	Consequently, from the bound
        \eqref{qualitative_bound_perimeter} we get
	\begin{align}
          M \geq P(\Omega_n) \geq \sum_{k \in \Z^2} P_{Q_k}(\Omega_n) \geq
          c m \eps^{-\frac{1}{2}} > M
	\end{align}
	for $\eps$ small enough, which is a contradiction.
 
 \vspace{1mm}
        \textit{Step 3: Splitting.}\vspace{1mm}\\
        In this case, there exists
        $0 < \theta < 1$ such that for any $\eps>0$ there exists $R>0$
        and a sequence $x_n \in \R^2$ with the following property: For
        any $\widetilde R > R$ we have
	\begin{align}\label{almost_all_mass}
          \limsup_{n\to \infty} \left( \Big| |\Omega_n \cap \ball{x_n}{R}| -
          \theta m \Big| + \left| |\Omega_n \cap
          B^{\mathsf{c}}_{\widetilde R}(x_n)| - (1-\theta) m  \right| \right) \leq
          \eps. 
	\end{align}
	
	We choose
	\begin{align}\label{distance_between_components}
          \widetilde R > R + 2 \widetilde M + 1,
	\end{align}
	where
        $\widetilde M := \frac{1}{2}\left( \inf_{\mathcal{A}_m}
          F_{\lambda,\delta,l} + 1+ \left(\frac{1}{\delta} +
            \frac{2}{3 l}\right) \frac{\pi \lambda}{ |\log \delta|} m
        \right) $ will be shown to bound the diameter of each
        connected component.  Let $\Omega_n^{(1)}$ be the union of all
        connected components of $\Omega_n$ that intersect
        $\ball{x_n}{R}$ and let $\Omega_n^{(2)}$ be the union of all
        connected components intersecting
        $B^{\mathsf{c}}_{\widetilde R}(x_n)$, see Figure
        \ref{fig:splitting} for a sketch. Let the remainder be
        $\Omega_n^{(3)} := \Omega_n \setminus ( \Omega_n^{(1)} \cup
        \Omega_n^{(2)})$, and observe that
        $\limsup_{n\to \infty} |\Omega_n^{(3)}| \leq \eps$ as a result
        of estimate \eqref{almost_all_mass}.

         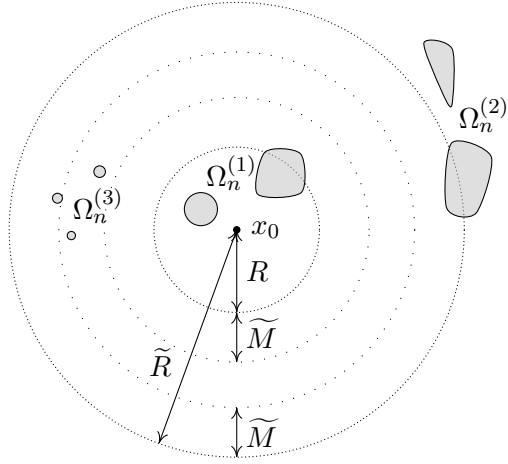
\begin{figure}
         	\centering
         	\begin{tikzpicture}[scale=1.1]
         		\node[circle,fill=black, inner sep=1pt,label=right:{$x_0$}] at (0,0) {};
         		\draw[densely dotted] (0,0) circle (1);
         		\draw[densely dotted] (0,0) circle (2.75);
         		\draw[loosely dotted] (0,0) circle (1.6);
         		\draw[loosely dotted] (0,0) circle (2.15);
         		\draw[{<[length=1.5mm,width=1.5mm]}-{>[length=1.5mm,width=1.5mm]}] (0,0) -- (250:2.75);
         		\node at (240:1.85) {$\widetilde R$};
         		\draw[{<[length=1.5mm,width=1.5mm]}-{>[length=1.5mm,width=1.5mm]}] (0,0) -- (270:1);
         		\node at (0.25,-.5) {$R$};
         		\draw[{<[length=1.5mm,width=1.5mm]}-{>[length=1.5mm,width=1.5mm]}] (270:1) --(270:1.6);
         		\node at (0.3,-1.25){$\widetilde M$};
         		\draw[{<[length=1.5mm,width=1.5mm]}-{>[length=1.5mm,width=1.5mm]}] (270:2.15) --(270:2.75);
         		\node at (0.3,-2.45){$\widetilde M$};
         		
				\fill[opacity=.5,color=lightgray] plot [smooth cycle] coordinates {(60:.5)(30:.9)(50:1.2)(70:1)};
         		\draw plot [smooth cycle] coordinates {(60:.5)(30:.9)(50:1.2)(70:1)};
         		
         		\fill[opacity=.5,color=lightgray] (150:.5) circle (.2);
         		\draw (150:.5) circle (.2);
         		
         		\fill[opacity=.5,color=lightgray] (182:2) circle (.05);
         		\draw  (182:2) circle (.05);
         		
         		\fill[opacity=.5,color=lightgray]  (157:1.8) circle (.07);
         		\draw  (157:1.8) circle (.07);
         		
         		\fill[opacity=.5,color=lightgray]  (170:2.2) circle (.06);
         		\draw  (170:2.2) circle (.06);
         		
         		\fill[opacity=.5,color=lightgray] plot [smooth cycle] coordinates {(6:2.55)(22:2.8)(16:3.2)(4:2.9)};
         		\draw plot [smooth cycle] coordinates {(6:2.55)(22:2.8)(16:3.2)(4:2.9)};
         		
         		\fill[opacity=.5,color=lightgray] plot [smooth cycle] coordinates {(35:3)(45:3.2)(40:3.4)(30:3)};
         		\draw plot [smooth cycle] coordinates {(35:3)(45:3.2)(40:3.4)(30:3)};
         		
         		\node at (95:.7){$\Omega_n^{(1)}$};
         		\node at (170:1.7) {$\Omega_n^{(3)}$};
         		\node at (25:3.3) {$\Omega_n^{(2)}$};
         	\end{tikzpicture}
         	\caption{\label{fig:splitting} Sketch of the
                  decomposition into the compact piece
                  $\Omega_n^{(1)}$, which contains all connected
                  components intersecting $\ball{x_0}{R}$, the piece
                  $\Omega_n^{(2)}$ consisting of all connected
                  components intersecting $\stcomp{B}_{\widetilde R}(x_0)$
                  and drifting off to infinity, and the vanishing
                  remainder $\Omega_n^{(3)}$.}
         \end{figure}
        
        As we chose a minimizing sequence consisting of smooth
          sets, it is easy to see that any connected component
        $\widetilde \Omega_n$ of $\Omega_n$ satisfies
        $\sup_{x,y\in \widetilde \Omega_n} |x-y| \leq \frac{1}{2}
        P(\widetilde \Omega_n) $ by noting that the supremum is
        achieved on $\partial \widetilde \Omega_n$. 
        Consequently, by Lemma \ref{lem:very_rough_a_priori_bound} we get
        \begin{align*}
        	\begin{split}
                  \sup_{x,y\in \widetilde \Omega_n} |x-y| \leq \frac{1}{2}
                  P(\widetilde \Omega_n) \leq \frac{1}{2}P(\Omega_n) \leq
                  \frac{1}{2} \left( F_{\lambda,\delta,l}(\Omega_n) +
                    \left(\frac{1}{\delta} + \frac{2}{3 l}\right)
                    \frac{\pi \lambda}{ |\log \delta|} m \right) \leq
                  \widetilde M
        	\end{split}
        \end{align*}
        for $n$ sufficiently large. Therefore, we get
        $ \Omega_n^{(1)} \subset \ball{x_n}{R + \widetilde M} $ and
        $ \Omega_n^{(2)} \subset B^{\mathsf{c}}_{\widetilde R - \widetilde
          M}(x_n)$. Thus the relation
        \eqref{distance_between_components} implies
        \begin{align}
          \inf_{x\in \Omega_n^{(1)},y \in \Omega_n^{(2)}} |x-y| \geq
          \inf_{x\in \Omega_n^{(1)},y \in \Omega_n^{(2)}} \big | |y| -
          |x| \big | \geq \widetilde R  - R  - 2\widetilde M > 1. 
        \end{align}
        In particular, $ \Omega_n^{(1)}$ and $ \Omega_n^{(2)}$ are
        disjoint.

        We now exploit the fact that the nonlocal term is repulsive at
        sufficiently large distances.  The representation
        \begin{align}
        	\begin{split}
                  V(\chi_{\Omega}) & = 2 \int_{\Omega}
                  \int_{\stcomp{\Omega}}
                  \left(\frac{g_\delta(|x-y|)}{|x-y|^3} -
                    \frac{|x-y|^2-2l^2}{(|x-y|^2+l^2)^\frac{5}{2}}
                  \right)\intd y \intd x
   	      \end{split}
        \end{align}
        and the observation
        \begin{align}\label{set_combinatorics}
         \begin{split}
          (\Omega_1 \cup \Omega_2) \times \stcomp{\left(\Omega_1 \cup
          \Omega_2\right)}
          &  =(\Omega_1 \cup \Omega_2) \times \left(\stcomp{\Omega_1}
            \cap \stcomp{\Omega_2}\right) \\ 
          & = \left(\Omega_1 \times \stcomp{\Omega_1} \setminus
            \Omega_1 \times \Omega_2\right) \cup \left(\Omega_2 \times
            \stcomp{\Omega_2} \setminus \Omega_2 \times
            \Omega_1\right) 
         \end{split}
        \end{align}
        for two disjoint sets $\Omega_1, \Omega_2 \subset \R^2$
        allow us to compute
	\begin{align}\label{nonlocal_term_repulsive}
		\begin{split}
                  & \quad   V(\chi_{\Omega_n^{(1)}}) +
                  V(\chi_{\Omega_n^{(2)}}) -  V(\chi_{\Omega_n^{(1)}}
                  + \chi_{\Omega_n^{(2)}}) \\ 
                  & = 4 \int_{\Omega_n^{(1)}} \int_{\Omega_n^{(2)}}
                  \left(\frac{g_\delta(|x-y|)}{|x-y|^3} -
                    \frac{|x-y|^2-2l^2}{(|x-y|^2+l^2)^\frac{5}{2}}
                  \right) \intd y \intd x.
   	     \end{split}
	\end{align}
	It is straightforward to see that the kernel in
          \eqref{nonlocal_term_repulsive} is positive for
        $|x-y| \geq \delta$, so that we get
	\begin{align}\label{repulsive}
          V(\chi_{\Omega_n^{(1)}}) + V(\chi_{\Omega_n^{(2)}}) -
          V(\chi_{\Omega_n^{(1)}} + \chi_{\Omega_n^{(2)}}) \geq 0 
	\end{align}
	as a result of
        $\inf_{x\in \Omega_n^{(1)},y \in \Omega_n^{(2)}} |x-y| \geq
        1>\delta$.  Choosing
        $\Omega_1 = \Omega_n^{(1)} \cup\Omega_n^{(2)}$ and
        $\Omega_2 = \Omega_n^{(3)}$ in equation
        \eqref{set_combinatorics} we see
	\begin{align}
		\begin{split}
                  & \quad V(\chi_{\Omega_n^{(1)}} +
                  \chi_{\Omega_n^{(2)}}) - V(\chi_{\Omega_n^{(1)}} +
                  \chi_{\Omega_n^{(2)}} +
                  \chi_{\Omega_n^{(3)}}) \\
                  & = 2 \int_{\Omega_n^{(1)} \cup\Omega_n^{(2)}}
                  \int_{\Omega_n^{(3)}} K_{\delta,l} \intd y \intd x -
                  \int_{\Omega_n^{(3)}}\int_{\stcomp{\Omega_n^{(3)}}}K_{\delta,l}
                  \intd y \intd x \\ 
                  & \geq - C_{\delta,l} |\Omega_n^{(3)}|
       \end{split}
	\end{align}
	for some $C_{\delta,l}>0$.
	
	Consequently, we obtain
	\begin{align}\label{splitting_in_energy}
          \inf_{\mathcal{A}_m} F_{\lambda,\delta,l} = \liminf_{n\to
          \infty} F_{\lambda,\delta,l}(\Omega_n) \geq \liminf_{n\to
          \infty} \left( F_{\lambda,\delta,l}(\Omega_n^{(1)}) +
          F_{\lambda,\delta,l}(\Omega_n^{(2)}) + P( \Omega_n^{(3)}) -
          C_{\delta,l} |\Omega_n^{(3)}|\right).  
	\end{align}
	Postprocessing this inequality by subtracting
        $\frac{6\pi \lambda}{5 |\log\delta| l}m$, we see 
	\begin{align}
		\begin{split}
                  f_{\lambda,\delta,l}(m) \geq \liminf_{n\to \infty}
                  \Big( F_{\lambda,\delta,l}(\Omega_n^{(1)}) -
                  \frac{6\pi \lambda}{5 |\log\delta|
                    l}|\Omega_n^{(1)}| +
                  F_{\lambda,\delta,l}(\Omega_n^{(2)} ) &
                  -\frac{6\pi \lambda}{5 |\log\delta|
                    l}|\Omega_n^{(2)}|\\ 
                  & + P( \Omega_n^{(3)}) -C_{\lambda,\delta,l}
                  |\Omega_n^{(3)}| \Big)\\ 
		\end{split}
	\end{align}
	for some $C_{\lambda,\delta,l}>0$.  Applying the isoperimetric
        inequality to $\Omega_n^{(3)}$, we see
	\begin{align}
          P( \Omega_n^{(3)}) -C_{\lambda,\delta,l} |\Omega_n^{(3)}|
          \geq 2\pi^\frac{1}{2}|\Omega_n^{(3)}|^\frac{1}{2}
          -C_{\lambda,\delta,l} |\Omega_n^{(3)}|  \geq 0 
	\end{align}
	provided $|\Omega_n^{(3)}| \leq \eps$ is small enough.
	As a result, we obtain
	\begin{align}\label{energy_splitting}
		\begin{split}
                  f_{\lambda,\delta,l}(m) \geq \liminf_{n\to \infty}
                  \Big( F_{\lambda,\delta,l}(\Omega_n^{(1)}) -
                  \frac{6\pi \lambda}{5 |\log\delta|
                    l}|\Omega_n^{(1)}| +
                  F_{\lambda,\delta,l}(\Omega_n^{(2)} ) & -\frac{6\pi
                    \lambda}{5 |\log\delta| l}|\Omega_n^{(2)}| \Big).
		\end{split}
	\end{align}	
 	This representation allows us to apply Lemma
        \ref{lem:rescaling} with
        $\alpha_{1,n} := \sqrt{\frac{|\Omega_n^{(1)}|}{m}}$ and
        $\Omega_n^{(1)}$, as well as
        $\alpha_{2,n} := \sqrt{\frac{|\Omega_n^{(2)}|}{m}}$ and
        $\Omega_n^{(2)}$.  The two resulting competitors for
        $f_{\lambda,\delta,l}(m)$ give us the estimate
	\begin{align}\label{postprocessing_1}
          f_{\lambda,\delta,l}(m) \geq \left(\alpha_{1,n}+
          \alpha_{2,n}\right)f_{\lambda,\delta,l}(m). 
	\end{align}
	
	We can now go to the limit $n \to \infty$ along some
        subsequence such that we have
        $\alpha_{1,n} \to \alpha_1(\eps)$ and
        $\alpha_{2,n} \to \alpha_2(\eps)$ with
	\begin{align}
          \left| \alpha_1^2(\eps) - \theta \right| + \left|
          \alpha^2_2(\eps) - (1-\theta) \right| \leq \frac{\eps}{m} 
	\end{align}
	due to the estimate \eqref{almost_all_mass}.
	Then estimate \eqref{postprocessing_1} turns into
	\begin{align}\label{break_for_generalized_minimizers}
          f_{\lambda,\delta,l}(m) \geq \left(\alpha_1(\eps)+
          \alpha_2(\eps)\right)f_{\lambda,\delta,l}(m). 
	\end{align}
	Therefore, taking the limit $\eps \to 0$ we see that
	\begin{align}\label{postprocessing_2}
          f_{\lambda,\delta,l}(m) \geq \left(\theta^\frac{1}{2} +
          (1-\theta)^\frac{1}{2} \right) f_{\lambda,\delta,l}(m), 
	\end{align}
	which in view of the positivity of $f_{\lambda,\delta,l}(m)$
        implies $\theta =0$ or $\theta=1$.  However, this is a
        contradiction to $0<\theta <1$, see the beginning of Step 3.
        This concludes the proof.
\end{proof}

We finally turn to prove existence of generalized minimizers:

\begin{proof}[Proof of Proposition
  \ref{prop:existence_of_minimizers_intermediate_masses}] 

  As in this proof $\lambda, \delta$ and $l$ are fixed, we drop them
  from the notation in $f_{\lambda,\delta,l}$, $F_{\lambda,\delta,l}$
  and $m_0(\lambda,\delta,l)$.  According to equation
  \eqref{positivity_small_masses} the set
  $\mathcal I^+ =\{m>0: f(m) >0\}$ contains a non-empty open
  interval with lower endpoint zero. Let $(0,m_0)$ with
  $m_0\in (0,\infty]$ be the largest such interval. Observe that
    by Lemma \ref{lem:existence_minimizers_small_mass} classical
    minimizers exist for all $m \in (0, m_0)$, so in the proof of
    existence of generalized minimizers we only need to consider the
    case $m \geq m_0$.

    \vspace{1mm}
    \textit{Step 1: We have $m_0< \infty$.}\vspace{1mm} \\
    We argue by providing an upper bound for the minimal energy, using
    disks as test configurations. Letting $\Omega := \ball{0}{1}$ and
    $\widetilde \Omega := \ball{0}{\alpha^{-1}}$, we only have to show
    that we can improve estimate
    \eqref{modulus_of_continuity_rescaling} to a lower bound that
    blows up in the limit $\alpha \to 0$.  Indeed, with $m = |\Omega|$
    and $\widetilde m = |\widetilde \Omega|$ we have
  	\begin{align}
		\begin{split}
                  & \quad F(\Omega) - \frac{6 \pi m \lambda}{5
                    |\log\delta| l} - \alpha \left(F(\widetilde \Omega) -
                    \frac{6\pi \widetilde m \lambda}{5|\log\delta| l}
                  \right)\\
                  & \geq \frac{\lambda}{2|\log \delta|} \int_{\Omega}
                  \int_{\stcomp{\Omega}} \left(
                    \frac{g_{\alpha\delta}(|x-y|)}{|x-y|^3}
                    -\frac{g_\delta(|x-y|)}{|x-y|^3} \right)
                  \intd y \intd x \\
                  & =\frac{\lambda}{2|\log \delta|} \int_{\ball{0}{1}}
                  \int_{\stcomp{B}_1(0)} \frac{\chi(\alpha\delta <
                    |x-y| \leq \delta)}{|x-y|^3} \intd y \intd x,
       \end{split} 
	\end{align}		
  	the right-hand side of which converges to
  	\begin{align}
          \frac{\lambda}{2|\log\delta|} \int_{\ball{0}{1}}
          \int_{\stcomp{B}_1(0)} \frac{\chi(|x-y| \leq
          \delta)}{|x-y|^3} \intd y \intd x = \infty. 
  	\end{align}

        \textit{Step 2: The functional $F$ has a minimizer at mass
          $m_0$ and all generalized minimizers are
          classical.}\vspace{1mm}\\  
        We may re-use large parts of the proof of Lemma
        \ref{lem:existence_minimizers_small_mass}.  The compactness
        and vanishing cases work exactly the same. We only have to
        rule out splitting, for which we follow the previous proof up
        to estimate \eqref{energy_splitting} and apply similar
        arguments as for estimates \eqref{postprocessing_1} and
        \eqref{postprocessing_2} to obtain
	\begin{align}
		\begin{split}
                 f(m_0) & \geq f(\theta m_0) + f((1-\theta) m_0)
		\end{split}
	\end{align}
	for some $0<\theta<1$.  By the choice of $m_0$ and Lemma
        \ref{lem:continuity_of_infimum} we have $f(m_0)=0$,
        $f(\theta m_0)>0$ and $ f((1-\theta) m_0) > 0$.  The obvious
        contradiction rules out the case of splitting, and concludes
        the proof of this step.

        \vspace{1mm} \textit{Step 3: Statements 1--3 of Proposition
          \ref{prop:existence_of_minimizers_intermediate_masses} are
          true.}\vspace{1mm}\\ 
        Let $\Omega$ be a minimizer of $F$ over $\mathcal A_{m_0}$.
        For $\alpha<1$, we test the infimum of $f(\widetilde m)$, where
        $\widetilde m := \alpha^{-2}m_0$, with the rescaling
        $\alpha^{-1} \Omega$ and apply the strict inequality of Lemma
        \ref{lem:rescaling} to get
  \begin{align}\label{observation}
    f(\widetilde m) < \alpha f(\widetilde m) < f( m_0) = 0
  \end{align}
  for all $\widetilde m > m_0$.  This clearly implies
  $\mathcal I^+ = (0,m_0)$.  A similar argument shows that $f$ is
  monotone decreasing on $(m_0,\infty)$.  All other listed statements
  have already been proven.
  
  \vspace{1mm}
  \textit{Step 4: Existence of minimizers for masses $m_0 < m \leq
      2 m_0$.}\vspace{1mm} \\
  We yet again re-use the proof of Lemma
  \ref{lem:existence_minimizers_small_mass}, which allows us to rule
  out vanishing, and of course implies existence of a minimizer in the
  compactness case.  However, we will not be able to rule out
  splitting, and instead provide a lower bound for the masses that may
  split off.
 
  As in Step 2 of this proof, we see that
  \begin{align}\label{splits}
    \begin{split}
      0> f(m) & \geq f(\theta m) + f((1-\theta)m)
    \end{split}
  \end{align}
  for some $0<\theta<1$.  If we had $\theta m \leq m_0$, then we would 
  have $f(\theta m)\geq 0$, which in turn would imply
  $ f((1-\theta)m) <0$.  Monotonicity of $f$ then would give
  \begin{align}
    f(m) & \geq  f((1-\theta)m)  >  f(m),
  \end{align}
  which is a clear contradiction.  Switching the roles of $\theta$ and
  $1-\theta$ in this argument we obtain
  \begin{align}\label{minimal_mass_1}
    \theta m >  m_0(\lambda,\delta,l),\\
    (1-\theta) m >  m_0(\lambda,\delta,l).\label{minimal_mass_2}
  \end{align}
	
  Consider two smooth competitors $\Omega_1$ for $f(\theta m)$ and
  $\Omega_2$ for $f((1-\theta)m)$.  We may use
  $\Omega := \Omega_1 \cup ( Re_1+ \Omega_2)$ for $R>0$ large enough
  such that $\Omega_1 \cap ( Re_1+ \Omega_2) = \emptyset$ as a
  competitor for $f(m)$.  As the nonlocal interaction between the two
  sets vanishes in the limit $R\to \infty$, we obtain
  $f(\theta m) + f((1-\theta) m ) \geq f(m)$.  Combining this with
  estimate \eqref{splits} we get
	\begin{align}\label{generalized_minimizers_exist}
		\begin{split}
                  f(m) & \geq f(\theta m) + f((1-\theta)m) \geq f(m),
		\end{split}
	\end{align}
        and the inequalities are, in fact, equalities.
        
        As for $m \leq 2 m_0$ any splitting would violate estimate
        \eqref{minimal_mass_1} or estimate \eqref{minimal_mass_2},
        classical minimizers exist and any generalized minimizer is
        classical.

 \vspace{1mm}
        \textit{Step 5: Existence of generalized minimizers for masses
          $m > 2 m_0$ and the bound for the number of their
          components}. \vspace{1mm}\\
        The insights from Step 4 allow us to set up an induction
        argument: For $k \in \N$ assume that generalized minimizers
        exist for all masses $m \leq (k+1) m_0$ and that they can have
        at most $k$ components.  Note that by the result in Step 4
        this assumption is true for $k = 1$. Let now
        $m \in ((k+1)m_0,(k+2)m_0]$.  Then equation
        \eqref{generalized_minimizers_exist} says that a generalized
        minimizer with mass $m$ is given by combining the generalized
        minimizers at masses $\theta m$ and $(1-\theta)m$ which exist
        because estimates \eqref{minimal_mass_1} and
        \eqref{minimal_mass_2} imply
	\begin{align}
		\theta m, (1-\theta)m < (k+1)m_0.
	\end{align}
	To see that the generalized minimizer has at most $k+1$
        components we note that the same proof as for estimates
        \eqref{minimal_mass_1} and \eqref{minimal_mass_2} implies that
        any component $\Omega_i$ must have $|\Omega_i| > m_0$.
        Therefore, any generalized minimizer at mass $m$ clearly has
        at most $k+1$ components.
	
	Furthermore, for masses $m > 2 m_0$ the condition
        $k m_0<m \leq (k+1)m_0$ implies $k < \frac{m}{m_0}$, which
        implies $k \leq \left\lceil\frac{m}{m_0} \right\rceil -1$.
        Taking into account the fact that minimizers must be classical
        for masses $m \leq m_0$ due to Lemma
        \ref{lem:existence_minimizers_small_mass} gives us the bound
    \begin{align}
    	\max\left(\left\lceil\frac{m}{m_0}
          \right\rceil -1,1\right)
    \end{align}
    on the number of components of generalized minimizers. 
\end{proof}

Next, we briefly discuss the regularity properties of generalized
minimizers of $F_{\lambda,\delta,l}$. Even though we only provide a
qualitative statement here, it is one of the crucial ingredients when
proving that minimizers are disks for $\lambda < 1$ and
$l=\infty$, as it will allow us to set up an iteration to enlarge the
scale at which the boundary of a minimizer is locally a graph with
respect to some coordinate direction.

\begin{prop}\label{prop:ELG_generalized_minimizers}
  Let $(\Omega_i)_{i\in \N}$ be a generalized minimizer of
  $F_{\lambda,\delta,l}$ over $\mathcal A_m$.  Then there exists a
  Lagrange multiplier $\mu \in \R$ such that for all $i \in \N$ with
  $ |\Omega_i| > 0$ the following holds:
  \begin{enumerate}[(i)]    
  \item After a possible redefinition on a set of zero Lebesgue
      measure, $\Omega_i$ is open and bounded.
    \item We have that $\Omega_i$ has a $C^{2,\alpha}$-boundary for
      all $\alpha \in (0, 1)$.
  \item For all $y \in \partial \Omega_i$ we have
    \begin{align}\label{ELG_generalized_minimizers}
      \kappa_i(y) + 2 v_{\delta,l,i}(y) - \mu = 0,
    \end{align}
    where $\kappa_i$ is the curvature of $\Omega_i$ (positive for
    convex sets) and $v_{\delta,l,i}$ is the potential of $\Omega_i$
    defined in equation \eqref{potential_definition}.
  \end{enumerate}
\end{prop}

\begin{proof}
  As the components of generalized minimizers clearly are classical
  minimizers, we may drop the index $i$ for the proof of regularity
  and the form of the Euler-Lagrange equation. We will only recall the
  dependence when proving that the Lagrange-multiplier is independent
  of $i$.
  
  In order to obtain $C^{1,\alpha}$-regularity for some
  $\alpha \in (0,1)$, one can directly apply \cite[Th\'eor\`eme 1.4.9
  or Th\'eor\`eme 5.1.3] {rigot00}, which yields $\alpha =
  \frac12$. One could alternatively prove that $\Omega$ is a
  $(\Lambda,r_0)$-minimizer of the perimeter functional according to
  \cite[Chapter 21]{maggi} for some $\Lambda, r_0>0$ and then apply
  the regularity theory presented therein to prove
  $C^{1,\alpha}$-regularity for all $\alpha \in (0, \frac12)$. The
  mass constraint can be dealt with by rescaling competitor sets to
  the appropriate mass and making use of Lemmas \ref{lem:rescaling}
  and \ref{lem:rescaling2}.
	
  The computation of the Euler-Lagrange equation for a single set
  $\Omega_i$:
  \begin{align}
    \kappa_i(y) + 2 v_i(y) - \mu_i = 0, \qquad y \in \partial \Omega_i,
  \end{align}
  with some $\mu_i \in \R$, is a standard exercise in view of
  Lipschitz continuity of $v_i$. For the latter, write
  \begin{align}
    \label{viLip}
    \begin{split}
      |v_i(x) - v_i(y)| & \leq {\lambda \over 2 |\log \delta|}
      \int_{\R^2} K_{\delta,l} (|z|) \left|
        \chi_{\Omega_i}(x + z) - \chi_{\Omega_i}(y + z) \right| \intd z \\
      & \leq {\lambda |x - y| \over 2 |\log \delta|} \int_{\R^2}
      K_{\delta,l} (|z|) \int_0^1 |\nabla
      \chi_{\Omega_i}(x t + (1 - t) y + z)| \intd t \intd z \\
      & = {\lambda P(\Omega_i) \| K_{\delta,l} \|_\infty \over 2 |\log
        \delta|} \, |x - y|,
    \end{split}
  \end{align}
  arguing by approximation.  To obtain higher regularity, we appeal to
  the regularity theory for the prescribed mean curvature equation,
  see for example \cite{giusti}.
        
  To see that the Lagrange multiplier does not depend on the
  component, let $i,j \in \N$ be such that
  $\Omega_i, \Omega_j \neq \emptyset$.  For $\eps \in \R$, dilating
  $\Omega_i$ with the factor $1+ \eps $ and $\Omega_j$ with the factor
  $1-\frac{m_i}{m_j}\eps $ gives a viable competitor
  $(\Omega_{\eps,i})_{i \in \N}$ for $(\Omega_i)_{i \in \N}$.  We
  therefore get
  \begin{align}
    0 & =\left. \frac{d}{d \eps } 
        E\left((\Omega_{\eps,i})_{i \in \N}\right) \right|_{\eps
        = 0} \notag \\ 
      & =  \int_{\partial \Omega_i} (x\cdot \nu_i)\left( \kappa_i(y) +
        2 v_i(y)\right) \intd \Hd^1(y) - \frac{m_i}{m_j}\int_{\partial
        \Omega_j} (x\cdot \nu_j) \left(\kappa_j(y) + 2 v_j(y) \right)\intd
        \Hd^1(y) \\ 
      & = 2 \mu_i m_i  - \frac{2 m_i}{m_j} \mu_j m_j, \notag
  \end{align}
  which immediately implies $\mu_i = \mu_j$.
\end{proof}

To conclude this section, we point out that in a large range of
parameters any minimizer has to be connected.  This is not clear
\emph{a priori} as the kernel $K_{\delta,l}$ is not necessarily
strictly positive.  In fact, if we have $\delta > \sqrt{2}l$, then it
is easy to see that $K_{\delta,l}(r) <0$ for $\sqrt{2}l < r < \delta$,
so that at certain distances the kernel is even attractive.

\begin{lemma}\label{lem:connectedness}
  If $\delta<\sqrt{2\pi m}$ in the case $l=\infty$,
  or $\delta < \sqrt{2} l$ in the case $l<\infty$, then the regular
  representatives constructed in Proposition
  \ref{prop:ELG_generalized_minimizers} of any minimizer $\Omega$ of
  $F_{\lambda,\delta,l}$ over $\mathcal A_m$ are connected.
\end{lemma}

\begin{proof}
  Let us assume that there exists disjoint, non-empty subsets
  $\Omega_{1}$ and $\Omega_{2}$ such that
  $\Omega =\Omega_{1} \cup \Omega_{2}$.  By using the competitors
  $\Omega_{1} \cup (Re_1+ \Omega_{2})$ for $R\to \infty$ we see that
	\begin{align}
          F_{\lambda,\delta,l} (\Omega_{1}) +
          F_{\lambda,\delta,l} (\Omega_{1}) \geq F_{1,\delta,l}
          \left(\Omega\right), 
	\end{align}
	which with the help of the expression in
        \eqref{nonlocal_term_repulsive} for the difference of the
        left- and right-hand side implies
	\begin{align}\label{minimizers_connected}
          0 \geq \frac{1}{|\log\delta|}\int_{\Omega_{1,n}} \int_{\Omega_{2,n}}
          K_{\delta,l}(|x-y|) \intd y \intd x.
	\end{align}
	
	Consider the case $l= \infty$ first.  Then
        \eqref{minimizers_connected} becomes
	\begin{align}
           \frac{1}{|\log\delta|}\int_{\Omega_{1,n}} \int_{\Omega_{2,n}}
          \frac{g_{\delta}(|x-y|)}{|x-y|^3} \intd y \intd x =0,
	\end{align}
	which implies $\Omega_{1,n} \subset \ball{x}{\delta}$ for all
        $x\in \Omega_{2,n}$ and vice versa.
        Therefore we have
	\begin{align}
          |\Omega| = |\Omega_{1}| + |\Omega_{2}| \leq
          2\pi \delta^2, 
	\end{align}
        which contradicts the assumptions of the lemma.
        
        Next, we deal with the case $l<\infty$. Then we have
	\begin{align}
          K_{\delta,l}(r) \geq
          \begin{cases}
            {2 l^2 - r^2 \over (r^2 + l^2)^{\frac52}} & r \leq \delta, \\
            {2 l^2 \over (r^2 + l^2)^{\frac52}} & r > \delta.
          \end{cases}
        \end{align}
        Consequently, for $\delta < \sqrt{2} l$ we get that
        $K_{\delta,l}(r) > 0$ for all $r > 0$ and, therefore,
        \eqref{minimizers_connected} yields $\Omega_{1} = \emptyset$
        or $\Omega_{2} = \emptyset$, a contradiction.
\end{proof}

\section{The subcritical case $\lambda<1$}\label{sec:subcritical} 
The main goal in this section is to prove that minimizers are disks
for $\lambda <1$ and all $\delta$ sufficiently small.  While in the
previous section rough estimates were acceptable, the main theme here
will be obtaining estimates that are as uniform as possible.  To this
end, it is crucial to realize that the nonlocal term localizes to
leading order in the sense that in the limit $\delta \to 0$ it
approaches $-\lambda P(\Omega)$.  Therefore, treating it as a volume
term as we did for example in Proposition
\ref{prop:ELG_generalized_minimizers} has no chance of being accurate.
Instead, one has to split the nonlocal contribution into the
leading-order local and a higher-order nonlocal part.

The first instance of this idea is contained in the following
lower bound for the energy in terms of the perimeter, adapted from
Kn\"upfer, Muratov and Nolte \cite{kmn:arma}.  Not only is it the
first step to obtain uniform regularity estimates and the lower bound
for the $\Gamma$-limit
$E_{\lambda,\delta} \stackrel{\Gamma}{\to} (1-\lambda)P$ , it will in
Section \ref{sec:critical} also be the starting point in proving the
$\Gamma$-$\liminf$ inequality of the $\Gamma$-convergence in the
critical case $\lambda=1$.

\begin{prop}\label{prop:a_priori_lower_bound}
  Let $\lambda > 0$, $0 < \delta < \frac12$, $m > 0$ and
  $\Omega \in \mathcal A_m$. If $P(\Omega) \leq \frac{\pi m}{\delta}$
  we have the bound
	\begin{align}
          E_{\lambda,\delta}(\Omega) \geq \left(1-\lambda  +
          \frac{\lambda}{|\log\delta|}\log\left(\frac{P(\Omega)}{e\,
          \pi m}  \right) \right) P(\Omega). 
	\end{align}
	If $P(\Omega) > \frac{\pi m}{\delta}$ we instead have
	\begin{align}
          E_{\lambda,\delta}(\Omega) \geq \left(1 -  \frac{
          \lambda}{|\log\delta|}\right)P(\Omega). 
	\end{align}
\end{prop}

\begin{proof}
  We derive an upper bound for the nonlocal term by splitting the
  inner integral in
	\begin{align}
          I :=\int_{\R^2} \int_{\R^2} |\chi_\Omega(x+z) -
          \chi_\Omega (x) 
          |^2 \frac{g_\delta(|z|)}{|z|^3} \intd z \intd x 
	\end{align}
	into two parts.
	
	\vspace{1mm}
        \textit{Step 1: Small scales.}
        \vspace{1mm}\\
	For some cut-off length $R>0$ to be determined later, we claim
        to have the estimate 
	\begin{align}
          \quad \int_{\R^2} \int_{\ball{0}{R}} |\chi_\Omega(x+z) -
          \chi_\Omega (x) |^2 \frac{g_\delta(|z|)}{|z|^3} \intd z
          \intd x \leq 4  P(\Omega) \int_0^{R}
          \frac{g_\delta(s)}{s} \intd s. 
	\end{align}	
	Note that the integral on the left hand side is continuous
        with respect to the $L^2$-topology due estimate
        \eqref{continuity_of_nonlocal_term}.  By approximating
        $\chi_\Omega$ with smooth functions
        $f_n \in C_c^\infty(\R^2;[0,1])$ such that
        $f_n \to \chi_\Omega$ in $L^2$, $\int f_n \intd x = m$ and
        $\int_{\R^2} | \nabla f_n| \intd x \to P(\Omega)$ it is
        sufficient to prove the analogous result for smooth functions.
        By \cite[estimate (3.8)]{kmn:arma} we get
	\begin{align}
	  \begin{split}
            \int_{\R^2} \int_{\ball{0}{R}} |f(x+z) - f (x) |
            \frac{g_\delta(|z|)}{|z|^3} \intd z \intd x
            & \leq \int_{\R^2} \int_{\ball{0}{R}} |\nabla f (x) \cdot
            z |\frac{g_\delta(|z|)}{|z|^3} \intd z \intd x \\ 
            & = \int_{\R^2} |\nabla f| \intd x
            \int_{\ball{0}{R}}\frac{|z_1|g_\delta(|z|)}{|z|^3} \intd
            z.
	  \end{split}
	\end{align}
	Using polar coordinates it is easy to see that
	\begin{align}
          \int_{\ball{0}{R}}\frac{|z_1|g_\delta(|z|)}{|z|^3} \intd z =
          4 \int_0^{R} \frac{g_\delta(s)}{s} \intd s, 
	\end{align}
	which concludes the proof of the claim after noticing
        $|\chi_\Omega(x+z) - \chi_\Omega(x)|^2 = |\chi_\Omega(x+z) -
        \chi_\Omega(x)|$.
	
	\vspace{1mm}
        \textit{Step 2: Large scales.}
        \vspace{1mm}\\
	In order to estimate the large scales, we make use of the
        representation \eqref{energy_as_nonlocal_perimeter} to see
	\begin{align}
          \int_{\R^2} \int_{B^{\mathsf{c}}_{R}(0)} |\chi_\Omega(x+z) -
          \chi_\Omega (x) |^2 \frac{g_\delta(|z|)}{|z|^3} \intd z
          \intd x \leq 4\pi m \int_{R}^\infty \frac{g_\delta(s)}{s^2}
          \intd s. 
	\end{align}
	
\textit{Step 3: Balancing the contributions.}\vspace{1mm}\\
	Formal minimization of the function
	\begin{align}
          F(R) := 4 P(\Omega) \int_0^{R} \frac{g_\delta(s)}{s} \intd s \,
          +  4\pi m \int_{R}^\infty \frac{g_\delta(s)}{s^2}
          \intd s
	\end{align}
	suggests to use 
	\begin{align}
		R := \frac{\pi m}{P(\Omega)}
	\end{align}
	as the cut-off radius.
	
	If $\frac{\pi m}{P(\Omega)}\geq \delta$, we obtain
	\begin{align}
          F\left(\frac{\pi m}{P(\Omega)}\right)
          & =  4\left(  |
            \log\delta|  +
            \log\left(
            \frac{\pi
            m}{P(\Omega)}
            \right)
            \right)
            P(\Omega) +  4
            P(\Omega). 
	\end{align}
	Combining this inequality with the local perimeter term in the
        energy gives the first desired estimate.
	
	If, on the other hand, $\frac{\pi m}{P(\Omega)} <
        \delta$, we instead choose
	\begin{align}
          F(\delta) & =  \frac{ 4\pi m}{\delta} \leq 4 P(\Omega),
	\end{align}		
	which gives the second desired inequality.
      \end{proof}

      The argument in the proof of Proposition
        \ref{prop:a_priori_lower_bound} also allows us to finally
      give an efficient proof of Lemma \ref{lem:rescaling2}.

\begin{proof}[Proof of Lemma \ref{lem:rescaling2}]
  Recalling the proof of Lemma \ref{lem:rescaling}, specifically the
  estimate \eqref{modulus_of_continuity_rescaling}, we have
	\begin{align}
	  \begin{split}
            D & := \alpha E_{\lambda,
              \delta}(\widetilde \Omega)
            -  E_{\lambda,\delta}(\Omega) \\
            & = \frac{\lambda}{4|\log \delta|} \int_{\R^2} \int_{\R^2}
            |\chi_\Omega(x+z) - \chi_\Omega (x) |^2
            \left(\frac{g_\delta(|z|)}{|z|^3} - \frac{g_{\alpha
                  \delta}(|z|)}{|z|^3} \right) \intd z \intd x,
	   \end{split}
	\end{align}
	where the kernel
        $\frac{g_\delta(|z|)}{|z|^3} - \frac{g_{\alpha
            \delta}(|z|)}{|z|^3} = \frac{\chi(\{\delta < |z|\leq
          \alpha \delta \})}{|z|^3}$ is positive due to $\alpha>1$.
        Choosing $R=\alpha \delta$ in Step 1 of the proof of
        Proposition \ref{prop:a_priori_lower_bound} gives
	\begin{align}
          D & \leq \frac{\lambda P(\Omega)}{|\log\delta|}
                   \int_\delta^{\alpha \delta} \frac{1}{s} \intd s
                   = \frac{\lambda \log
                   \alpha}{|\log\delta|}P(\Omega). \qedhere 
	\end{align}
\end{proof}

Throughout the rest of this section we assume that $0 < \lambda < 1$
and $0 < \delta < \frac12$.  Combining the result of Proposition
\ref{prop:a_priori_lower_bound}, the isoperimetric inequality and the
result of Lemma \ref{lem:existence_minimizers_small_mass} immediately
yields existence of minimizers of $E_{\lambda,\delta}$ over
$\mathcal A_\pi$ in the subcritical case for all $\delta$ sufficiently
small.
      \begin{cor}
        \label{cor:exist_subcr}
        There exists a universal $\sigma > 0$ such that if
          \begin{align}
            \label{eq:sigma1bis}
             \frac{\lambda}{ (1 - \lambda) |\log \delta|} \leq \sigma,   
          \end{align}
          then there exists a minimizer of $E_{\lambda,\delta}$ over
          $\mathcal A_\pi$.
      \end{cor}

As already mentioned, Proposition \ref{prop:a_priori_lower_bound}
allows us to compute the $\Gamma$-limit of $E_{\lambda,\delta}$ in the
$L^1$-topology.  Note that we do not need a compactness statement
since we will quantify the convergence of the minimizers to disks,see
Lemma \ref{lem:convergence_rates}, and we will in the end even see
that minimizers are disks for $\delta>0$.

\begin{prop}\label{prop:subcritical_gamma_convergence}
  Let $m > 0$. As $\delta \to 0$, the $L^1$-$\Gamma$-limit of the
  functionals $E_{\lambda,\delta}$ restricted to $\mathcal{A}_m$ is
  given by
	\begin{align}
          E_{\lambda,0}\left(\Omega\right) : = (1- \lambda) P(\Omega)
	\end{align}
	for $\Omega \in \mathcal{A}_m$.
\end{prop}
	
\begin{proof}
  Using the lower bound of Proposition
  \ref{prop:a_priori_lower_bound}, the $\Gamma$-$\liminf$-statement
  follows from lower semi-continuity of the perimeter: Let
  $\Omega_\delta \subset \R^2$ be such that
  $|\Omega_\delta\Delta \Omega| \to 0$ as $\delta \to 0$.  Then we get
  \begin{align}
    \liminf_{\delta \to 0} E_{\lambda,\delta}(\Omega_\delta)
    \geq \liminf_{\delta \to 0} (1-\lambda )P(\Omega_\delta)
    \geq P(\Omega), 
  \end{align}
  noting that the term involving $\log P$ is bounded below by the
  isoperimetric inequality.
	
  Turning towards the $\Gamma$-$\limsup$-statement, it is well-known
  that any set $\Omega \subset \R^2$ of finite perimeter can be
  approximated with smooth sets $\Omega_n \subset \R^2$ such that
  $|\Omega_n \Delta \Omega | \to 0$ and $P(\Omega_n) \to P(\Omega)$ as
  $n\to \infty$, i.e., smooth sets not satisfying the mass constraint
  are dense in energy with respect to the perimeter $P$.  In order to
  enforce the constraint, one approximates $\Omega\cap \ball{0}{R}$
  for $R>0$ large by smooth sets and uses \cite[Lemma 4]{figalli11} to
  correct the mass of the approximation.  Therefore it is sufficient
  to prove that
	\begin{align}
		\lim_{\delta \to 0} E_{\lambda,\delta} (\Omega) = P(\Omega)
	\end{align}
	for smooth sets $\Omega \subset \R^2$.
	
	To this end we make
        use of the boundary representation of the energy
        \eqref{energy_on_boundary} of Lemma
        \ref{lem:energy_on_boundary}.  For $x\in \partial \Omega$ we
        focus on the inner integral
	\begin{align}
          I := \int_{\partial \Omega} \nu(x) \cdot \nu(y)
          \Phi_{\delta}(|x-y|) \intd \Hd^1(y).
	\end{align}
	We parametrize $\partial \Omega$ about $x$ via
        $\gamma: (-c,c) \to \R^2$ such that $\gamma(0)=x$ and
        $|\dot \gamma|=1$ for some $c>0$.  Expanding the various
        expressions, see identity \eqref{Phi_explicit}, about $t=0$
        gives
	\begin{alignat}{3}
          \nu(\gamma(t)) & = \nu(\gamma(0)) + O(t),&&\\
          |\gamma(0)-\gamma(t)| & = | t \dot \gamma(0) + O(t^2)| &&=
          |t| + O(t^2),\\ \frac{1}{|\gamma(0)-\gamma(t)| } & =
          \frac{1}{|t|}\frac{1}{1+O(|t|)} &&= \frac{1}{|t|} + O(1)
	\end{alignat}
	and
	\begin{align}\label{logarithm}
          \frac{1}{\delta}\left( 1- \log \frac{|\gamma(0)-\gamma(t)|
          }{\delta} \right)& = \frac{1}{\delta}\left( 1- \log
                             \frac{|t| }{\delta}  + O(|t|)\right). 
	\end{align}
	As a result, we get that the leading order contribution of the
        inner integral is
	\begin{align}
          \int_{-c}^c \Phi_\delta(|t|) \intd t +
          O\left(\frac{1}{|\log\delta|} \right) = 2 |\log\delta| +
          O\left(1\right). 
	\end{align}
	Inserting this expression into the representation
        \eqref{energy_on_boundary} gives the desired statement.
\end{proof}

Our next result shows that the minimizers of $E_{\lambda,\delta}$ have
small isoperimetric deficit whenever $\delta$ is sufficiently small
depending only on $\lambda$.

\begin{lemma}\label{lem:nearly_iso}
  Any minimizer $\Omega$ of $E_{\lambda,\delta}$ over
  $\mathcal{A}_\pi$ satisfies
  \begin{align}
    \label{eq:isodef}
    P(\Omega) - P(B_1(0)) \leq \frac{C\lambda}{(1-\lambda)
    |\log\delta|}.
  \end{align}
\end{lemma}

\begin{proof}
  The lower bound of Proposition \ref{prop:a_priori_lower_bound}
  implies
  \begin{align}
    (1-\lambda) P(\Omega) \leq C \frac{\lambda}{|\log\delta|} +
    E_{\lambda,\delta}(\Omega) 
  \end{align}
  for a universal constant $C>0$ as the map $x \mapsto x \log x$ is
  bounded from below.  On the other hand we have
  \begin{align}
    E_{\lambda,\delta}\left(\ball{0}{1}\right) \leq (1-\lambda)
    P\left(\ball{0}{1}\right) + C\frac{\lambda}{|\log\delta|} 
  \end{align}
   by the construction in Proposition
  \ref{prop:subcritical_gamma_convergence} for $\ball{0}{1}$.  As a
  result, we obtain
	\begin{align}\label{isoperimetric_deficit_controlled}
          (1-\lambda) \left(P(\Omega) - P\left(\ball{x}{1} \right)
          \right)  \leq E_{\lambda,\delta}(\Omega) -
          E_{\lambda,\delta}\left(\ball{0}{1}\right) + C
          \frac{\lambda}{|\log\delta|} \leq   C
          \frac{\lambda}{|\log\delta|} 
	\end{align}
	by minimality of $\Omega$.
\end{proof}

We now turn to the heart of the matter: obtaining uniform regularity
estimates for the minimizers.  As is usually the case for these
problems, the first step is to prove uniform density estimates.  To
this end, we present two non-optimality criteria in the spirit of
Kn\"upfer and Muratov \cite[Lemma 4.2]{km:cpam14}.  The idea is to cut
away small appendages and fill small holes, see Figure
\ref{fig:non-optimality}.

\begin{figure}
	\centering
	\subcaptionbox{\label{fig:non-optimality_cutting}}{
		\centering
		\begin{tikzpicture}
			\draw plot [smooth cycle] coordinates {(0,0)(2,2)(4,.3)(5,.8)(5.4,.2)(4.7,-.8)(4,-.2)(1.6,-1.7)};
			\draw[densely dashed] (4,.3)--(4,-.2);
			\node at (2,0) {$F_1$};
			\node at (4.7,0) {$F_2$};
			\node at (2,2.3) {$F$};
		\end{tikzpicture}
	}
	\subcaptionbox*{}{
		\centering
		\begin{tikzpicture}[scale=.2]
			\fill[white] (-3,3)--(3,3)--(3,-3) --(-3,-3) --cycle;
		\end{tikzpicture}
	}
	\subcaptionbox{\label{fig:non-optimality_filling}}{
		\centering
		\begin{tikzpicture}
			\draw plot [smooth cycle] coordinates {(0,0)(2,2)(4,.3)(3,.5)(2.5,.2)(3.2,-.3)(4,-.2)(1.6,-1.7)};
			\draw[densely dashed] (4,.3)--(4,-.2);
			\node at (1.5,0) {$F$};
			\node at (3.5,0) {$F_2$};
			\node at (2,2.3) {$F_1$};
		\end{tikzpicture}
	}
	\caption{Sketch of the sets $F, F_1, F_2$ discussed in Lemma
          \ref{lem:non-optimality}.  The length of the dashed line is
          $\frac{1}{2}\Sigma$. a) If the common boundary of $F_1$ and
          $F_2$ is short, it is beneficial to cut away $F_2$ and
          dilate $F_1$ to account for the lost mass. b) Similarly, if
          the hole $F_2$ has little extra boundary, one can lower the
          energy by filling the hole and dilating.  }
	\label{fig:non-optimality}
\end{figure}
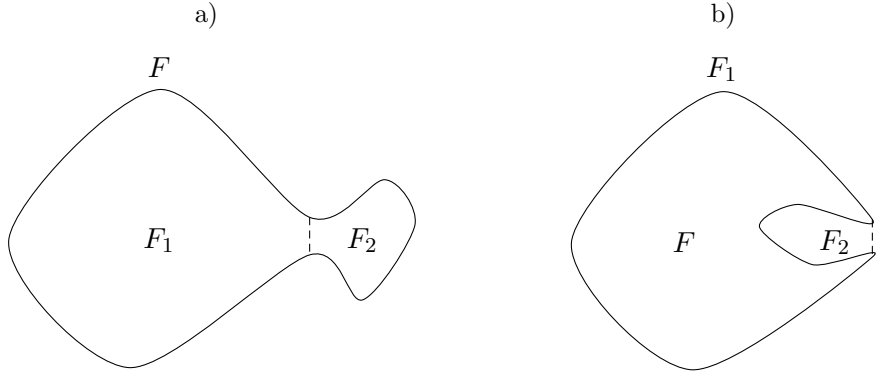

\begin{lemma}\label{lem:non-optimality}
  There exists a universal constant $\sigma_0>0$ with the following
  property: Let $F \subset \R^2$ be a set of finite perimeter with
  $P(F) \leq 3 \pi^\frac{1}{2} |F|^\frac{1}{2}$.  Let the two sets
  $F_1$ and $F_2$ be such that
  $|F_2| \leq \sigma_0 (1-\lambda)^2 \min (1, |F_1|)$ and, up to sets
  of measure zero, we have
	\begin{align}
		\text{ either } F = F_1 \dot \cup F_2 \text{ or }  F_1 = F \dot \cup F_2,
	\end{align}
	where the symbol ``\,$\dot \cup$'' stands for the disjoint union.
	If
	\begin{align}\label{non-optimality_assumption}
		\Sigma := P(F_1) + P(F_2) - P(F) \leq \frac{1}{2} E_{\lambda,\delta}(F_2)
	\end{align}
	then there exists a set $G$ of finite perimeter such that $|G|
        = |F|$ and $E_{\lambda,\delta}(G) < E_{\lambda,\delta}(F)$.
\end{lemma}

\begin{proof}
  Throughout the proof we drop the indices of the energy, as $\lambda$
  and $\delta$ are fixed.  Furthermore, we will only deal with the
  case $F = F_1 \dot \cup F_2$ in detail and highlight the changes for
  the other case at the end.
	
  Let $m_1 := |F_1|$, $m_2 := |F_2|$, $m := |F| = m_1 + m_2 $ and
  $\gamma := \frac{m_2}{m_1}.$ We will consider the competitor
  $G:= l F_1$, where we choose $l := \sqrt{1+ \gamma}$ in order to
  have $|G|=|F|$.
	
  In a first step we compare $E(G)$ and $E(F_1)$.  By scaling, see
  Lemma \ref{lem:rescaling}, and concavity of the square root we know
  that
	\begin{align}\label{non-optimality_caution_1}
		E(G) \leq l E(F_1) \leq E(F_1) + \frac{\gamma}{2} E(F_1).
	\end{align}
	Abbreviating
        $N(\Omega):=\frac{\lambda}{2|\log\delta|}
        \int_{\Omega}\int_{\stcomp{\Omega}}
        \frac{g_\delta(|x-y|)}{|x-y|^3} $ for measurable sets
        $\Omega\subset \R^2$, we see that
	\begin{align}
		\begin{split}
			E(F_1) + E(F_2) &  =  P(F_1) + P(F_2) - N(F_1) - N(F_2)  \\
			& \leq P(F) + \frac{1}{2} E(F_2) - N(F) \\
			&  =  E(F) + \frac{1}{2} E(F_2),
		\end{split}
	\end{align}
	where we made use of assumption
        \eqref{non-optimality_assumption} and 
	\begin{align}\label{non-optimality_caution_2}
		N(F) \leq N(F_1) + N(F_2)
	\end{align}
	obtained by straightforward combinatorics.
	Thus we get 
	\begin{align}\label{intermediate}
		E(F_1) \leq E(F) - \frac{1}{2}E(F_2)
	\end{align}
	and, together with estimate \eqref{non-optimality_caution_1},
	\begin{align}\label{intermediate_2}
          E(G) - E(F) \leq -\frac{1}{2}E(F_2) + \frac{\gamma}{2}E(F_1).
	\end{align}
	
	The isoperimetric inequality gives
	\begin{align}\label{logarithmic_contribution_vanishes_1}
          \frac{P(F_2)}{e\pi |F_2|} \geq
          \frac{2}{e\pi^\frac{1}{2}|F_2|^\frac{1}{2}} \geq 1 
	\end{align}
	due to the assumption
        $|F_2| \leq \sigma_0(1-\lambda)^2\min(1,|F_1|) \leq
        \frac{4}{e^2\pi}$, provided we choose
        $\sigma_0 <\frac{4}{e^2\pi}$.  Plugging this into the a priori
        bound from Proposition \ref{prop:a_priori_lower_bound} and
        again applying the isoperimetric inequality we get
	\begin{align}\label{logarithmic_contribution_vanishes_2}
          E(F_2) \geq (1-\lambda) P(F_2) \geq
          2\pi^\frac{1}{2}(1-\lambda) m_2^\frac{1}{2} > 0. 
	\end{align}	
	Using this and the assumption on $P(F)$ in inequality
        \eqref{intermediate} we obtain
	\begin{align}\label{other_case_similar}
		E(F_1) \leq E(F) \leq P(F) \leq 3 \pi^\frac{1}{2} m ^\frac{1}{2}.
	\end{align}
	The previous two inequalities combined with estimate
        \eqref{intermediate_2} then gives
	\begin{align}
          E(G) -  E(F) \leq  - \pi^{\frac{1}{2}}
          (1-\lambda)m_2^\frac{1}{2} + \frac{3 \pi^\frac{1}{2}
          }{2}\gamma\, m^\frac{1}{2}. 
	\end{align}
	Furthermore, we observe that
	\begin{align}\label{non-optimality_caution_3}
          \gamma m^\frac{1}{2} =
          \left(\frac{m_2}{m_1}\right)^\frac{1}{2}m_2^\frac{1}{2}\left(1+
          \frac{m_2}{m_1}\right)^\frac{1}{2} \leq
          2(1-\lambda)\sigma_0^\frac{1}{2} m_2^\frac{1}{2} 
	\end{align}
	due to $\frac{m_2}{m_1} \leq \sigma_0 (1-\lambda)^2 <1$ for
        $\sigma < 1$.  Therefore, we finally get
	\begin{align}
          E(G) -  E(F) \leq \left( - 1 + 3  \eps^\frac{1}{2}\right)
          \pi^\frac{1}{2} (1-\lambda) m_2^\frac{1}{2} < 0 
	\end{align}
	as long as we choose $\sigma_0 < \frac{1}{9}$.
		
	For the second case, we set $G:= l F_1$ with
        $l := \sqrt{1-\frac{m_2}{m_1}}$, as well as
        $\gamma:= \frac{m_2}{m_1}$ as before, but with
          $m = m_1 - m_2$ now.  There are only three estimates that
        are significantly different than in the first case, namely
        estimates \eqref{non-optimality_caution_1},
        \eqref{non-optimality_caution_2} and
        \eqref{non-optimality_caution_3}.  Dealing with the second
        estimate is a matter of adjusting the combinatorics.  In order
        to obtain an analogue of inequality
        \eqref{non-optimality_caution_1} we have to use Lemma
        \ref{lem:rescaling2} to get
	\begin{align}
          E(G) & \leq l^{-1} E(F_1) + l^{-1} \frac{\lambda |\log l
                 |}{|\log \delta|}P(F_1) \leq E(F_1) +
                 \gamma E(F_1) + C\frac{\lambda
                 \gamma}{|\log \delta|} P(F_1) 
	\end{align}
	for a universal constant $C>0$, as long as $\sigma_0>0$ is
        small enough to carry out the expansions.
	
	The additional term can be estimated via 
	\begin{align}
          C \frac{\lambda \gamma}{|\log \delta|} P(F_1) \leq C \gamma
          P(F_1) \leq C\gamma P(F), 
	\end{align}
	where we exploited that $\Sigma \leq \frac{1}{2} E(F_2)$
        straightforwardly implies
        $P(F_1) \leq P(F) - \frac{1}{2}P(F_2) \leq P(F)$.  Therefore,
        the additional term is estimated against the same quantity as
        the contribution $\gamma E(F_1)$, see estimate
        \eqref{other_case_similar}.
	
	The analogue of the third estimate
        \eqref{non-optimality_caution_3} is given by
	\begin{align}
          \gamma m^\frac{1}{2}  =
          \left(\frac{m_2}{m_1}\right)^\frac{1}{2}m_2^\frac{1}{2}\left(1-
          \frac{m_2}{m_1}\right)^\frac{1}{2} \leq
          (1-\lambda)\sigma_0^\frac{1}{2} m_2^\frac{1}{2}. 
	\end{align}
	This concludes the proof.
\end{proof}

With this tool at hand we are able to prove uniform density bounds for
minimizers.

\begin{lemma}\label{lem:density_bounds}
  There exist universal constants $\sigma,c,r_0>0$ such that under the
  condition
  $|\log\delta|^{-1} \leq \sigma \frac{1-\lambda}{\lambda}$ the
  following holds: Let $\Omega \subset \R^2$ be a minimizer of the
  energy $E_{\lambda,\delta}$ over
  $\mathcal{A}_\pi$.  Then for
  every $x\in \partial \Omega$ and $r\leq (1-\lambda) r_0$ we have
  \begin{align}\label{density_bound_mass}
    c(1-\lambda)^2 r ^2 \leq |\Omega \cap \ball{x}{r} | \leq (\pi -
    c(1-\lambda)^2) r^2 
  \end{align}
  and
  \begin{align}\label{density_bound_length}
    c(1-\lambda) r \leq \Hd^1( \partial \Omega \cap \ball{x}{r}) \leq
    \frac{r}{c(1-\lambda)}. 
  \end{align}
\end{lemma}

\begin{proof}
  We start by noting that by Lemma \ref{lem:nearly_iso} and our
  assumption we have
	\begin{align}\label{bound_perimeter}
          P(\Omega) \leq 3\pi= 3\pi^\frac{1}{2 } |\Omega|^\frac{1}{2} 
	\end{align}
	as long as we have $\sigma>0$ small enough universal.
	
	To prove the lower $\Leb^2$-density bound we re-use large
        portions of the proof of \cite[Lemma 4.3]{km:cpam14},
        exploiting the first case in our non-optimality criterion.
        Let $x\in \partial \Omega$,
        $F_1^r:=\Omega \setminus \ball{x}{r}$ and
        $F_2^r:= \Omega \cap \ball{x}{r}$.  As for $r>0$ small enough
        an application of Lemma \ref{lem:non-optimality} requires
        $|F_2^r| \leq \sigma_0 (1-\lambda)^2$, we choose
        $r \leq (1-\lambda)r_0$ for a universal $r_0>0$.  For the
        analogue of estimate \cite[(4.10)]{km:cpam14}, we
        additionally apply Proposition
        \ref{prop:a_priori_lower_bound} to get
	\begin{align}
          P(F_1^r) + P(F_2^r) - P(\Omega)
          >\frac{1}{2}E_{\lambda,\delta} (F_2^r) \geq
          \frac{1-\lambda}{2}P(F_2^r), 
	\end{align}
	where we used the same arguments involving estimates
        \eqref{logarithmic_contribution_vanishes_1} and
        \eqref{logarithmic_contribution_vanishes_2} to omit the
        logarithmic contribution.
	
	The result is that the constant in estimate
        \cite[(4.12)]{km:cpam14} takes the form
        $c= \tilde c (1-\lambda)$ for a universal $\tilde c > 0$, and
        the differential inequality \cite[(4.13)]{km:cpam14} for the
        Lipschitz continuous function
        $U(r):= \min \left\{ |\Omega \cap \ball{x}{r}|,
            \frac{1}{2} |\ball{x}{r}| \right\}$ can be seen to take
        the form
	\begin{align}
          \frac{\intd U^\frac{1}{2}(r)}{\intd r} \geq \tilde
          c(1-\lambda) \quad \text{for a.e.} \quad r < (1 -
          \lambda) r_0.
	\end{align}
	Integrating this inequality yields the lower density
        bound $U(r) \geq \tilde c(1-\lambda)r$.
	
	The upper $\Leb^2$-density bound follows in the same way by
        using the second case of Lemma \ref{lem:non-optimality}.  The
        lower density bound for the length follows from those of the
        mass via the relative isoperimetric inequality, while the
        upper bound is a straightforward consequence of 
	\begin{align}
          2 \Hd^1\left(\Omega \cap \partial \ball{x}{r} \right) \geq
          \frac{1-\lambda}{2} \left(\Hd^1\left(\partial \Omega \cap
          \ball{x}{r} \right) + \Hd^1\left(\Omega \cap \partial
          \ball{x}{r} \right) \right), 
	\end{align}
	see \cite[equation (4.11)]{km:cpam14} for the proof. 
\end{proof}

As a first application of the uniform density estimates, we briefly
quantify the convergence of minimizers to disks in order to have
a quantitative dependence of the required smallness of $\delta$ on
$\lambda$ in the final statement.
	
\begin{lemma}\label{lem:convergence_rates}
  There exist universal constants $C,\sigma>0$ such that under the
  condition
  $|\log\delta|^{-1} \leq \sigma \frac{(1-\lambda)^2}{\lambda}$ the
  following holds: The regular representative of any minimizer
  $\Omega$ of $E_{\lambda,\delta}$ over $\mathcal{A}_\pi$ is simply
  connected and, after translation, satisfies
  \begin{align}\label{statement_after_clean-out}
    \operatorname{dist} \left(\partial \Omega, \Sph^1 \right)
    \leq 
    \frac{C \lambda^\frac{1}{2}}{(1-\lambda)^\frac{1}{2}
    |\log\delta|^\frac{1}{2}},
  \end{align}
 where $\operatorname{dist}$ denotes
  the Hausdorff-distance of sets.
\end{lemma}

\begin{proof}
  Let $\Omega$ be the regular representative obtained in Proposition
  \ref{prop:ELG_generalized_minimizers}. Observe that by Lemma
  \ref{lem:connectedness} the set $\Omega$ is connected.  As the set
  $\Omega \subset \R^2$ is bounded with a smooth boundary, there
  exists a simply connected set $\widetilde \Omega$ such that
  $\Omega\subset \widetilde \Omega$,
  $\partial \widetilde \Omega \subset \partial \Omega$ is a Jordan
  curve, and such that
  $\partial \Omega \setminus \partial \widetilde \Omega$ is either
  empty or itself a union of disjoint Jordan curves $\gamma_i$ for
  $1\leq i \leq N$ and some $N\in \N$.  In order to prove that
  $\Omega$ is simply connected, we only have to rule out the latter
  case.  Towards a contradiction, we assume that
  $\partial \Omega \setminus \partial \widetilde \Omega$ is indeed a
  non-empty union of Jordan curves.
  
  As we have
  $\tilde r := \sqrt{\frac{|\widetilde \Omega|}{\pi}}\geq 1$, the
  isoperimetric inequality and Lemma \ref{lem:nearly_iso} imply
  \begin{align}\label{perimeter_of_holes}
    \Hd^{1}(\partial \Omega \setminus \partial \widetilde \Omega) \leq
    \Hd^{1}(\partial \Omega \setminus \partial \widetilde \Omega) +
    P(\widetilde \Omega) - P(\ball{0}{\tilde r}) \leq P(\Omega) - 
    P(\ball{0}{1}) \leq \frac{C\lambda}{(1-\lambda) |\log\delta|}.
  \end{align}
  Furthermore, for $F_2 := \widetilde \Omega \setminus \Omega$ we have
  $\partial F_2 = \partial \Omega \setminus \partial \widetilde
  \Omega$. As a result of estimate \eqref{perimeter_of_holes}, the
  isoperimetric inequality implies
  \begin{align}
    |F_2| \leq \frac{C\lambda^2}{(1-\lambda)^2 |\log\delta|^2}
    \leq \sigma_0 (1-\lambda)^2 \min\left\{1,\left|\widetilde
    \Omega\right|\right\} 
  \end{align}
  with $\sigma_0$ as in Lemma \ref{lem:non-optimality}, provided
  $\frac{\lambda}{(1-\lambda)^2|\log\delta|}\leq \sigma$ for
  $\sigma>0$ universally small enough and where we bounded
  $|\widetilde \Omega|$ by a universal constant due to
  $P(\widetilde \Omega) \leq P(\Omega)$ and Lemma
  \ref{lem:nearly_iso}.  Similarly to estimate
  \eqref{logarithmic_contribution_vanishes_2} we see that
  \begin{align}
    \frac{1}{2}E(F_2) > 0 = P(F_2) + P\left(\widetilde \Omega
    \right) -P(\Omega). 
  \end{align}
  Therefore, we can apply the second case of Lemma
  \ref{lem:non-optimality} with $F_1=\widetilde \Omega$ and
  $F= \Omega$ to conclude that $\Omega$ cannot have been a minimizer,
  which contradicts the assumption that $\Omega$ is not simply
  connected.
  
  As $\Omega$ is simply connected, we may apply Bonnesen's inequality
  \cite{bonnesen24,osserman79,fuglede91} and Lemma
  \ref{lem:nearly_iso} to $ \Omega$ to get that after a suitable
  translation we have
  \begin{align}
    \operatorname{dist} \left(\partial \Omega, \Sph^1 \right)
    & \leq  C \sqrt{P^2\left( \Omega\right) - P^2(\ball{0}{1})} \leq 
      \frac{C \lambda^\frac12}{(1-\lambda)^\frac12
      |\log\delta|^\frac12},
  \end{align}
  which concludes the proof.
\end{proof}

We now have all the necessary ingredients to make the qualitative
regularity of Proposition \ref{prop:ELG_generalized_minimizers}
quantitative.  The basic strategy is to improve the scale at which
$\partial \Omega$ is a graph.  We first give a precise meaning to this
statement:

\begin{Def}\label{def:C_alpha_at_scale_r}
  We say that a set $\Omega$ \emph{has a uniform $C^{2}$-boundary at
    scale $r>0$} if the boundary is an oriented $C^2$-manifold such
  that for every $x \in \partial \Omega$ there exists a rotation
  $S \in SO(2)$ and a function $h \in C^2([-r,r])$ with the following
  properties: 
  \begin{enumerate}[(i)]
        \item
          $S(\Omega-x) \cap (-r,r)^2 = \left\{(t,s) \in (-r,r)^2:
            -r< s < h(t) \right\}$,
        \item $h(0) = h'(0)=0$ and $|h''(t)| \leq r^{-1}$ for all
            $t \in [-r,r]$.
	\end{enumerate}
\end{Def}

In this terminology, Proposition \ref{prop:ELG_generalized_minimizers}
states that a minimizer $\Omega_\delta$ of $E_{\lambda,\delta}$ has a
uniform $C^{2}$-boundary at some small scale $r(\Omega_\delta) >0$.
In the next statement, Proposition \ref{prop:improvement_of_flatness},
we present the machinery which will allow us to iteratively increase
the scale of regularity.  The idea is to expand the potential
$v_{\delta}$, see definition \eqref{potential_definition}, into a
localizing part that approaches the curvature and into a larger scale
remainder.  As we have $\lambda <1$, the curvature is still the
leading order term in the Euler-Lagrange equation
\eqref{ELG_generalized_minimizers}.  Once we obtain convexity of the
minimizers, the proposition will also allow us to identify the
minimizers as disks in combination with a rigidity statement for the
disk.

\begin{prop}\label{prop:improvement_of_flatness}
  There exist universal constants $\sigma,C>0$ with the following
  property: Let
  $|\log\delta|^{-1} \leq \sigma \frac{1-\lambda}{\lambda}$ and assume
  that a minimizer $\Omega$ of $E_{\lambda,\delta}$ over
  $\mathcal{A}_\pi$ has a uniform $C^{2}$-boundary at scale $r > 0$.
  Then the curvature $\kappa$ of $\Omega$ satisfies
	\begin{align}\label{curvature_vs_remainder}
	  \begin{split}
            \osc \kappa & := \max_{x,y\in\partial \Omega} \left(
              \kappa(x) - \kappa(y) \right) \\
            & \leq \frac{C \lambda}{(1-\lambda) |\log\delta|}
            \inf_{a\in \R} \max_{x\in \partial \Omega}
            \left|\int_{\partial\Omega \setminus \ball{x}{r}} \nabla
              \Phi_\delta (|y-x|) \cdot \nu (y) \intd y - a
            \right|.
	  \end{split}
	\end{align}
\end{prop}

\begin{proof}
  As already mentioned, we will argue by expanding the potential $v$
  on the boundary of $\Omega$.  To this end, we make use of the
  representation \eqref{potential_boundary_formulation} of the
  potential.  As we will later argue by taking differences of the
  Euler-Lagrange equation in two different points, the constant term
  drops out, so that we may focus on the integral.  For
  $x \in \partial \Omega$ we split it up into the local
  contribution
\begin{align}\label{local_part}
  L(x) := \int_{\partial \Omega \cap  \ball{x}{ r}} \nabla
  \Phi_\delta (|y-x|)
  \cdot \nu (y) \intd \Hd^1(y) 
\end{align}
from $\ball{x}{r}$ and
into a remainder we abbreviate with
\begin{align}
  \label{eq:R}
  R(x) := \int_{\partial\Omega \setminus \ball{x}{r}} \nabla
  \Phi_\delta (|y-x|) \cdot \nu (y)  \intd y. 
\end{align}

\textit{Step 1: Expand the integrand in
  \eqref{local_part}.}\vspace{1mm}\\ 
Without loss of generality we may assume $x=0$.  To control the
localized part we transform the integral in
\eqref{potential_boundary_formulation} into the graph coordinates
$y = (t, h(t))$.  Recall that we asked $h(t)$ to satisfy
$h(0) = h'(0) =0$ in Definition \ref{def:C_alpha_at_scale_r}.
	
Let $T_1,T_2>0$ be such that
$y(-T_1), y(T_2) \in \partial \Omega \cap \partial \ball{0}{r}$.
As the Jacobian in the transformation formula
of boundary integrals is the length of the tangent, we see that
$\nu(y) \intd \Hd^1(y)$ transforms into $(-h'(t),1) \intd t$, so that
we aim to compute
	\begin{align}
          L(x) =  \int_{-T_1}^{T_2} \frac{\Phi'_\delta \left(
          |y(t)| \right)}{ |y(t)| } \,
          (t,h(t)) \cdot (-h'(t),1) \intd 
          t.
	\end{align}
	  Note that
        the scalar product above may be represented as
	\begin{align}
		(t,h(t)) \cdot (-h'(t),1)= h(t) - th'(t) =
          - \int_0^t s h''(s) \intd s.
	\end{align}
      As we defined the curvature of convex bodies to be positive, the
      curvature can be computed as
	\begin{align}\label{curvature_in_graph_coordinates}
          \kappa \left(y(s)\right) = -
          \frac{h''\left(s\right)}{\left(1+ 
          \left(h' \left(s\right)\right)^2\right)^{\frac{3}{2}}}.  
	\end{align}
	  Therefore,
          we get
	\begin{align}
          h(t) - th'(t) =\int_0^ts\kappa(y(s) )
          \left(1+ \left(h' \left(s\right)\right)^2\right)^{\frac{3}{2}} \intd s. 
	\end{align}
	
	The bound $|h''(t)| \leq r^{-1}$ and the equalities
        $h(0) = h'(0) =0$, all resulting from Definition
        \ref{def:C_alpha_at_scale_r}, imply
	\begin{align}
          \left(1+ \left(h' \left(s\right)\right)^2\right)^{\frac{3}{2}} 
          = 1+ O\left( r^{-2}s^2\right), 
	\end{align}
	as well as $|h(t)| \leq \frac{1}{2}r^{-1}t^2$ and
	\begin{align}\label{Peano}
          |y(t)| = \left( t^2 + h^2(t) \right)^{\frac{1}{2}}
          =  |t| \left( 1 + O \left(r^{-2}
          t^2 \right) \right).
	\end{align}
	Here, we use the $O$-notation to denote estimates up to
        universal constants.  Also note that we always have
        $|y(t)| \geq |t|$, so that as a consequence we get
	\begin{align}
          \frac{1}{|y(t)|^2} = \frac{1}{|t|^2}\left( 1 + O \left(r^{-2}
          t^2 \right) \right)
	\end{align}
	and
	\begin{align}\label{expansion_minus_3}
          \frac{1}{|y(t)|^3} = \frac{1}{|t|^3}\left( 1 + O \left(r^{-2}
          t^2 \right) \right).
	\end{align}

        \textit{Step 2: If $r \leq \delta$ we have the following:
          There exist
          $y_1, y_2 \in \partial \Omega \cap
            \overline{\ball{0}{r}}$ such that}
	\begin{align}\label{step_2_expansion}
          O\left( \lambda \kappa(y_1)
          |\log\delta|^{-1}\right) + R(x)< v(x) <  O\left( \lambda
          \kappa(y_2)|\log\delta|^{-1}\right) + R(x).
	\end{align} \\
        Using the explicit form of $\Phi_\delta$ in
          \eqref{Phi_explicit}, we see 
	\begin{align}\label{Phi_nabla}
          \Phi'_\delta \left( |y| \right)=
		\begin{cases}
			- \frac{1}{|y|^2} & \text{ if } |y|\geq \delta\\
			-\frac{1}{\delta |y|} & \text{ if } |y|<\delta
		\end{cases},
	\end{align}
      so that from the estimates in Step 1 we get
	\begin{align}
	   \begin{split}
             L(x) & = - \int_{0}^{T_2} \frac{1}{\delta t^2}\left( 1 +
               O(r^{-2}t^{2}) \right) \left(
               \int_0^ts\kappa\left(y(s)\right)\left( 1+
                 O\left(r^{-2} s^{2}\right) \right)\intd s \right)
             \intd t\\
             & \quad - \int_{-T_1}^{0} \frac{1}{\delta t^2}\left( 1 +
               O(r^{-2}t^{2}) \right) \left(
               \int_t^0|s|\kappa\left(y(s)\right)\left( 1+
                 O\left(r^{-2} s^{2}\right) \right)\intd s \right)
             \intd t.
	  \end{split}
	\end{align}
	
	In order to get a lower bound, we choose
        $t_1 \in \operatorname{arg max}_{s \in [-T_1, T_2]
          }\kappa(y(s))$ and estimate
	\begin{align}
          L(x)
          & \geq  - \kappa(y(t_1))\int_{-T_1}^{T_2}
            \frac{1}{\delta}\left( 1+ O\left(r^{-2}
            t^{2}\right)  \right) \intd t = O(\kappa(y(t_1))),
	\end{align}
	as a result of
        $\frac{T_1+T_2}{\delta} \leq \frac{|y(T_1)|+|y(T_2)|}{\delta}
        \leq \frac{2r}{\delta}\leq 2$.  This is the desired lower
        bound of estimate \eqref{step_2_expansion}.  The upper bound
        follows similarly by choosing
        $t_2 \in \operatorname{arg min}_{s \in
          [-T_1,T_2]}\kappa(y(s))$.
	
	\vspace{1mm}
        \textit{Step 3: If instead we have $r > \delta$ there exist
          $y_1, y_2 \in \partial \Omega \cap
            \overline{\ball{0}{r}}$ such that we have the two bounds
	\begin{align}
          \frac{\lambda}{2} \kappa(y_1) \left(1 +
          O\left(|\log\delta|^{-1}\right) \right) + R(x)< v(x) <
          \frac{\lambda}{2} \kappa(y_2) \left(1 +
          O\left(|\log\delta|^{-1}\right) \right) + R(x). 
	\end{align}}\\	
      In the case $r>\delta$, we additionally define $\widetilde
      T_1,\widetilde T_2 >0$ such that $y(-\widetilde T_1),
      y(\widetilde T_2) \in \partial \Omega \cap \partial
      \ball{0}{\delta}$ and write $L(x) = L_\delta^-(x) +
        L_\delta^0(x) + L_\delta^+(x)$, where the three terms are the
        contributions from $[-T_1, -\widetilde T_1]$, $[-\widetilde
        T_1, \widetilde T_2]$ and $[\widetilde T_2, T_2]$, respectively.
      Again choosing $t_1\in
        \operatorname{arg max}_{s\in [-T_1,T_2] }\kappa(y(s))$ we argue as in
      the previous step that $L_\delta^0 =  O(\kappa(y(t_1)))$.
      For $L_\delta^+(x)$, we get from equations
      \eqref{Phi_nabla} and \eqref{expansion_minus_3} that 
	\begin{align}
	  \begin{split}
            & \quad L_\delta^+(x)= - \int_{\widetilde T_2}^{T_2}
            \frac{1}{t^3}\left( 1 + O(r^{-2}t^{2}) \right) \left(
              \int_0^ts\kappa\left(y(s) \right)\left( 1+
                O\left(r^{-2}s^{2}\right) \right)\intd s \right)
            \intd t\\
            & \geq - \kappa(y(t_1)) \int_{\widetilde T_2}^{T_2}
            \frac{1}{2t} \left( 1 + O(r^{-2}t^{2}) \right) \intd t\\
            & = - \frac{\kappa(y(t_1))}{2} \left( \log T_2 - \log
              \widetilde T_2   + O(r^{-2} T_2^{2})\right)\\
            & = - \frac{\kappa(y(t_1))}{2}(|\log\delta| + O(1)),
	  \end{split}
	\end{align}
	where we used $T_2 \leq |y(T_2)| = r$ and
        $\log \widetilde T_2 = \log \delta + \log \frac{\widetilde
          T_2}{\delta} = \log \delta +O(1)$ due to
        $\widetilde T_2 \leq \delta$.  Treating the term
          $L_\delta^-(x)$ analogously, we obtain
	\begin{align}
          L(x) \geq -\kappa(y(t_1))(|\log\delta| + O(1)). 
	\end{align}
	Again the argument for the upper bound is similar.
	
\vspace{1mm}
        \textit{Step 4: Conclusion.}\vspace{1mm}\\
	Let $y_1$ and $y_2$ realize the maximum in
        $\osc \kappa = \max_{y_1,y_2 \in \partial \Omega} (\kappa(y_1)
        - \kappa(y_2))$.  We first deal
        with the more complicated case $r > \delta$.  Taking the
        difference of the Euler-Lagrange equation
        \eqref{ELG_generalized_minimizers} in the two points we see
        from Step 3 that there exist
        $\tilde y_1,\tilde y_2 \in \partial \Omega$ such that we have
	\begin{align}
	  \begin{split}
            \osc \kappa
            & = \kappa(y_1) - \kappa(y_2) = -2( v(y_1)  -  v(y_2)) \\
            & \leq \lambda \left(1+ O\left(|\log \delta |^{-1}
                \right) \right) (\kappa( \tilde y_1) - \kappa(\tilde
              y_2) ) -
              \frac{\lambda}{|\log\delta|} (R(y_1) - R(y_2)) \\
            & \leq \lambda \left(1+ O\left(|\log \delta |^{-1} \right)
            \right) (\kappa( \tilde y_1) - \kappa(\tilde y_2) ) +
            \frac{2 \lambda}{|\log\delta|} \inf_{a\in \R}
            \max_{x\in
              \partial \Omega} | R(x) -a |\\
            & \leq \lambda \left(1+ O\left(|\log \delta |^{-1} \right)
            \right) \osc \kappa + \frac{2 \lambda}{|\log\delta|}
            \inf_{a\in \R} \max_{x\in \partial \Omega} | R(x) -a |.
	  \end{split}
	\end{align}
	This gives 
	\begin{align}
	  \begin{split}
            \osc\kappa \leq \frac{C \lambda}{(1-\lambda) |\log
              \delta|} \inf_{a\in \R} \max_{x\in \partial \Omega} |
            R(x) -a |,
		\end{split}
	\end{align}
	if we have $|\log\delta|^{-1} \leq \sigma
        \frac{1-\lambda}{\lambda}$ for a universal $\sigma>0$ small
        enough.  A similar argument in the case
        $r<\delta$ concludes the proof.
\end{proof}

We are now in a position to prove convexity of the minimizers.  The
strategy is to show that the $|\log\delta|^{-1}$ factor in front of
the remainder in estimate \eqref{curvature_vs_remainder} can be used
to iteratively improve the scale $r>0$ of regularity all the way up to
some uniform radius independent of $\delta$.  The resulting improved
bound for the curvature can be combined with the non-optimality
criterion in Lemma \ref{lem:non-optimality} to also improve the
scale at which $\partial \Omega$ is a graph.

\begin{prop}\label{prop:convexity}
  There exist universal constants $\sigma,C>0$ with the following
  properties: If
  $|\log\delta|^{-1} \leq \sigma \frac{(1-\lambda)^{5}}{\lambda}$ then
  a minimizer $\Omega$ of $E_{\lambda,\delta}$ over $\mathcal{A}_\pi$
  satisfies
	\begin{align}\label{estimate_oscillation}
		\osc \kappa \leq \frac{C\lambda}{(1-\lambda)^2
          |\log\delta|}. 
	\end{align}
        Additionally, the set $\Omega$ is convex and has a uniform
        $C^2$-boundary at scale $\sigma (1-\lambda)$.
\end{prop}

\begin{proof}
  The proof consists of five steps. In the first four steps, we assume
  that $\Omega$ has a uniform $C^2$-boundary at scale
  $0<r < \frac{1}{2}(1-\lambda)r_0$, where $r_0>0$ is the universal
  constant given in Lemma \ref{lem:density_bounds}. Recall that such a
  value of $r$ exists by Proposition
  \ref{prop:ELG_generalized_minimizers}, but depends on $\Omega$ and
  the values of $\delta$ and $\lambda$.

\vspace{1mm}
  \textit{Step 1: The remainder term defined in \eqref{eq:R} satisfies
    $\|R\|_{L^\infty(\partial\Omega)} \leq
    \frac{C}{1-\lambda}r^{-1}$ for a universal constant
    $C>0$.}\vspace{1mm}\\ 
  Without loss of generality we again work with $x=0$ belonging to
  $\partial \Omega$.  For all $k \in \mathbb{N}$ such that
  $ 2^k r \leq (1-\lambda) r_0$ we get
  $\Hd^{1} (\partial \Omega \cap \ball{0}{2^k r} ) \leq \frac{C}{
    1-\lambda}2^k r$ by Lemma \ref{lem:density_bounds}.  Using
  \eqref{Phi_nabla}, we see that
	\begin{align}
          |R(x)| & \leq \int_{\partial \Omega \setminus
                   \ball{0}{r} } |\Phi_\delta'(|y|)| \intd \Hd^1(y) \leq C
                   \int_{\partial \Omega \setminus \ball{0}{r} }
                   \frac{1}{|y|^2} \intd \Hd^1(y) 
	\end{align}
	even in the case $r <\delta$.
	Dyadically decomposing the domain of integration we then
          see that
	\begin{align}\label{remainder}
	  \begin{split}
            | R(x)| & \leq \frac{C}{1-\lambda}
            \sum_{k=1}^{\left\lfloor \log_2 \frac{r_0(1 - \lambda)}
                   { r} \right\rfloor} \left(2^{k+1}r\right)\,
            \left( 2^{-2k} r^{-2}\right) + C \int_{\partial \Omega
              \setminus \ball{0}{\frac{1}{2}(1-\lambda) r_0} }
            \frac{1}{|y|^2}
            \intd \Hd^1(y)  \\
            & \leq \frac{C}{1-\lambda}r^{-1} +
            \frac{C P(\Omega)}{r^2_0 (1-\lambda)^2}\\
            & \leq \frac{C}{1-\lambda}r^{-1}
	  \end{split}
	\end{align}
	as a result of $r< (1-\lambda)r_0$ and estimate
        \eqref{bound_perimeter}. 
	
	\vspace{1mm}
        \textit{Step 2: For any $K > 0$ and all $\sigma$ sufficiently
          small depending only on $K$, the curvature satisfies
          $\|\kappa\|_\infty \leq \frac{1 - \lambda}{Kr}$.}\vspace{1mm} \\
        We apply the Gauss-Bonnet theorem to the Jordan curve
        $\partial \Omega$, see Lemma \ref{lem:convergence_rates}, to
        obtain
	\begin{align}
          \int_{\partial \Omega} \kappa \intd \Hd^1 = 2\pi.
	\end{align}
	The mean value theorem then implies that there exists
        $x\in \partial \Omega$ with $\frac{2}{3}\leq \kappa(x) \leq 1$
        because by estimate \eqref{bound_perimeter} and by the
        isoperimetric inequality we have
	\begin{align}
          3\pi \geq P(\Omega) \geq  2\pi^\frac{1}{2}
          |\Omega|^\frac{1}{2} = 2\pi. 
	\end{align}
        In particular, for any $K>0$ to be chosen later we have
	 \begin{align}\label{bound_curvature}
           \|\kappa\|_{\infty}  \leq \osc \kappa + 1 \leq
           \frac{C\lambda}{(1-\lambda)^2 |\log\delta|} 
           r^{-1} + 1 \leq  \frac{1-\lambda}{2Kr} + 1 \leq
           \frac{1-\lambda}{Kr}, 
	 \end{align}
	 where the second inequality holds due Proposition
         \ref{prop:improvement_of_flatness} and Step 1, the third is a
         result of
         $ |\log\delta|^{-1} \leq \sigma
         \frac{(1-\lambda)^{3}}{\lambda}$ and the fourth follows from
         $r \leq \sigma( 1-\lambda)$, provided $\sigma >0$ is small
         enough depending only on $K$.
	 
	 \vspace{1mm} \textit{Step 3: For every $x\in \partial \Omega$
           there exists a rotation $S\in SO(2)$ and a box
           $Z_{4r}(x):= x + S^{-1}(x)\left(-4r, 4r \right)^2$ such
           that the segment $G(x)$ of $\partial \Omega \cap Z_{4r}(x)$
           containing $x$ is an appropriately rotated and translated
           graph of a function $h \in C^2([-4r,4r])$ satisfying
           $h(0) = h'(0) = 0$ and $|h''(t)| \leq (2r)^{-1}$ for $t \in
           (-2r,2r)$.} \vspace{1mm}\\
         We parametrize $\partial \Omega$ by arc length via a closed
         curve
         $\gamma : [-\frac12 P(\Omega),\frac12 P(\Omega)] \to \R^2$ of
         class $C^2$ such that we have
         $\nu(\gamma(t)) = \dot \gamma ^\perp (t)$ for the outer unit
         normal $\nu$ to $\Omega$.  For convenience, we then extend
         $\gamma$ periodically to the whole of $\R$. The curvature
         bound \eqref{bound_curvature} immediately gives
	\begin{align}\label{to_be_expanded}
          |\dot \gamma (t) - \dot \gamma (s)| \leq (K r)
          ^{-1}|t-s| \qquad \forall t,s \in \R.
	\end{align}
        
	Without loss of generality, for
        $t \in [-\frac12 P(\Omega),\frac12 P(\Omega)]$ we may assume
        that $\dot \gamma(t)= e_1$, see Figure \ref{fig:Z} for an
        illustration. Note that the latter implies that
        $S(\gamma(t)) = \operatorname{Id}$.  In that situation, we can
        square and expand estimate \eqref{to_be_expanded} for
        $|s - t| \leq \sqrt{2 - \sqrt{2}}K r $ to obtain
	\begin{align}\label{graph_estimate}
          e_1 \cdot \dot \gamma(s) \geq 1 - \frac{1}{2}(K r)
          ^{-2}|s - t|^2 \geq \frac{1}{\sqrt{2}}. 
	\end{align} 
	Therefore, if we choose $K$ large enough to have
        $\sqrt{2 - \sqrt{2}}K \geq 4$ we can find
        $h \in C^2([-4r, 4r])$ such that the segment $G(\gamma(t))$ is
        the graph of $h$ shifted by $\gamma(t)$. In particular, we
        have $\dot \gamma = \frac{1}{\sqrt{1+| h'|^2 }} (1, h')$.
        Estimate \eqref{graph_estimate} also implies
        $\sqrt{1+| h'(t) |^2} \leq \sqrt{2}$ for all $t \in (-4r,4r)$,
        and thus $h$ is $1$-Lipschitz.  Using the representation of
        the curvature in graph coordinates
        \eqref{curvature_in_graph_coordinates} we obtain
	\begin{align}
          \| h''\|_\infty = \left\| \kappa (1+ |h'|^2)
          ^\frac{3}{2} \right\|_\infty \leq \left(
          2^{-\frac{3}{2}} K r \right)^{-1} \leq (2r)^{-1} 
	\end{align}
	for $K$ big enough universal.

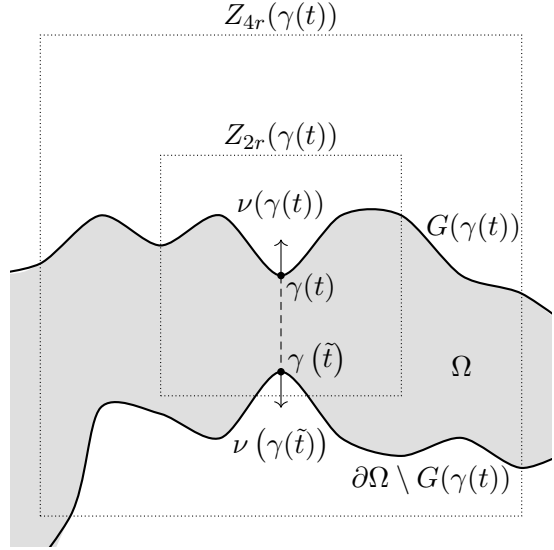
\begin{figure}
	\centering
	\begin{tikzpicture}[scale=.8]
		\begin{scope}
			\clip (-4.5,-4.5) -- (-4.5,5) -- (4.5,5) -- (4.5,-4.5) -- cycle;
			\clip plot [smooth cycle] coordinates {(0,-10)(-10,0)(-5,0)(-4,.2)(-3,1)(-2,.5)(-1,1)(0,0) (1,1)(2,1)(3,0)(4,-.3)(5,-1)(6,-1)(15,-1)};
			\clip plot [smooth cycle] coordinates {(0,10)(-10,-6)(-5,-6)(-3.5,-4)(-3,-2.2)(-2,-2.3)(-1,-2.7)(0,-1.6) (1,-2.7)(2,-3)(3,-2.7)(4,-3.2)(5,-2.8)(6,-2.8)(10,-2.8)(20,-1)};
			\fill[opacity=.5,color=lightgray] (-4.5,-4.5) -- (-4.5,5) -- (4.5,5) -- (4.5,-4.5) -- cycle;
		\end{scope}
		\clip (-4.5,-4.5) -- (-4.5,5) -- (4.5,5) -- (4.5,-4.5) -- cycle;
		\draw[densely dotted] (-4,-4) -- (4,-4) -- (4,4) -- (-4,4) -- cycle;
		\draw[densely dotted] (-2,-2) -- (2,-2) -- (2,2) -- (-2,2) -- cycle;
		\draw[thick] plot [smooth] coordinates {(-5,0)(-4,.2)(-3,1)(-2,.5)(-1,1)(0,0) (1,1)(2,1)(3,0)(4,-.3)(5,-1)};
		\draw[thick] plot [smooth] coordinates {(-5,-6)(-3.5,-4)(-3,-2.2)(-2,-2.3)(-1,-2.7)(0,-1.6) (1,-2.7)(2,-3)(3,-2.7)(4,-3.2)(5,-2.8)};
		\node[circle,fill=black, inner sep=1pt] at (0,0) {};
		\node at (.5,-.2) {$\gamma(t)$};
		\node[circle,fill=black, inner sep=1pt] at (0,-1.6) {};
		\node at (.6,-1.4) {$\gamma\left(\tilde t\right)$};
		\draw[densely dashed] (0,0) -- (0,-1.6);
		\draw[->] (0,0) -- (0,.6) node[label=above:{$\nu(\gamma(t))$}]{};
		\draw[->] (0,-1.6) -- (0,-2.2)
                node[label=below:{$\nu\left(\gamma(\tilde t)
                  \right)$}]{}; 
		\node at (3,-1.5) {$\Omega$};
		\node at (0,2.3) {$Z_{2r}(\gamma(t))$};
		\node at (0,4.3) {$Z_{4r}(\gamma(t))$};
		\node at (3.2,.8) {$G(\gamma(t))$};
		\node at (2.5,-3.4) {$\partial \Omega \setminus G(\gamma(t))$};
	\end{tikzpicture}
	\caption{\label{fig:Z} Sketch of the geometry in Steps 3 and 4
          in the case
          $[\gamma(t), \gamma(\tilde t)] \subset \overline\Omega$,
          where $[\gamma(t), \gamma(\tilde t)]$ denotes the line
          segment connecting $\gamma(t)$ and $\gamma(\tilde
          t)$, shown as the dashed line in the figure. Note that also the case
          $[\gamma(t), \gamma(\tilde t)] \subset
          \overline{\stcomp{\Omega}}$ is possible.}
\end{figure}
	
\vspace{1mm} \textit{Step 4: The set $\Omega$ has a uniform
  $C^2$-boundary at
  scale $2r$.}\vspace{1mm}\\
To this end, it only remains to prove
$\partial \Omega \cap Z_{2r}(\gamma(t)) = G(\gamma(t)) \cap
Z_{2r}(\gamma(t))$ for every
$t\in [-\frac12 P(\Omega),\frac12 P(\Omega)] $.  The fact that
$S(x)(\Omega -x)\cap (-2r,2r)^2$ is the subgraph of the function $h$,
in the setting of Step 3, then follows from $\nu(\gamma(t)) = e_2$.
 
For $t\in [-\frac12 P(\Omega),\frac12 P(\Omega)]$ we abbreviate the
Hausdorff distance
\begin{align}
  d(t) := \operatorname{dist}\big(\{\gamma(t)\}, \partial \Omega
  \setminus G(\gamma(t))   \big). 
\end{align}
In terms of this function, the statement of the claim reduces to
proving $\inf_t d(t) \geq 2\sqrt{2} \, r$ since
$2\sqrt{2} \, r = \sup_{y \in Z_{2r}(\gamma(\tilde t))}
|y-\gamma(\tilde t)|$ for all
$\tilde t \in [-\frac12 P(\Omega),\frac12 P(\Omega)]$.
	
Towards a contradiction we assume the opposite.  It is easy to see
that $d$ is continuous and thus the set
\begin{align}
  A:=\left\{t\in
  \left[-\tfrac12 P(\Omega),\tfrac12 P(\Omega)\right] : d(t)
  < 2\sqrt{2} \, r\right\}
\end{align}
is open.  As a result, $d$ achieves its minimum on $A$, so that we can
choose $ t, \tilde t \in [-\frac12 P(\Omega),\frac12 P(\Omega)] $ such
that $|\gamma(t) - \gamma(\tilde t)| = d(t) = \min d$ and
$\gamma(\tilde t) \in \partial \Omega \setminus G(\gamma(t)) $, see
Figure \ref{fig:Z}.
	
By
$\operatorname{dist}(\{\gamma(t)\}, \overline{G(\gamma(t))} \setminus
G(\gamma(t)) \geq 4r$, we also get
$\gamma(\tilde t) \in \partial \Omega \setminus
\overline{G(\gamma(t))}$.  Therefore, the minimization over $t$
implies that the vector $\gamma(t) - \gamma(\tilde t)$ must be
parallel to the outer normal $\nu(\gamma(t))$, a statement which we
abbreviate by $\nu(\gamma(t)) \parallel \gamma(t) - \gamma(\tilde t)$.
Similarly, as
$\ball{\gamma(t)}{2\sqrt{2} \, r} \subset\subset Z_{4r}(\gamma(t))$ we
also have that
$\nu(\gamma(\tilde t)) \parallel \gamma(t) - \gamma(\tilde t)$.  Thus
we either have $\nu(\gamma(t)) = -\nu(\gamma(\tilde t))$ or
$\nu(\gamma(t)) =\nu(\gamma(\tilde t))$.  However, the latter can be
excluded, since then the line segment
$\left[\gamma(t), \gamma(\tilde t) \right]$ would cross
$\partial \Omega \setminus G(\gamma(t))$ at least once in between the
two endpoints, contradicting minimality of $d(t)$.  As a result, we
get $|\nu(\gamma(t)) - \nu(\gamma(\tilde t)) | =2$.

              Let $I$ be the shorter of the two intervals
              $[\min(t,\tilde t), \max(t,\tilde t)]$ and
              $[\max(t,\tilde t), \min(t,\tilde t) + P(\Omega)]$.
              Using
              $2= |\nu(\gamma(t)) - \nu(\gamma(\tilde t)) | \leq
              \|\kappa\|_\infty | I | $ together with the bound for
              the curvature \eqref{bound_curvature} we get
	\begin{align}\label{length_of_cut}
		\frac{P(\Omega)}{2} \geq  | I |   \geq \frac{2 Kr}{1-\lambda}.
	\end{align}
        Next, let $F_2$ be the interior of the Jordan curve
        $\tilde \gamma$ defined as the concatenation of $\gamma|_I$
        and the segment $\left[\gamma(t), \gamma(\tilde t) \right]$,
        see Figure \ref{fig:F_2}.  Under the assumption
        $|\log \delta|^{-1} \leq \sigma
        \frac{(1-\lambda)^2}{\lambda}$, Lemma
        \ref{lem:convergence_rates} implies that after a suitable
        translation we have
        \begin{align}
          \label{eq:gammaI}
          \gamma(I) \subset \Sph^1 + \ball{0}{\eps(\delta)}, 
        \end{align}
        where
        \begin{align}
          \eps(\delta)
          :=\frac{C\lambda^\frac{1}{2}}{(1-\lambda)^\frac{1}{2}
          |\log\delta|^\frac{1}{2}}. 
        \end{align}

        	\begin{figure}
		\centering
		\begin{tikzpicture}[scale=2]
			\begin{scope}
				\clip (3.5,1)--(6,1) -- (6,-1) -- (3.5,-1);
				\draw plot [smooth] coordinates {(2.5,3)(4,.3)(5,.8)(5.4,.2)(4.7,-.8)(4,-.2)(2.5,-3)};
				\draw[densely dashed] (4.1,.245)--(4.1,-.14);
			\end{scope}
				\node[circle,fill=black, inner sep=1pt,label=above:{$\gamma(t)$}] at (4.1,.245) {};
				\node[circle,fill=black, inner sep=1pt,label=below:{$\gamma(\tilde t)$}] at (4.1,-.14) {};
				\node at (3.6,0.05) {$\left[\gamma(t), \gamma(\tilde t) \right]$};
				\node at (4.7,0) {$F_2$};
				\node at (3.5,1.15) {$\partial \Omega$};
		\end{tikzpicture}
		\caption{\label{fig:F_2} Sketch of $F_2$, the interior of the Jordan curve obtained by concatenating $\left[\gamma(t), \gamma(\tilde t) \right]$ with the shorter part of $\partial \Omega$ between $\gamma(t)$ and $\gamma(\tilde t)$.}
              \end{figure}
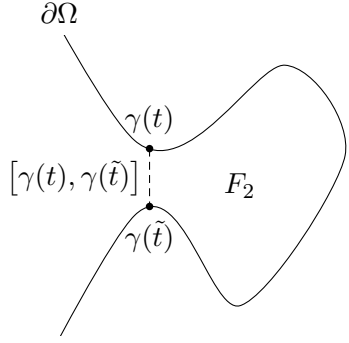

        By \eqref{eq:gammaI} and convexity of
        $\ball{0}{1-\eps(\delta)}$, we either have
        $ \ball{0}{1-\eps(\delta)} \cap F_2 = \emptyset $ or
        $\ball{0}{1-\eps(\delta)} \subset F_2$.  The latter would
        imply $P(F_2) \geq 2\pi (1-\eps(\delta)) > \frac{7}{4}\pi$ for
        $\eps(\delta) < \frac{1}{8}$, which can be taken to hold under
        our assumptions on $\delta$.  However, this would contradict
	\begin{align}
          P(F_2) = I + |\gamma(t) -\gamma(\tilde t)| \leq
          \frac{P(\Omega)}{2} + 2  \sqrt{2} \, r \leq \frac{3}{2}\pi +
          \sqrt{2}(1-\lambda)r_0 < \frac{7}{4}\pi, 
	\end{align}
	obtained with help of estimate \eqref{length_of_cut}, the
        assumption $d(t) <2\sqrt{2} \, r$ and estimate
        \eqref{bound_perimeter} provided
        $r_0 < \frac{\pi}{4\sqrt{2}}$. As a result, we get
      $ F_2 \subset \Sph^1 + \ball{0}{\eps(\delta)}$ and, in
      particular,
	\begin{align}
          |F_2| \leq
          \frac{C\lambda^\frac{1}{2}}{(1-\lambda)^\frac{1}{2}
          |\log\delta|^\frac{1}{2}}\leq \sigma_0 (1-\lambda)^2,
	\end{align}
	provided
        $|\log \delta|^{-1} \leq \sigma
        \frac{(1-\lambda)^{5}}{\lambda}$ for $\sigma>0$ small enough
        and where $\sigma_0>0$ is the constant of the non-optimality
        Lemma \ref{lem:non-optimality}.
        
        Proposition
        \ref{prop:a_priori_lower_bound} (as usual we choose $|F_2|$
        small enough such that the logarithm is positive), estimate
        \eqref{length_of_cut} and the assumption
        $|\gamma(\tilde t) - \gamma(t)| = d(t) \leq 2 \sqrt{2} \, r$
        furthermore imply
	\begin{align}
          \begin{split}
            \frac{1}{2} E(F_2) \geq \frac{1}{2}(1-\lambda) P(F_2) &=
            \frac12 (1-\lambda)(d(t) + |I|) \\
            &\geq \frac{1}{2}(1-\lambda) | I | \geq K r \geq 2 |
            \gamma(\tilde t) - \gamma(t) |
          \end{split}
        \end{align}
	under the condition that $K \geq 4\sqrt{2}$.  If
        $(\gamma( t), \gamma(\tilde t)) := \gamma(t) +
        (0,1)(\gamma(\tilde t) -\gamma(t)) \subset \Omega$, as in
        Figure \ref{fig:Z}, we can therefore use Lemma
        \ref{lem:non-optimality} with $F_1:= \Omega\setminus F_2$ and
        $\Sigma = 2 | \gamma(\tilde t) - \gamma(t) |$ to deduce that
        $\Omega$ cannot have been the minimizer.  If we had
        $(\gamma(\tilde t), \gamma(t)) \subset \stcomp{\Omega}$ we
        instead use $F_1:= \Omega \cup F_2$.  This exhausts all
        possible cases, since we have already seen above that
        $(\gamma( t), \gamma(\tilde t)) \cap \partial \Omega =
        \emptyset$.  This concludes the proof of the claim.

	\vspace{1mm}
        \textit{Step 5: Conclusion.}\vspace{1mm}\\
	We can now iterate Proposition
        \ref{prop:improvement_of_flatness} and the result of Step 4
        until we know that the minimizers have uniform
        $C^2$-boundaries at scale $r=(1-\lambda)r_0$ for a universal
        $r_0>0$.  Then the estimate of Proposition
        \ref{prop:improvement_of_flatness} and the estimate
        \eqref{remainder} combine to
	\begin{align}
          \osc \kappa \leq \frac{C \lambda}{(1-\lambda)^2|\log\delta|}
	\end{align}
	for a universal $C>0$.  As there exists $x \in \partial
        \Omega$ such that $\kappa(x) \geq \frac{2}{3}$, see Step 2, we
        have $\kappa >0$ on $\partial \Omega$ under the assumption
        $|\log\delta|^{-1} \leq \sigma \frac{(1-\lambda)^{5}}{\lambda}$.
        Therefore $\Omega$ is convex.
\end{proof}

The final ingredient of the proof of Theorem
\ref{thm:characterization_minimizers} is a quantitative rigidity
result in the $C^1$-topology for the circle in terms of the
oscillation of the curvature.  We note that in general space
dimensions this is the question of stability for the famous
Aleksandrov's soap bubble theorem \cite{aleksandrov58}, stating that
the only compact embedded hypersurfaces of constant mean curvature are
spheres. As such the problem has attracted a significant amount of
attention in recent years.  We only single out a few contributions:
Estimates controlling the distance to a circle in a $C^1$-sense are
\cite[Theorem 1.1]{ciraolo17} and \cite[Corollary 1.2]{ciraolo18}.
However, these statements require control of $(\osc \kappa)^\alpha$
for some (explicitly given) exponent $\alpha<1$, which turn out to be
suboptimal at least in two space dimensions.  On the other hand,
\cite[Theorem 1.1]{ciraolo18} and \cite[Theorem 4.3]{magnanini17}
provide a $C^0$-estimate in terms of $\osc \kappa$, which, however, is
still not enough for our purposes.  We therefore give our own version,
the proof of which turns out to be an elementary calculation.

\begin{lemma}\label{lemma:convex_sets}
  Let $P > 0$ and let
  $\gamma :[-\frac{1}{2}P, \frac{1}{2}P] \to \R^2 $ be a $C^2$-regular
  Jordan curve, parametrized by arc length and oriented such that
  $\ddot\gamma (t) = \kappa(t) \dot\gamma^\perp(t)$.  Let
  $R= \frac{P}{2\pi}$ be the radius of the circle whose length equals
  the length of $\gamma$.  Then there exists $t_0 \in [0,P)$ and
  $x_0 \in \R^2$ such that for all
  $t \in [-\frac{1}{2}P, \frac{1}{2}P]$ we have
	\begin{align}
          \left| \gamma(t) - x_0 - R\begin{pmatrix}
              \cos(R^{-1}(t-t_0)) \\ \sin(R^{-1}(t-t_0))
		\end{pmatrix}\right| &  \leq \frac{P^2}{4}\osc
                                       \kappa, \label{parametrization_close}\\ 
          \left| \dot\gamma(t) - \begin{pmatrix}\label{tangents_close}
              -\sin(R^{-1}(t-t_0)) \\ \cos(R^{-1}(t-t_0))\end{pmatrix}
          \right| &  \leq \frac{P}{2} \osc \kappa. 
	\end{align}
\end{lemma}

\begin{proof}
  By the Gauss-Bonnet theorem, we have
  $\int_{-\frac{P}{2}}^{\frac{P}{2}} \kappa(\gamma(t)) \intd t =2
  \pi$.  Therefore the mean value theorem implies that there exists
  $t_1 \in [-\frac{1}{2}P, \frac{1}{2}P]$ such that
  $\kappa(t_1) = \frac{2\pi}{P} = R^{-1}$.  In particular, we obtain
  $\|\kappa - R^{-1}\|_\infty \leq \osc \kappa$.

  Without loss of generality we may assume
  $ \dot\gamma (0) = \begin{pmatrix} 0 \\ 1
\end{pmatrix}$, which will turn out to fix the choice $t_0=0$.
	We abbreviate
		\begin{align}
			f(t) := \gamma (t) -
				R\begin{pmatrix}
					 \cos(R^{-1}t) \\
					\sin(R^{-1}t)
			\end{pmatrix}
		\end{align}
	and calculate
	\begin{align}
          \dot f(t) = \dot \gamma (t) -
				\begin{pmatrix}
                                  - \sin(R^{-1}t) \\
                                  \cos(R^{-1}t)
			\end{pmatrix}
		\end{align}
	and
	\begin{align}
          \ddot f(t) = \kappa(t) \dot \gamma^\perp(t) +
          R^{-1}\begin{pmatrix} 
            \cos(R^{-1}t) \\
            \sin(R^{-1}t)
          \end{pmatrix} = R^{-1} \dot f^\perp(t) + (\kappa(t)
          -R^{-1}) \dot \gamma^\perp(t).
	\end{align}
	
	Therefore we see
	\begin{align}\label{unit_curvature}
          \left| \frac{\intd}{\intd t} \left(\dot f^2(t) \right)
          \right| = \left| 2 (\kappa(t) 
          -R^{-1}) \dot \gamma^\perp(t) \cdot \dot f(t)
          \right| \leq 2  |\dot f(t)| 
          \osc\kappa, 
	\end{align}
	which together with the fact that the Lipschitz-function
        $t\mapsto |\dot f(t)|$ is differentiable almost
        everywhere implies
	\begin{align}
          \left\| \frac{\intd}{\intd t}|\dot f| \right\|_\infty
          \leq    \osc\kappa. 
	\end{align}
	The choice $\dot f (0)=0$ thus implies
	\begin{align}
          \|\dot f\|_\infty \leq \frac{P}{2} \osc\kappa.
	\end{align}
	Integrating once more, we see that
	 there exists a constant $x_0 \in \R^2$ such that
	\begin{align}
		\left| \gamma(t) -x_0 - R\begin{pmatrix}
		\cos(R^{-1}t) \\ \sin(R^{-1}t)
              \end{pmatrix} \right| \leq \frac{P^2}{4} \osc\kappa
	\end{align}
	for all $t \in \left[-\frac{1}{2}P,\frac{1}{2}P\right]$.
\end{proof}

We are now in a position to prove that minimizers are disks.

\begin{proof}[Proof of Theorem \ref{thm:characterization_minimizers}]
  The existence statement of the theorem is Corollary
  \ref{cor:exist_subcr}. The strategy to prove uniqueness is to argue
  that the remainder term in Proposition
  \ref{prop:improvement_of_flatness} at scale $r=(1-\lambda)r_0$ for a
  universal $r_0>0$ is itself controlled by $\osc \kappa$.  Let
  $\gamma : \left[-\frac{1}{2}P(\Omega),\frac{1}{2}P(\Omega)\right]
  \to \R^2$ be the arc length parametrization of $\partial \Omega$
  oriented as in Lemma \ref{lemma:convex_sets}, and recall that by
  Lemma \ref{lem:nearly_iso} the perimeter $P$ is bounded by a
  universal constant for $\sigma$ sufficiently small.  Furthermore,
  let $R=\frac{P}{2\pi}$ and
  \begin{align}
    \label{eq:gamma0}
    \gamma_0 (t) := x_0 + R \begin{pmatrix} \cos(R^{-1}(t - t_0)) \\
      \sin(R^{-1}(t - t_0))
  \end{pmatrix}
  \end{align}
  be the circle constructed in Lemma \ref{lemma:convex_sets}.  After a
  suitable translation, we may suppose $x_0=0$.
	
  For $t\in \left[-\frac{1}{2}P(\Omega),\frac{1}{2}P(\Omega)\right]$
  fixed and
  $A:=\left\{s \in
    \left[-\frac{1}{2}P(\Omega),\frac{1}{2}P(\Omega)\right]: \gamma(s)
    \not \in \ball{\gamma(t)}{r}\right\}$ we then have
	\begin{align}
          \int_{\partial\Omega \setminus \ball{\gamma(t)}{r}} \nabla
          \Phi_\delta (|y-\gamma(t)|)  \cdot \nu (y)  \intd \Hd^{1} (y)
          = \int_{A}  \nabla \Phi_\delta (|\gamma(s)-\gamma(t)|) \cdot
          \dot \gamma^\perp (s) \intd s.    
	\end{align}
	As a result of Lemma \ref{lemma:convex_sets}, and the
        estimates $|\nabla \Phi_{\delta}(|z|)| \leq r^{-2}$ and
        $| D^2 \Phi_\delta(|z|)| \leq C r^{-3}$ for all
        $|z| > \frac{r}{2}$ and $C>1$ universal, we have
	\begin{align}
          & \quad \Bigg| \int_{A}  \nabla \Phi_\delta (|\gamma(s) -
            \gamma(t)|) \cdot \dot \gamma^\perp (s) \intd s -
            \int_{A} \nabla \Phi_\delta (|\gamma_0(s)-\gamma_0(t)|)
            \cdot \dot \gamma_0^\perp (s) \intd s \Bigg| \notag\\
          & \leq \int_A \left| \nabla \Phi_\delta(|\gamma (s) -
            \gamma(t)|) \cdot (\dot \gamma^\perp(s) - \dot
            \gamma_0^\perp(s)) \right| \intd s \notag \\
          & \quad + \int_A \left| \big(\nabla \Phi_\delta(|\gamma
            (s) - \gamma(t)|) - \nabla \Phi_\delta(|\gamma_0 (s)
            - \gamma_0(t)|) \big) \cdot \dot
            \gamma_0^\perp(s) \right| \intd s \\
          & \leq C r^{-2} \osc \kappa\notag\\
          & \quad  + \int_A \int_0^1 \left|D^2
            \Phi_\delta\big(\theta(\gamma(s) - \gamma(t)) +
            (1-\theta)(\gamma_0 (s) 
            - \gamma_0(t)\big) \right|(|
            \gamma(s) - \gamma_0 (s) | + |\gamma(t) - \gamma_0(t) |)
            \intd \theta  \intd s \notag\\  
          & \leq C (r^{-2} + C r^{-3}) \osc \kappa
            \leq 2 C^2 r_0^{-3} (1-\lambda)^{-3} \osc \kappa.\notag
	\end{align}
	
	Next, let
	\begin{align}
          A_0& := \left\{s \in
               \left[-\tfrac{1}{2}P(\Omega),\tfrac{1}{2}P(\Omega)\right]:
               \gamma_0(s) \not \in \ball{\gamma_0(t)}{r}\right\},\\  
          A_1& :=\{s\in
                    \left[-\tfrac{1}{2}P(\Omega),\tfrac{1}{2}P(\Omega)\right]
                    : |\gamma(s)-\gamma(t)| < r, |\gamma_0(s)- \gamma_0(t)|
                    \geq r \}, \\ 
          A_2&:=\{s\in
                    \left[-\tfrac{1}{2}P(\Omega),\tfrac{1}{2}P(\Omega)\right]
                    : |\gamma(s)-\gamma(t)| \geq r, |\gamma_0(s)-
                    \gamma_0(t)| < r \}, 
	\end{align}
	so that we have $A \Delta A_0 = A_1 \cup A_2$.  For
        $s\in A_1$ we obtain by the estimate
        \eqref{parametrization_close}, the bound
        \eqref{estimate_oscillation} and our assumption on $\delta$
        that
	\begin{align}
		r\leq |\gamma_0(s)-\gamma_0(t)| \leq
          |\gamma(s)-\gamma(t)|  + |\gamma_0(s)-\gamma(s)|  +
          |\gamma_0(t)-\gamma(t)| \leq r + c\osc \kappa \leq 2 r.
	\end{align}
        Therefore, we have
        $\gamma_0(A_1) \subset \ball{\gamma_0(t)}{r+c\osc\kappa}
        \setminus \ball{\gamma_0(t)}{r}$ and, hence,
        $|A_1| \leq C \osc \kappa$ if $r_0$ is chosen small
        enough universal. Similarly, for $s \in A_2$ we have for
          $C > 0$ universal
          \begin{align}
            \tfrac12 r \leq r - c \osc \kappa \leq |\gamma(s) -
            \gamma(t)| - |\gamma(s) - \gamma_0(s)| - |\gamma(t) -
            \gamma_0(t)| \leq |\gamma_0(s) - \gamma_0(t)| \leq r,             
          \end{align}
          so that
          $\gamma_0(A_2) \subset B_r(\gamma_0(t)) \setminus B_{r
            - c \osc \kappa}(\gamma_0(t_0))$ and, hence,
          $|A_2| \leq C \osc \kappa$ as well. Combining this
        observation with the bound
        $|\nabla \Phi_{\delta}(|z|)| \leq |z|^{-2}$ gives
	\begin{align}
          \label{eq:A1}
          \left| \int_{A_1 \cup A_2}  \nabla \Phi_\delta
          (|\gamma_0(s)-\gamma_0(t)|)  \cdot \dot \gamma_0^\perp (s)
          \intd s \right| \leq C r^{-2} \osc \kappa  =
          C r_0^{-2} (1-\lambda)^{-2} \osc \kappa.  
	\end{align}	
        Thus we have
	\begin{align}
             \Bigg| \int_{A}  \nabla \Phi_\delta
             (|\gamma_0(s)-\gamma_0(t)|) \cdot  \dot \gamma_0^\perp
             (s)\intd s   - \int_{A_0}  \nabla \Phi_\delta
             (|\gamma_0(s)-\gamma_0(t)|) \cdot \dot \gamma_0^\perp (s)
             \intd s \Bigg| 
             \leq C(1-\lambda)^{-2} \osc \kappa.
	\end{align}

	Recalling that $P(\Omega) = 2\pi R$, we may write the integral
        below in a more geometric way:
	\begin{align}
          \int_{ A_0}  \nabla \Phi_\delta (|\gamma_0(s)-\gamma_0(t)|)
          \cdot \dot \gamma_0^\perp (s) \intd s  =
          \int_{R\,\Sph^1\setminus\ball{\gamma_0(t)}{r}} \nabla
          \Phi_\delta (|y-\gamma_0(t)|)  \cdot \nu_0 (y)  \intd \Hd^{1}
          (y), 
	\end{align}
	where $\nu_0(y)$ is the outer unit normal of $R\, \Sph^1$ at $y$.
	Therefore, the above estimates imply
	\begin{align}
	   \begin{split}
             \Bigg| \int_{\partial\Omega \setminus
               \ball{\gamma(t)}{r}} \nabla \Phi_\delta (|y-\gamma(t)|)
             \cdot \nu (y)  \intd \Hd^{1} (y)   -
             \int_{R\,\Sph^1\setminus\ball{\gamma_0(t)}{r}} \nabla
             \Phi_\delta (|y-&\gamma_0(t)|)  \cdot \nu_0 (y)  \intd
             \Hd^{1} (y) \Bigg| \\ 
             & \leq C(1-\lambda)^{-3} \osc \kappa.
	  \end{split}
	\end{align}
	
	Finally, note that the integral over
        $R\,\Sph^1 \setminus B_r(\gamma_0(t))$ is independent of
        $t$ by symmetry of the circle, so that it can play the role of
        the constant $a$ in the estimate of Proposition
        \ref{prop:improvement_of_flatness}.  As a result, we get
	\begin{align}
          \osc \kappa  \leq \frac{C\lambda}{(1-\lambda)^4|\log\delta|}
          \osc\kappa, 
	\end{align}
	which implies $\osc \kappa = 0$ under the conditions on
        $\delta$.  Another application of Lemma
        \ref{lemma:convex_sets} shows that $\Omega$ must be a disk and
        the volume constraint $|\Omega|=\pi$ then gives that $\Omega$
        has radius one.
\end{proof}

\section{The critical case $\lambda =1$}\label{sec:critical}
We finally turn to analyzing the critical case.  The first step is to
prove that the pointwise limit of $|\log\delta| E_{1,\delta}$ exists
and coincides with the $L^1$-$\Gamma$-limit.  To this end, we observe
that the inequality used in Step 1 in the proof of Proposition
\ref{prop:a_priori_lower_bound} allows us to decouple the dependence
of the energy $|\log\delta| E_{1,\delta}$ and the sets $\Omega_\delta$
on $\delta$.  In order to identify the limit, we re-use Step 1 of the
proof of Lemma \ref{lem:energy_on_boundary}.

\begin{proof}[Proof of Theorem \ref{thm:critical_convergence}]
  \textit{Step 1: Compactness for sequences of uniformly bounded sets
    with finite energy.}\vspace{1mm}\\ 
  Let $\Omega_{\delta_n} \subset \ball{0}{R}$ for some $R>0$ be a
  sequence such that
	\begin{align}
          \limsup_{n\to \infty} |\log\delta_n|
          E_{\lambda,\delta_n}(\Omega_{\delta_n}) \leq M < \infty 
	\end{align}
	for $M>0$ and with $\delta_n \to \infty$ as $n\to \infty$.  If
        we had $ P(\Omega_{\delta_n}) > \frac{\pi m}{\delta_n}$ along
        some subsequence (not relabeled), then Proposition
        \ref{prop:a_priori_lower_bound} would imply along a further
        subsequence that
	\begin{align}
          \lim_{n\to \infty} |\log\delta_n| P(\Omega_{\delta_n}) \leq
          2 \limsup_{n\to \infty}  |\log\delta_n|
          E_{\lambda,\delta_n}(\Omega_{\delta_n}) \leq 2 M. 
	\end{align}
	However, this gives the contradiction $P(\Omega_{\delta_n})
        \to 0$ as $n\to \infty$, as
          $\Omega_{\delta_n} \in \mathcal A_m$.  Therefore, we have
        $ P(\Omega_{\delta_n}) \leq \frac{\pi m}{\delta_n}$ for all
        $n \in \N$ large enough.  Proposition
        \ref{prop:a_priori_lower_bound} then gives
	\begin{align}
          \limsup_{n\to \infty}\, \log\left(
          \frac{P(\Omega_{\delta_n})}{e\pi m} \right)
          P(\Omega_{\delta_n}) \leq \limsup_{n\to \infty}
          |\log\delta_n| E_{\lambda,\delta_n}(\Omega_{\delta_n}) \leq
          M. 
	\end{align}
	As a result of $x \log\left(\frac{x}{e\pi m} \right) \to
        \infty$ as $x \to \infty$, we get a bound
        $\limsup_{n\to \infty} P(\Omega_{\delta_n}) \leq \widetilde M$
        for some $\widetilde M<\infty$.  As we have the assumption
        $\Omega_{\delta_n} \subset \ball{0}{R}$ we may apply the
        compact embedding theorem for $BV$-functions \cite[Theorem
        12.26]{maggi} to obtain a set $\Omega \subset \mathcal A_m$
        such that $|\Omega_{\delta_n} \Delta \Omega| \to 0$ as
        $n\to \infty$ along some subsequence.

\vspace{1mm}
  \textit{Step 2: Monotonicity for fixed sets.}\vspace{1mm}\\
  In the first step of the proof of Proposition
  \ref{prop:a_priori_lower_bound} we proved
	\begin{align}
          \int_{\R^2} \int_{\ball{0}{\tilde \delta}} |
          \chi_\Omega(x+z) - \chi_\Omega(x)|^2
          \frac{g_\delta(|z|)}{|z|^3} \intd z \intd x \leq 4 \left(
          |\log\delta| - \left|\log\tilde \delta\right| \right)
          P(\Omega) 
	\end{align}
	for $\delta < \tilde \delta<\frac12$, which is equivalent
        to
	\begin{align}
          \left|\log\tilde \delta\right| E_{1, \tilde \delta}(\Omega)
          \leq |\log\delta| E_{1,\delta}(\Omega).  
	\end{align}
	In particular, the pointwise limit
        $\widetilde E_{1,0}(\Omega) := \lim_{\delta \to 0}|\log\delta|
        E_{1,\delta}(\Omega)$ exists in $\R \cup \{+\infty\}$ for
        every $\Omega \in \mathcal A_m$.
	
	\vspace{1mm} \textit{Step 3: The $\Gamma$-limit is given by
          $\widetilde
          E_{1,0}$.}\vspace{1mm}\\
	We start with proving the lower bound.  Let
        $\Omega_{\delta_n} \in \mathcal{A}_m$ for $\delta_n \to 0$ as
        $n\to \infty$ be such that there exists
        $\Omega \in \mathcal{A}_m$ with
        $\lim_{n\to \infty} | \Omega_{\delta_n} \Delta \Omega| = 0$.
        For every $\tilde \delta >0$ we then have
	\begin{align}
          \liminf_{n \to \infty} |\log\delta_n|
          E_{1,\delta_n}(\Omega_{\delta_n}) \geq \liminf_{n \to \infty}
          |\log \tilde \delta| E_{1,\tilde \delta}(\Omega_{\delta_n})
          \geq  |\log \tilde \delta| E_{1,\tilde \delta}(\Omega) 
	\end{align}
	by lower semi-continuity of the perimeter and continuity of
        the nonlocal term, see estimate
        \eqref{continuity_of_nonlocal_term}.  We obtain the desired
        statement in the limit $\tilde \delta \to 0$, i.e., we have
	\begin{align}
          \liminf_{n\to \infty} |\log\delta_n|
          E_{1,\delta_n}(\Omega_{\delta_n})  \geq \lim_{\tilde \delta \to 0}
          \left|\log\tilde \delta\right| E_{1,\tilde \delta}(\Omega) =
          \widetilde E_{1,0}(\Omega). 
	\end{align}
	The $\Gamma$-$\limsup$ inequality is an immediate consequence
        of
        $\widetilde E_{1,0}$ being a pointwise limit.
        
	\vspace{1mm}
        \textit{Step 4: Prove $\widetilde E_{1,0}(\Omega) =
          E_{1,0}(\Omega)$ for all $\Omega \subset
          \mathcal{A}_m$}.\vspace{1mm}\\ 
	We start with equation \eqref{fubini?}, which gives
	\begin{align}\label{representation_parts}
	  \begin{split}
            \int_{\Omega}
            \int_{\stcomp{\Omega}}\frac{g_\delta(|x-y|)}{|x-y|^3}
            \intd y \intd x & = \int_{\partial^* \Omega} \int_{\Omega}
            \nu(y) \cdot
            \nabla_x \Phi_{\delta}(|y-x|)\intd x \intd \Hd^1(y)\\
            & = - \frac{1}{\delta } \int_{\partial^* \Omega}
            \int_{\Omega\cap\ball{y}{\delta}} \nu(y) \cdot
            \frac{x-y}{|x-y|^2} \intd x \intd \Hd^1(y)\\
            & \quad - \int_{\partial^* \Omega}
            \int_{\Omega\setminus\ball{y}{\delta}} \nu(y) \cdot
            \frac{x-y}{|x-y|^3} \intd x \intd \Hd^1(y)
	  \end{split}
	\end{align}
	by the representation \eqref{Phi_explicit}.  In order to treat
        the second term, for $y\in \partial^\ast\Omega$ we compute
	\begin{align}\label{half_plane}
          - \int_{H_-(y) \cap(\ball{y}{1}\setminus\ball{y}{\delta})}
          \nu(y) \cdot \frac{x-y}{|x-y|^3} \intd x
          & = \int_{\{x_1>0\}
            \cap
            (\ball{0}{1}\setminus
            \ball{0}{\delta})}
            \frac{x_1}{|x|^3}
            \intd x =  2
            |\log\delta|. 
	\end{align}
	Combining equations \eqref{representation_parts} and
        \eqref{half_plane} we see (cf. Figure
          \ref{fig:domains_of_integration}) that
	\begin{align}
	   \begin{split}
             2 |\log\delta| E_{1,\delta}(\Omega) & =-
             \int_{\partial^\ast \Omega}\int_{H_-(y) \cap
               (\ball{y}{1}\setminus\ball{y}{\delta})} \nu(y) \cdot
             \frac{x-y}{|x-y|^3} \intd x \intd \Hd^1(y)\\
             & \quad + \frac{1}{\delta } \int_{\partial^* \Omega}
             \int_{\Omega\cap\ball{y}{\delta}} \nu(y) \cdot
             \frac{x-y}{|x-y|^2} \intd x \intd \Hd^1(y)\\
             & \quad + \int_{\partial^* \Omega}
             \int_{\Omega\setminus\ball{y}{\delta}} \nu(y) \cdot
             \frac{x-y}{|x-y|^3} \intd x \intd \Hd^1(y).
	  \end{split}
	 \end{align}
	 Splitting up the first integral into an integration over
         $\Omega$ and its complement, splitting the third integral
         into a contribution over $\ball{y}{1}$ and its complement,
         and combining the terms we see
	 \begin{align}\label{computation}
	   \begin{split}
             2 |\log\delta| E_{1,\delta}(\Omega) & =
             \frac{1}{\delta } \int_{\partial^* \Omega}
             \int_{\Omega\cap\ball{y}{\delta}} \nu(y) \cdot
             \frac{x-y}{|x-y|^2} \intd x \intd \Hd^1(y)\\
             & \quad + \int_{\partial^* \Omega} \int_{\big(\Omega
               \setminus H_-(y) \big)\cap
               \big(\ball{y}{1}\setminus\ball{y}{\delta}\big)} \nu(y)
             \cdot \frac{x-y}{|x-y|^3} \intd x \intd \Hd^1(y)\\
             & \quad - \int_{\partial^* \Omega}
             \int_{\big((H_-(y)\setminus \Omega\big) \cap
               \big(\ball{y}{1}\setminus\ball{y}{\delta}\big)}\nu(y)
             \cdot \frac{x-y}{|y-x|^3} \intd x \intd \Hd^1(y)\\
             & \quad + \int_{\partial^* \Omega} \int_{\Omega \setminus
               \ball{y}{1} }\nu(y) \cdot \frac{y-x}{|y-x|^3}
             \intd x \intd \Hd^1(y)\\
             & = \frac{1}{\delta } \int_{\partial^* \Omega}
             \int_{\Omega\cap\ball{y}{\delta}} \nu(y) \cdot
             \frac{x-y}{|x-y|^2} \intd x \intd \Hd^1(y)\\
             & \quad + \int_{\partial^* \Omega}
             \int_{\big(\Omega\Delta
               H_-(y)\big)\cap\big(\ball{y}{1}\cap
               \ball{y}{\delta}\big) }\left|\nu(y) \cdot
               \frac{x-y}{|x-y|^3}\right| \intd x \intd \Hd^1(y)\\
             & \quad + \int_{\partial^* \Omega} \int_{\Omega \setminus
               \ball{y}{1} }\nu(y) \cdot \frac{y-x}{|y-x|^3} \intd x
             \intd \Hd^1(y).
	  \end{split}
	\end{align}
	
	If we assume $y \in \partial^* \Omega$ and $\nu(y)=e_1$,
        then we can rewrite
	\begin{align}
          \frac{1}{\delta } \int_{\Omega\cap\ball{y}{\delta}} \nu(y)
          \cdot \frac{x-y}{|x-y|^2} \intd x  =
          \int_{\delta^{-1}(\Omega-y) \cap \ball{0}{1}
          }\frac{x_1}{|x|^2} \intd x. 
	\end{align}
	By the blow-up properties for sets of finite perimeter
        \cite[Theorem 15.5]{maggi}, see also \cite[Chapter
        12.1]{maggi}, we get that
        $\chi(\delta^{-1}(\Omega-y)) \to \chi(\nu(y)\cdot x<0)$
        pointwise almost everywhere after passage to a subsequence
        $\delta_n$ that depends on $y$.  As the map
        $x \mapsto \frac{1}{|x|}$ is integrable on $\ball{0}{1}$ and
        provides a uniform majorant, Lebesgue's dominated convergence
        theorem implies
	\begin{align}
          \lim_{n \to \infty} \int_{\delta_n^{-1}(\Omega-y) \cap
          \ball{0}{1} }\frac{x_1}{|x|^2} \intd x =  \int_{\{x_1<0\}
          \cap \ball{0}{1} }\frac{x_1}{|x|^2} \intd x = - 2.
	\end{align}
	Furthermore, since the limit is independent of the
        subsequence we obtain
	\begin{align}\label{local_convergence}
          \lim_{\delta \to 0} \frac{1}{\delta } \int_{\partial^*
          \Omega} \int_{\Omega\cap\ball{y}{\delta}} \nu(y) \cdot
          \frac{x-y}{|x-y|^2} \intd x \intd \Hd^1(y)  = -2P(\Omega). 
	\end{align}
	
	The second term on the right hand side of \eqref{computation}
        converges by the monotone convergence theorem and we get
	\begin{align}
          \lim_{\delta \to 0} |\log\delta| E_{1,\delta}(\Omega)    =
          -P(\Omega)
          & + \frac12 \int_{\partial^* \Omega}
            \int_{\left(\Omega\Delta
            H_-(y)\right)\cap\ball{y}{1} }\left|\nu(y)
            \cdot \frac{x-y}{|x-y|^3}\right| \intd x \intd
            \Hd^1(y)\\ 
          & \quad +  \frac12 \int_{\partial^* \Omega}
            \int_{\Omega \setminus 
            \ball{y}{1} }\nu(y) \cdot \frac{y-x}{|y-x|^3} \intd x
            \intd \Hd^1(y). 
            \qedhere
	\end{align}
\end{proof}

As a next step we re-derive the representation of the $\Gamma$-limit
used by Bernoff and Kent-Dobias \cite{kent-dobias15}.

\begin{proof}[Proof of Proposition \ref{prop:representations_limit}]
	Here, we use the representation
	\begin{align}
          |\log\delta| E_{1,\delta}(\Omega) = |\log\delta| P(\Omega) -
          \frac{1}{2} \int_{\partial \Omega} \int_{\partial \Omega}
          \nu(x) \cdot \nu(y) \Phi_{\delta}(|x-y|) \intd \Hd^1(y)
          \intd \Hd^1(x) 
	\end{align}
	of Lemma \ref{lem:energy_on_boundary}, which is valid for sets
        satisfying the mild regularity assumption
        \eqref{mild_regularity}, which is surely applicable for
          sets with $C^2$-regular boundary.  For every
        $x \in \partial \Omega$ we have by \eqref{Phi_explicit} that
	\begin{align}
	  \begin{split}
            & \int_{\partial \Omega} \int_{\partial \Omega \cap
              \ball{x}{\delta}} \nu(x) \cdot \nu(y)
            \Phi_{\delta}(|x-y|) \intd \Hd^1(y) \intd \Hd^1(x) \\
            & \qquad = \frac{1}{\delta} \int_{\partial \Omega}
            \int_{\partial \Omega \cap \ball{x}{\delta}} \nu(x) \cdot
            \nu(y) \left(1-\log\left(\frac{|x-y|}{\delta} \right)
            \right) \intd \Hd^1(y) \intd \Hd^1(x).
	   \end{split}
	\end{align}
	Assuming $x=0$ without loss of generality, the inner
        integral becomes
	\begin{align}
	   \begin{split}
             \int_{\partial( \delta^{-1} \Omega )\cap \ball{0}{1}}
             \nu(x) \cdot \nu(y) \left(1-\log\left(|y| \right) \right)
             \intd \Hd^1(y) & = \int_{-1}^1 \left(1-\log\left(|t|
               \right) \right) \intd t + o(1) = 4 + o(1)
	  \end{split}
	\end{align}
	with an error term that is uniform in $\delta$, where we used
        the calculation in \eqref{logarithm}.  As a result we get
	\begin{align}\label{representation_hand}
          \begin{split}
            E_{1,0}(\Omega) = & - 2 P(\Omega) \\
            & + \lim_{\delta \to 0} \left( |\log \delta| P(\Omega) -
              \frac{1}{2}\int_{\partial \Omega} \int_{\partial \Omega}
              \frac{\nu(x) \cdot \nu(y)}{|x-y|}
              \chi(|x-y|>\delta)\intd \Hd^1(y) \intd \Hd^1(x) \right).
          \end{split}
	\end{align}
	
	As $\partial \Omega$ is $C^2$-regular we can decompose it into
        closed, positively oriented Jordan curves
        $\gamma_i : [0,P_i] \to \R^2$ parametrized by arclength with
        $1\leq i \leq N$ and $i,N \in \N$.  Computing
	\begin{align}
          \frac{1}{2} \int_{-\frac{P_i}{2}}^{\frac{P_i}{2}}
          \frac{\chi(|s| >\delta)}{|s|} \intd s =
          \log\left(\frac{P_i}{2} \right)  +  |\log\delta| 
	\end{align}
	and observing that $P(\Omega) = \sum_{i=1}^N P_i$ we have
	\begin{align}
	  \begin{split}
            (|\log \delta| -2 )P(\Omega) = - \sum_{i=1}^N P_i
            \left[\log\left(\frac{P_i}{2} \right)+2 \right] +
            \frac{1}{2} \sum_{i=1}^N \int_{0}^{P_i}
            \int_{-\frac{P_i}{2}}^{\frac{P_i}{2}} \frac{\chi(|s|
              >\delta)}{|s|} \intd s \intd t.
	  \end{split}
	\end{align}
	Consequently, we obtain
	\begin{align}\label{jordan_decomposition}
          E_{1,0}(\Omega) = \lim_{\delta \to 0} \Bigg\{
          &- \sum_{i=1}^N P_i \left[ \log\left(\frac{P_i}{2} \right)+2
            \right] \notag \\ 
          & + \frac{1}{2} \sum_{i=1}^N \int_{0}^{P_i}
            \int_{-\frac{P_i}{2}}^{\frac{P_i}{2}} \left( \frac{\chi(|s|
            >\delta)}{|s|} - \frac{\dot\gamma^\perp_i(t+s) \cdot
            \dot\gamma^\perp_i(t)}{|\gamma_i(t+s)-\gamma_i(t)|} 
            \chi(|\gamma_i(t+s)-\gamma_i(t)|>\delta) \right) \intd 
            s \intd t\notag\\ 
          & - \sum_{i=1}^{N-1} {\sum_{j=i+1}^N}
            \int_{0}^{P_i}\int_{0}^{P_j} 	\frac{\dot\gamma^\perp_j(s) \cdot
            \dot\gamma^\perp_i(t)}{|\gamma_j(s)-\gamma_i(t)|}
            \chi(|\gamma_j(s)-\gamma_i(t)|>\delta)\intd s \intd t
            \Bigg\}.
	\end{align}
	
	Expanding $\gamma_i$ and $\dot\gamma^\perp_i$ in $s$, see
        Kent-Dobias \cite{kent14thesis} for the details, one can
        obtain that the first two lines satisfy
	\begin{align}
            & \lim_{\delta \to 0} \Bigg\{ - \sum_{i=1}^N P_i \left[
              \log\left(\frac{P_i}{2} \right)+2
            \right] \notag \\
            & + \frac{1}{2} \sum_{i=1}^N \int_{0}^{P_i}
            \int_{-\frac{P_i}{2}}^{\frac{P_i}{2}} \left(
              \frac{\chi(|s| >\delta)}{|s|} -
              \frac{\dot\gamma^\perp_i(t+s) \cdot
                \dot\gamma^\perp_i(t)}{|\gamma_i(t+s)-\gamma_i(t)|}
              \chi(|\gamma_i(t+s)-\gamma_i(t)|>\delta) \right)
            \intd
            s \intd t \Bigg\} \notag \\
            & = - \sum_{i=1}^N P_i \left[ \log\left(\frac{P_i}{2}
              \right)+2 \right] + \frac{1}{2} \sum_{i=1}^N
            \int_{0}^{P_i} \int_{-\frac{P_i}{2}}^{\frac{P_i}{2}}
            \left( \frac{1}{|s|} - \frac{\dot\gamma^\perp_i(t+s) \cdot
                \dot\gamma^\perp_i(t)}{|\gamma_i(t+s)-\gamma_i(t)|}
            \right) \intd s \intd t.
	\end{align}
	As the Jordan curves are at a positive distance from each
        other, the limit in the third line of equation
        \eqref{jordan_decomposition} can easily carried out and gives
	\begin{align}
	  \begin{split}
            & \quad \lim_{\delta \to 0} \sum_{i=1}^{N-1}
            {\sum_{j=i+1}^N} \int_{0}^{P_i}\int_{0}^{P_j}
            \frac{\dot\gamma^\perp_j(s) \cdot
              \dot\gamma^\perp_i(t)}{|\gamma_j(s)-\gamma_i(t)|}
            \chi(|\gamma_j(s)-\gamma_i(t)|>\delta)\intd s \intd
            t\\
            & = \sum_{i=1}^{N-1} {\sum_{j=i+1}^N}
            \int_{0}^{P_i}\int_{0}^{P_j} \frac{\dot\gamma^\perp_j(s)
              \cdot \dot\gamma^\perp_i(t)}{|\gamma_j(s)-\gamma_i(t)|}
            \intd s \intd t,
	  \end{split}
	\end{align}
	which concludes the proof.
\end{proof}

For the compactness part, we exploit the control over the number of
pieces of a generalized minimizer given in Proposition
\ref{prop:existence_of_minimizers_intermediate_masses}.

\begin{proof}[Proof of Proposition \ref{prop:asymptotic_compactness}]
  \textit{Step 1: A generalized minimizer of $E_{1,\delta}$ can have
    at most $N(m)$ components}.\vspace{1mm}\\
  According to Proposition
  \ref{prop:existence_of_minimizers_intermediate_masses}, in order to
  prove that generalized minimizers of $E_{1,\delta}$ enjoy a uniform
  control over the number of their parts, we only have to see that
  $m_0(1,\delta,\infty) = \sup\{m:
  \inf_{\mathcal{A}_m}E_{\lambda,\delta} >0\}$ is strictly positive
  uniformly in $\delta$.  To this end, we use the lower bound of
  Proposition \ref{prop:a_priori_lower_bound} to get a universal
  constant $c>0$ such that for all sets of finite perimeter
  $\Omega \subset \R^2$ we have
  \begin{align}
    \label{eq:PlogP}
    |\log\delta| E_{1,\delta}(\Omega) \geq P(\Omega) \log\left(
    \frac{c P(\Omega)}{m}\right). 
  \end{align}
  Applying the isoperimetric inequality, we see that
  \begin{align}
    |\log\delta| E_{1,\delta}(\Omega) \geq c m^\frac{1}{2}
    \log\left(c m^{-\frac{1}{2}}\right) > 0 
  \end{align}
  for all $0<m<\bar m_0$, where $\bar m_0 >0$ is
  small enough and independent of $\delta$.
  
  \vspace{1mm} \textit{Step 2: We have
    $\limsup_{\delta \to 0} \sum_{i=1}^{N_\delta} P(\Omega_{i,\delta})
    \leq C(m) $ for some $0<C(m)<\infty$, where
    $(\Omega_{i,\delta})_{i\in \N}$ is a generalized minimizer of
    $E_{1,\delta}$ over $\mathcal{A}_m$ and $N_\delta \leq N(m)$ is
    the number of
    its components.}\vspace{1mm}\\
  The same computation as for Proposition
  \ref{prop:subcritical_gamma_convergence} implies that
  $ |\log\delta| E_{1,\delta}(\ball{0}{r})< \widetilde C(r)$ for some
  $0<\widetilde C(r)<\infty$.  By minimality, we therefore have
  $ \sum_{i=1}^{N_\delta} |\log\delta| E_{1,\delta}(\Omega_{i,\delta})
  < \widetilde C(m)$. Letting
  \begin{align}
  	\widetilde \Omega_\delta := \bigcup_{i=1}^{N_\delta} \left(
    \Omega_{i,\delta} + K e_1 \right)
  \end{align}
  for $K \in \N$, we obtain for $K$ large enough that
  $P(\widetilde \Omega_\delta) =
  \sum_{i=1}^{N_\delta}P(\Omega_{i,\delta})$ and
  $|\log\delta| E_{1,\delta}(\widetilde \Omega) < \widetilde C(m)$.  The
  same argument as for Step 1 in the proof of Theorem
  \ref{thm:critical_convergence} implies that
  \begin{align}
    \limsup_{\delta \to 0} \sum_{i=1}^{N_\delta} P(\Omega_{i,\delta})
    =	\limsup_{\delta \to 0}  P(\widetilde \Omega_{\delta})\leq
    C(m). 
  \end{align} 

  \vspace{1mm} \textit{Step 3: For each $i\in \N$ with
    $1\leq i\leq N(m)$ the sequence $\Omega_{i,\delta_n}$ has a
    convergent subsequence in $L^1$ after a suitable translation,
    where $(\delta_n)_{n\in \N}$ is any sequence satisfying
    $\delta_n \to 0$ as $n \to \infty$.}\vspace{1mm}\\
  After extracting a subsequence, we may suppose that
  $N_{\delta_n} \equiv N$ is stationary. In order to apply Lemma
  \ref{lem:connectedness} to obtain connectedness of
  $\Omega_{i,\delta_n}$ for all $1\leq i\leq N$, we require a lower
  bound for $|\Omega_{i,\delta_n}|$ independent of $\delta_n$.
        
        If $N>1$,
        then we proved in Step 5 of Proposition
        \ref{prop:existence_of_minimizers_intermediate_masses} that
        $|\Omega_{i,\delta_n}| \geq m_0(1,\delta,\infty) \geq \bar
          m_0$.  If instead we have $N=1$ then we simply have
        $|\Omega_{1,\delta_n}| = m$.  Therefore, by passing to another
        subsequence, we may assume that
	\begin{align}\label{choice_delta_n}
		\delta_n < \sqrt{\frac{|\Omega_{i,\delta_n}|}{2\pi}}
	\end{align}
	 for all $1\leq i\leq N$, so that Lemma
         \ref{lem:connectedness} applies. 
	
         Consequently, we have that $\Omega_{i,\delta_n}$ is connected
         for all $1\leq i\leq N$, and thus enjoys the uniform diameter
         bound
         $\operatorname{diam} \Omega_{i,\delta_n} \leq \frac{1}{2}
         P(\Omega_{i,\delta_n}) \leq C(m)$.  After translation, it is
         therefore uniformly bounded and there exists a subsequence
         (not relabeled) and $(\Omega_i)_{i\in\N}$ such that
         $|\Omega_{i,\delta} \Delta \Omega_i| \to 0$ as $\delta\to 0$
         due to the compact embedding theorem for $BV$-functions on
         compact domains \cite[Theorem 12.26]{maggi}.
        
         \vspace{1mm} \textit{Step 4: The collection of sets
           $(\Omega_i)_{i\in \N}$ is a generalized minimizer of
           $E_{1,0}$ over
           $\mathcal{A}_m$.}\vspace{1mm}\\
         First, observe that
         $\sum_{i=1}^\infty |\Omega_i| = \sum_{i=1}^N |\Omega_i| = m$.
         Let $(\widetilde \Omega_j)_{j\in \N}$ with
         $\widetilde \Omega_j \subset \R^2$ for all $j\in \N$ be
         any sets such that
         $\sum_{j=1}^\infty |\widetilde \Omega_j | = \sum_{j=1}^M
         |\widetilde \Omega_j| = m$ for some $M \in \N$.  By the
         $\Gamma$-liminf part of Theorem
         \ref{thm:critical_convergence}, minimality of
         $(\Omega_{i,\delta})_{i\in \N}$ and the fact that $E_{1,0}$
         is the pointwise limit of $|\log\delta| E_{1,\delta}$, we
         immediately obtain
      \begin{align}
      	\sum_{i=1}^N E_{1,0}(\Omega_i)
        & \leq \liminf_{\delta \to 0}
          \sum_{i=1}^N |\log \delta|
          E_{1,\delta}(\Omega_{i,\delta})
          \leq \limsup_{\delta \to 0 }
          \sum_{j=1}^M |\log\delta|
          E_{1,\delta} \left(\widetilde
          \Omega_j\right) =
          \sum_{j=1}^M  E_{1,0}
          (\widetilde
          \Omega_j),
      \end{align}
      yielding the claim.
\end{proof}

\begin{proof}[Proof of Lemma \ref{lem:modified_compactness}]
  The representation is a straightforward consequence of the
  modification being an $L^1$-continuous perturbation of
  $|\log\delta|E_{1,\delta}$.  Regarding the compactness part, the
  previous proof carries over to the $\Gamma$-limit of
  $F_{1,\delta,l}$. In order to bound the number of components of any
  generalized minimizer, we again only need to ensure that
  $\inf_{\mathcal{A}_m} F_{1,\delta,l} - \frac{6\pi}{5|\log\delta| l
  }m >0$ for $m>0$ small enough.  For all $\Omega \in \mathcal{A}_m$, 
  noting that
  $\frac{1}{2} \int_{\R^2} \frac{|z|^2+2l^2}{(|z|^2+l^2)^{5/2}} \intd
  z = \frac{4 \pi}{3 l}$ we indeed have that
	\begin{align}
	  \begin{split}
            |\log\delta|F_{1,\delta,l} (\Omega) - \frac{6\pi}{5l}m &
            \geq |\log\delta| E_{1,\delta} (\Omega) - \frac{6\pi
              m}{5l} - \frac12 \int_\Omega \int_{\stcomp{\Omega}}
            \frac{|x - y|^2 + 2 l^2}{\left( |x - y|^2 + l^2
              \right)^{5/2}} \intd y
            \intd x \\
            & \geq c m^\frac{1}{2} \log\left( cm^{-\frac{1}{2}}\right)
            - \frac{3\pi m}{l} >0
	  \end{split}
	\end{align}
	for all $0<m < \bar m_0(l)$, independently of
        $\delta$.
	
	The argument for the perimeter of a generalized minimizer can
        easily be adjusted using the bound
        $|\log\delta|F_{1,\delta,l} (\widetilde \Omega) \leq
        |\log\delta| E_{1,\delta,} (\widetilde \Omega) + \frac{4\pi
          m}{3l}$ for all $\widetilde \Omega \in \mathcal{A}_m$.
	The rest of the argument carries over with the only difference
        that the application of Lemma \ref{lem:connectedness} is
        immediate in the case $l<\infty$.
\end{proof}

Finally, we come to proving existence of non-radial minimizers.
The proof relies heavily on some tedious calculations, of which we
delegate most to {\sc Mathematica 11.1.1.0} software.

\begin{proof}[Proof of Theorem \ref{thm:non-spherical_minimizers}] 
  \textit{Step 1: Compute the energy of disks.}\vspace{1mm}\\
  The type of computation involved in evaluating
  $E_{1,0}(\ball{0}{r})$ was carried out before by many authors
  \cite{mcconnell88,langer92,tsebers80}. In our case, we can write
	\begin{align}
	  \begin{split}
            |\log\delta| F_{1,\delta,l} (\ball{0}{1}) & = 2\pi
            |\log\delta| - \frac{1}{2} \int_{\Sph^1}\int_{\Sph^1}
            \nu(x)\cdot \nu(y)\Phi_{\delta,l}(|x-y|) \intd \Hd^1(y)
            \intd
            \Hd^1(x)\\
            & = 2\pi |\log\delta| - \pi \int_{-\pi}^{\pi} \cos(t)
            \Phi_{\delta,l}(\sqrt{2(1-\cos(t))}) \intd t.
	  \end{split}
	\end{align}
	The integral can be computed to leading order to satisfy
	\begin{align}
          F_{1,0,l}(\ball{0}{1}) = -2\pi \log 4 + \frac{2 \pi}{
          \sqrt{4 + l^2}} \left((2+l^2)K\left( \frac{4}{4+l^2}\right)
          - (4+l^2)E\left( \frac{4}{4+l^2}\right)  \right) , 
	\end{align}
	where
        $K(k) := \int_0^\frac{\pi}{2} (1 - k
        \sin^2(\theta))^{-\frac{1}{2}}\intd \theta$ and
        $E(k) := \int_0^\frac{\pi}{2} (1 - k
        \sin^2(\theta))^\frac{1}{2}\intd \theta$ are the complete
        elliptic integrals of the first and second kind
        \cite{abramowitz}.
	
	The rescalings of Lemma \ref{lem:rescaling} imply that for all
        $r>0$ we have
	\begin{align}
		\begin{split}
		  \quad F_{1,0,l}(\ball{0}{r}) & = -2\pi r \log 4 -
                  2\pi r\log r \\
                  & \quad + \frac{2 \pi r}{ \sqrt{4 +
                      \frac{l^2}{r^2}}}
                  \left(\left(2+\frac{l^2}{r^2}\right)K\left(
                      \frac{4}{4+\frac{l^2}{r^2}}\right) -
                    \left(4+\frac{l^2}{r^2}\right)E\left(
                      \frac{4}{4+\frac{l^2}{r^2}}\right) \right) ,
		\end{split}
	\end{align}
	which in turn gives
	\begin{align}
          f_{disk}(a)
          &
            :=\frac{F_{1,0,l}\left(\ball{0}{\frac{1}{a}}\right)}{\left|\ball{0}{\frac{1}{a}}  
            \right|} \notag\\ 
          & = 2a\Bigg(   -\log 4 + \log a  + \frac{1 }{
            \sqrt{4 + a^2l^2}}\Bigg(
            \left(2+a^2l^2\right)K\left(
            \frac{4}{4+a^2l^2}\right)\\ 
          & \phantom{2a\Bigg(   -\log(4) + \log(a)  +
            \frac{1 }{ \sqrt{4 +
            a^2l^2}}\Bigg(\big(2+a^2l^2))) } -
            \left(4+a^2l^2\right)E\left(
            \frac{4}{4+a^2l^2}\right)  \Bigg)
            \Bigg).\notag 
	\end{align}
	
        \textit{Step 2: The function $f_{disk}(a)$ is strictly convex
          for all $0<a<\infty$.}\vspace{1mm}\\ 
	Note that
        $f_{disk}(a) = \frac{1}{l}\left( g(a l) - 2al \log(l)\right)$
        with
	\begin{align}
	  \begin{split}
            g(\alpha) :=  2\alpha\Bigg(   -\log 4 + \log \alpha  +
            \frac{1 }{ \sqrt{4 + \alpha^2}} \Bigg(&
            \left(2+\alpha^2\right)K\left(
              \frac{4}{4+\alpha^2}\right)\\ 
            & \quad - \left(4+\alpha^2\right)E\left(
              \frac{4}{4+\alpha^2}\right) \Bigg) \Bigg),
	  \end{split}
	\end{align}
	so that it is sufficient to prove convexity of $g$.
	We get
	\begin{align}
          g''(\alpha) = \frac{2\left((4+\alpha^2)^\frac{3}{2} -
          2(4+7\alpha^2+\alpha^4) E\left( \frac{4}{4+\alpha^2}\right) 
          + 2\alpha^2 (5 + \alpha^2 ) K\left(
          \frac{4}{4+\alpha^2}\right)
          \right)}{\alpha (4+\alpha^2)^\frac{3}{2}}, 
	\end{align}
	of which we want to see that the numerator is non-negative.
        Solving the equation $t = \frac{4}{4+\alpha^2}$ with
        $t \in [0,1]$ for $\alpha>0$ gives
	\begin{align}
		\alpha(t) := 2\sqrt{\frac{1-t}{t}}
	\end{align}
	and rewriting the numerator in terms of $t$ gives
        $\frac{\alpha(t) (4+\alpha^2(t))^\frac{3}{2}g''(\alpha(t))}{2}
        = \frac{8 h_1(t)}{t^2}$ with
	\begin{align}
          h_1(t):= t^\frac{1}{2} + (-4 +t +2 t^2) E(t) - (-4 +3t +
          t^2) K(t). 
	\end{align}
	Since $0\leq t\leq 1$, we have
        $h_1(t) \geq h_2(t)$ with
	\begin{align}
		h_2(t) :=  t - (4 -t -2 t^2) E(t) +(4 -3t -t^2) K(t).
	\end{align}
	Therefore, it is enough to prove that $h_2(t) \geq 0$.
	
	We first prove that $h_2(t) > 0$ for $0.85 \leq t <1$.
	To this end we compute
	\begin{align}
          h_2'(t) = 1 + (2 + 5t)E(t) - \frac{1}{2}(4+5t)K(t) \leq 1 +
          7E(t) - \frac{1}{2}(4+5t)K(t) 
	\end{align}
	and note that the right-hand side is decreasing in $t$ since
        $E$ is decreasing and $K$ is increasing.  Therefore, to prove
        $h_2'(t) < 0$ for $0.85\leq t <1$ we only have to see that
	\begin{align}
          1 + 7E(0.85) - \frac{1}{2}(4+5\times 0.85 )K(0.85)  \approx
          -0.850922 <0, 
	\end{align}
	where the computation can be carried out to arbitrary precision.
	Since for $t$ close to $1$ we have $h_2(t) =
        O((1-t)\log(1-t))$ we get $h_2(0) = 0$, which proves the
        claim. 
	
	In a second step we prove $h_2(t) > 0$ for $0<t < 0.85$.  Let
        $E_1$ and $K_1$ be the first order Taylor polynomials of $E$
        and $K$ at the origin.  It is easily seen from the power
        series representation of the elliptic integrals
        \cite{abramowitz} that we have
	\begin{align}
		E(t) \leq E_1(t) \text{ and } K(t) \geq K_1(t)
	\end{align}
	for all $0\leq t <1$.
	We therefore have
	\begin{align}
          h_2(t) \geq  t - (4 -t -2 t^2) E_1(t) +(4 -3t -t^2) K_1(t) =
          t - \frac{3\pi}{8}t^3 > 0 
	\end{align}
	for all $0 < t < \sqrt{\frac{8}{3\pi}} \approx 0.92$.
	
	\vspace{1mm}
        \textit{Step 3: Compute the energy per mass of a long stripe.}\vspace{1mm}\\
        For $a,m>0$ let
        $S_{a,m}:= \left(-\frac{am}{2},\frac{am}{2}\right) \times
        \left(-\frac{1}{2a},\frac{1}{2a} \right)$.  Evidently we
        have $|S_{a,m}|=m$ for all $a>0$.  We write
	 	\begin{align}
	 	  \begin{split}
                    E_{1,\delta}(S_{a,m}) = |\log\delta| \left(2a m +
                      \frac{2}{a}\right) & -
                    \int_{-\frac{am}{2}}^{\frac{am}{2}}
                    \int_{-\frac{am}{2}}^{\frac{am}{2}} \left[
                      \Phi_\delta(|s-t|) -
                      \Phi_\delta\left(\sqrt{(s-t)^2 + \frac{1}{a^2}}
                      \right) \right] \intd s \intd t\\
                    & - \int_{-\frac{1}{2a}}^{\frac{1}{2a}}
                    \int_{-\frac{1}{2a}}^{\frac{1}{2a}} \left[
                      \Phi_\delta(|s-t|) -
                      \Phi_\delta\left(\sqrt{(s-t)^2 + a^2m^2} \right)
                    \right] \intd s \intd t\\
                    \phantom{E_{1,\delta}(S_{a,m})} = |\log\delta|
                    \left(2a m + \frac{2}{a}\right) & + I_1 + I_2 +
                    I_3+I_4,
	 	\end{split}
	\end{align}
	where we abbreviated the terms coming from the integrals
        with $I_1,\dots, I_4$ in order of their appearance.

        Next, we compute
	\begin{align}
          \label{eq:I1I3}
	  \begin{split}
            I_1& = - 2a m |\log\delta| - 2am \log(am) -2am +
            o(\delta),\\
            I_3 & = - \frac{2}{a} |\log\delta| - \frac{2}{a}
            \log\left(\frac{1}{a}\right) -\frac{2}{a} + o(\delta).
	  \end{split}
	\end{align}
	In the other two terms we can go to the limit $\delta \to
        0$ directly under the integral sign and obtain 
        \begin{align}
	  \begin{split}
            \lim_{\delta \to 0} I_2& =
            \int_{-\frac{am}{2}}^{\frac{am}{2}}
            \int_{-\frac{am}{2}}^{\frac{am}{2}} \frac{1}{\sqrt{(s-t)^2
                + \frac{1}{a^2}}} \intd s \intd t = \frac{2 - 2
              \sqrt{1+ a^4m^2}}{a} + 2 a m \arcsinh\left(a^2 m\right)
            ,\\
            \lim_{\delta \to 0} I_4 & =
            \int_{-\frac{1}{2a}}^{\frac{1}{2a}}
            \int_{-\frac{1}{2a}}^{\frac{1}{2a}}\frac{1}{\sqrt{(s-t)^2
                + a^2m^2} } \intd s \intd t = \frac{2 \left[a^2 m
                -\sqrt{1 + a^4 m^2}+\coth ^{-1}\left(\sqrt{1 + a^4
                    m^2}\right)\right]}{a}.
	  \end{split}
	\end{align}
        Combining this with \eqref{eq:I1I3} yields
	\begin{align}
	  \begin{split}
            E_{1,0}(S_{a,m}) & = - 2am \log(am) - \frac{2}{a}
            \log\left(\frac{1}{a}\right) - \frac{2}{a} \\
            & \quad -4am + 2 a m \arcsinh (a^2 m) + O(m^{-1})
	  \end{split}
	\end{align}
        for large values of $m$. As a consequence we have
	\begin{align}
          \lim_{m\to \infty} \frac{E_{1,0}(S_{a,m})}{m} = 2a
          \left( \log(2a) - 2 \right). 
	\end{align}
	
	The contribution of the modification for $l<\infty$ to
        $\frac{|\log\delta|F_{1,\delta,l}(S_{a,m})}{m}$ is given by
	\begin{align}
	 	  \begin{split}
                    \frac{|\log\delta|F_{1,\delta,l}(S_{a,m})}{m}
                    & = \frac{|\log\delta|E_{1,\delta}(S_{a,m})}{m} \\
                    & \quad +\frac{1}{m}
                    \int_{-\frac{am}{2}}^{\frac{am}{2}}
                    \int_{-\frac{am}{2}}^{\frac{am}{2}} \left(
                      \frac{1}{\sqrt{|s-t|^2+l^2}} -
                      \frac{1}{\sqrt{|s-t|^2 + \frac{1}{a^2} +l^2}}
                    \right) \intd s \intd t\\
                    & \quad
                    +\frac{1}{m}\int_{-\frac{1}{2a}}^{\frac{1}{2a}}
                    \int_{-\frac{1}{2a}}^{\frac{1}{2a}} \left(
                      \frac{1}{\sqrt{|s-t|^2+l^2}} -
                      \frac{1}{\sqrt{|s-t|^2 + a^2m^2 +l^2}} \right)
                    \intd s \intd t.
	 	\end{split}
	\end{align}
	The second integral clearly vanishes in the limit $m \to
        \infty$, while for the first we can write
	\begin{align}
          & \quad \lim_{m\to \infty} \frac{1}{m}
            \int_{-\frac{am}{2}}^{\frac{am}{2}}
            \int_{-\frac{am}{2}}^{\frac{am}{2}} \left( 
            \frac{1}{\sqrt{|s-t|^2+l^2}} - \frac{1}{\sqrt{|s-t|^2 +
            \frac{1}{a^2} +l^2}} \right) \intd s \intd t \\ 
          & = a \int_{-\infty}^\infty \left( \frac{1}{\sqrt{s^2+l^2}} -
            \frac{1}{\sqrt{s^2 + \frac{1}{a^2} +l^2}} \right) \intd s \\ 
          & = -2 a \log l + a \log\left(\frac{1}{a^2} + l^2 \right),
	\end{align}
	so that we obtain
	\begin{align}
          f_{stripe}(a):=\lim_{m\to \infty}
          \frac{F_{1,0,l}(S_{a,m})}{m} = 2 a \log \left( {2 a \over l}
          \right) - 4 a + a \log\left(\frac{1}{a^2} + 
          l^2 \right). 
	\end{align}
	
        \textit{Step 4: Conclusion.}\vspace{1mm}\\
	We can compute $f_{disk}(a) = 2 a \log(a) + O(a)$ for $a \to
        \infty$.  Combining this with the strict convexity established
        in Step 2 we get that it attains its unique minimum at
        $a(l) \in [0,\infty)$.  Expanding the functions $f_{disk}(a)$
        and $f'_{disk}(a)$ at $a=0$ gives
	\begin{align}
          f_{disk}(a) & = 2 a\left[
                        \log\left(\frac{2}{l}\right) -2 \right]  +
                        \frac{3 
                        l^2}{8} a^3 \log a^{-1}+ O(a^3) 
	\end{align}
	and
	\begin{align}
          f'_{disk}(a) = 2\left[ 
          \log\left(\frac{2}{l}\right) -2 \right] + \frac{9 l^2}{8} a^2
          \log a^{-1} + O(a^2),
	\end{align}
	where the term in the square bracket is negative if and
          only if $l > {2 \over e^2}$. By uniqueness of the
        minimizer, this implies that $a(l) \to 0$ as
        $l \to \frac{2}{e^2}$ from above.  In turn, expanding
        $f_{stripe}(a)$ gives
	\begin{align}
          f_{stripe}(a) = 2 a\left[
          \log\left(\frac{2}{l}\right) -2 \right]   +
          \frac{l^2}{3}a^3 + O(a^4), 
	\end{align}
	which implies $f_{stripe}(a(l)) < f_{disk}(a(l)) = \min
        f_{disk}$ for all $0<l- \frac{2}{e^2}<c$ with $c >0$
        small enough universal.
	
	Consequently, there exists a mass $M(l)>0$ such that for all
        $m > M(l)$ we have
	\begin{align}
          \left|\frac{F_{1,0,l}(S_{a(l),m})}{m} - f_{stripe}(a(l))
          \right| < \frac{1}{2} \big( \min f_{disk} - f_{stripe}(a(l))
          \big) 
	\end{align}
	and thus
	\begin{align}
		F_{1,0,l}(S_{a(l),m}) < m \min f_{disk} \leq
          \sum_{i=1}^N F_{1,0,l}(\ball{0}{r_i}) 
	\end{align}
	for any finite collection of disks $\ball{0}{r_i}$ satisfying
        $\sum_{i=1}^N |\ball{0}{r_i}| = m$.  Therefore, the
        generalized minimizer cannot consist exclusively of disks,
        which gives the desired statement.
      \end{proof}

        \section{The supercritical case $\lambda > 1$}
\label{sec:supercr-case-lambda}

We finally briefly prove Proposition \ref{prop:supercritical} in the
supercritical case $\lambda >1$.

\begin{proof}[Proof of Proposition \ref{prop:supercritical}]
  Recall that by assumption of the proposition we have, up to a
  negligible set, $\ball{0}{r} \subset \Omega \subset \R^2$ for a set
  of finite perimeter $\Omega$ and some $r>0$, which we will later
  choose to satisfy some relation to the small number $\delta>0$. Let
  $\tilde r:=\frac{1}{2}r$. We first deal with the case
  $0<\tilde r<1$.

\vspace{1mm}
  \textit{Step 1: Compute the energy of $\widetilde \Omega := \Omega
    \setminus \ball{0}{\tilde r}$.}\vspace{1mm}\\
  Elementary combinatorics imply that
  \begin{align}
    \widetilde \Omega \times \stcomp{\widetilde \Omega } = \Big(\Omega
    \times \stcomp{\Omega} \setminus  \ball{0}{\tilde r} \times
    \stcomp{\Omega}\Big) \cup \Big(B^{\mathsf{c}}_{\tilde r}(0)\times
    \ball{0}{\tilde r} \setminus \stcomp{\Omega} \times
    \ball{0}{\tilde r}\Big).  
  \end{align}
  In terms of the nonlocal contribution this reads
	\begin{align}
	  \begin{split}
            \int_{\widetilde \Omega}\int_{\stcomp{\widetilde \Omega }}
            \frac{g_\delta(|x-y|)}{|x-y|^3} \intd y \intd x &=
            \int_{\Omega}\int_{\stcomp{\Omega }}
            \frac{g_\delta(|x-y|)}{|x-y|^3} \intd y \intd x +
            \int_{\ball{0}{\tilde r}}\int_{B^{\mathsf{c}}_{\tilde
                r}(0)}
            \frac{g_\delta(|x-y|)}{|x-y|^3} \intd y \intd x\\
            & \quad - 2 \int_{\ball{0}{\tilde r}}\int_{\stcomp{ \Omega
              }} \frac{g_\delta(|x-y|)}{|x-y|^3} \intd y \intd x.
	  \end{split}
	\end{align}
	Therefore, we get
	\begin{align}
          E_{\lambda,\delta}\left(\widetilde \Omega\right) =
          E_{\lambda,\delta}(\Omega) + E_{\lambda,\delta}(\ball{0}{\tilde r})
          + \frac{\lambda}{|\log\delta|}
          \int_{\ball{0}{\tilde r}}\int_{\stcomp{ \Omega }}
          \frac{g_\delta(|x-y|)}{|x-y|^3} \intd y \intd x. 
	\end{align}

        \textit{Step 2: To satisfy the mass constraint,
          include the
          ball $\ball{0}{\tilde r}$ at infinity and compare the energies.}\vspace{1mm}\\
	With the above computation, we obtain
	\begin{align}
          E_{\lambda,\delta}\left(\widetilde \Omega\right) +
          E_{\lambda,\delta}(\ball{0}{\tilde r}) = E_{\lambda,\delta}(\Omega)
          + 2 \left( E_{\lambda,\delta}(\ball{0}{\tilde r}) +
          \frac{\lambda}{2|\log\delta|}
          \int_{\ball{0}{\tilde r}}\int_{\stcomp{ \Omega }}
          \frac{g_\delta(|x-y|)}{|x-y|^3} \intd y \intd x\right) 
	\end{align}
	and therefore we only have to check that the bracket is negative.
	
	Using Lemma \ref{lem:rescaling2} and the assumption $\tilde
        r<1$, the energy of the ball can be estimated as
	\begin{align}
          E_{\lambda,\delta}(\ball{0}{\tilde r})  \leq \tilde r \left(
          E_{\lambda,\delta}(\ball{0}{1}) - \frac{2\pi\lambda \log \tilde r
          }{|\log\delta|} \right). 
	\end{align}
	The same calculation as in the proof of 
        Proposition \ref{prop:subcritical_gamma_convergence}
        furthermore gives
	\begin{align}
          E_{\lambda,\delta}(\ball{0}{1}) \leq 2\pi (1-\lambda) +
          \frac{C\lambda}{|\log\delta|}. 
	\end{align}
	By observing that
        $|\stcomp{\Omega}\cap (\ball{x}{2\tilde r}\setminus
        \ball{x}{\tilde r})|=0$ 
        we get
	\begin{align}
          \frac{\lambda}{2|\log\delta|}
          \int_{\ball{0}{\tilde r}}\int_{\stcomp{ \Omega }}
          \frac{g_\delta(|x-y|)}{|x-y|^3} \intd y \intd x
          & \leq
            \frac{\lambda}{2|\log\delta|}
            \int_{\ball{0}{\tilde r}}\int_{\stcomp{B}_{2\tilde r}(0)}
            \frac{g_\delta(|x-y|)}{|x-y|^3}
            \intd y
            \intd x
            \notag\\ 
          & \leq \frac{ \lambda \pi \tilde r^2}{2|\log\delta|}
            \int_{B^{\mathsf{c}}_{\tilde r}(0)} \frac{g_\delta(|z|)}{|z|^3}
            \intd z\\ 
          & = \frac{ \lambda \pi^2 \tilde r}{|\log\delta|}, \notag
	\end{align}
	where in the last step we required $\tilde r \geq \delta$.
	
	Combining all of the above, we see that 
	\begin{align}
		E_{\lambda,\delta}(\ball{0}{\tilde r}) +
          \frac{\lambda}{2|\log\delta|}
          \int_{\ball{0}{\tilde r}}\int_{\stcomp{ \Omega }}
          \frac{g_\delta(|x-y|)}{|x-y|^3} \intd y \intd x \leq - 2\pi
          \lambda \tilde r \left( \frac{\lambda-1}{\lambda} + \frac{
          \log \tilde r - C}{|\log\delta|}\right)
	\end{align}
	for some universal $C > 0$, and the right-hand side is
        negative if and only if
        $\tilde r > e^C \delta^\frac{\lambda-1}{\lambda}$.  Therefore,
        the statement holds for
        $\tilde r > \max \{e^C \delta^\frac{\lambda-1}{\lambda},
        \delta \} = e^C \delta^\frac{\lambda-1}{\lambda}$ if
        $\delta>0$ is sufficiently small due to
        $\frac{\lambda-1}{\lambda} <1$.
	
	Finally, the case $\tilde r \geq 1$ can be treated by using the
        above argument for some $\hat r>0$ such that
        $e^C \delta^\frac{\lambda-1}{\lambda}< \hat r<1$ for
        $\delta$ sufficiently small.
\end{proof}

\paragraph{Acknowledgments.}

      This work was supported by NSF via grant DMS-1614948. The
      authors wish to acknowledge valuable discussions with A. Bernoff,
      V. Julin and M. Novaga.

\bibliographystyle{plain}
\bibliography{../../nonlin,../../mura,../../stat,../../bio}

\end{document}